\documentclass[12pt]{article}
\usepackage[utf8]{inputenc}
\usepackage[lite,abbrev,msc-links,numeric]{amsrefs}
\usepackage[mathscr]{eucal}
\usepackage{amsfonts}
\usepackage{amsmath}
\usepackage{amssymb}
\usepackage{amscd}
\usepackage{color}
\usepackage{amsthm}
\usepackage{bbm}
\usepackage{graphicx}
\usepackage[hmargin=3cm,vmargin=3cm]{geometry}
\usepackage{parskip}
\usepackage[all,cmtip]{xy}
\usepackage{enumitem}
\usepackage{IEEEtrantools}

\newtheorem{thm}{Theorem}

\newtheorem{prop}{Proposition}[section]
\newtheorem{lma}[prop]{Lemma}

\newtheorem{cor}[prop]{Corollary}

\theoremstyle{definition}

\newtheorem{df}[prop]{Definition} 

\theoremstyle{remark}

\newtheorem{rmk}[prop]{Remark} 

\newcommand{\R}{{\mathbb{R}}}
\newcommand{\Z}{{\mathbb{Z}}}
\newcommand{\F}{{\mathbb{F}}}

\newcommand{\C}{{\mathbb{C}}}

\newcommand{\D}{{\mathbb{D}}}
\newcommand{\K}{{\mathbb{K}}}
\newcommand{\bK}{{\mathbb{K}}}

\newcommand{\bs}{\bigskip}

\newcommand{\del}{\partial}

\newcommand{\sm}[1]{C^\infty(#1)}

\newcommand{\bb}[1]{{\mathbb #1}}
\newcommand{\ul}[1]{{\underline  #1}}

\newcommand\sign{\operatorname{sign}}

\newcommand{\ol}[1]{\overline{#1}}

\newcommand{\sS}{A}

\newcommand{\til}[1]{\widetilde{#1}}
\newcommand{\wh}[1]{\widehat{#1}}

\newcommand{\zp}{{\mathbb{Z}/p\mathbb{Z}}}
\newcommand{\m}{m}
\renewcommand{\l}{k}
\newcommand{\re}{\mrm{Re}}

\newcommand{\ip}{{(p)}}
\newcommand{\rp}{{R_p}}
\newcommand{\hrp}{{\widehat{R}_p}}

\def\mrm#1{{\mathrm{#1}}}
\def\bb#1{{\mathbb{#1}}}
\def\cl#1{{\mathcal{#1}}}
\def\ul#1{{\underline{#1}}}

\newcommand{\om}{\omega}
\newcommand{\al}{\alpha}

\newcommand{\Om}{\Omega}

\newcommand{\eps}{\epsilon}

\newcommand{\cA}{\mathcal{A}}
\newcommand{\cB}{\mathcal{B}}
\newcommand{\cC}{\mathcal{C}}

\newcommand{\cF}{\mathcal{F}}

\newcommand{\cH}{\mathcal{H}}

\newcommand{\cJ}{\mathcal{J}}
\newcommand{\cK}{\mathcal{K}}

\newcommand{\cP}{\mathcal{P}}
\newcommand{\sP}{\mathscr{P}}
\newcommand{\cZ}{\mathcal{Z}}
\newcommand{\cR}{\mathcal{R}}
\newcommand{\cS}{\mathcal{S}}
\newcommand{\cM}{\mathcal{M}}

\newcommand{\sQ}{\mathscr{Q}}

\newcommand{\id}{{\mathrm{id}}}

\newcommand{\coker}{\mathrm{Coker}}
\newcommand{\fix}{\mathrm{Fix}}

\DeclareMathOperator{\thetam}{\theta_M}

\DeclareMathOperator{\Det}{Det}
\DeclareMathOperator{\Ham}{\mathrm{Ham}}

\DeclareMathOperator{\Hom}{\mathrm{Hom}}
\DeclareMathOperator{\Symp}{\mathrm{Symp}}
\DeclareMathOperator{\Ker}{\mathrm{Ker}}

\DeclareMathOperator{\spec}{\mathrm{Spec}}
\DeclareMathOperator{\loc}{\mathrm{loc}}
\DeclareMathOperator{\ima}{\mathrm{Im}}

\def\H2{H^{(2)}}

\DeclareFontFamily{U}{mathx}{\hyphenchar\font45}
\DeclareFontShape{U}{mathx}{m}{n}{
      <5> <6> <7> <8> <9> <10>
      <10.95> <12> <14.4> <17.28> <20.74> <24.88>
      mathx10
      }{}
\DeclareSymbolFont{mathx}{U}{mathx}{m}{n}
\DeclareFontSubstitution{U}{mathx}{m}{n}
\DeclareMathAccent{\widecheck}{0}{mathx}{"71}
\DeclareMathAccent{\wideparen}{0}{mathx}{"75}

\title{The $\zp$-equivariant product-isomorphism \\ in fixed point Floer cohomology}
\author{Egor Shelukhin, Jingyu Zhao}
\date{}

\begin{document}

\maketitle

\begin{abstract}
    Let $p \geq 2$ be a prime, and $\bb F_p$ be the field with $p$ elements. We construct an isomorphism between the Floer cohomology of an exact or Hamiltonian symplectomorphism $\phi,$ with $\bb F_p$ coefficients, and the $\zp$-equivariant Tate Floer cohomology of its $p$-th power $\phi^p.$ This extends a result of Seidel for $p=2.$ The construction involves a Kaledin-type quasi-Frobenius map, as well as a $\zp$-equivariant pants product: an equivariant operation with $p$ inputs and $1$ output. Our method of proof involves a spectral sequence for the action filtration, and introduces a new key component: a local $\zp$-equivariant coproduct providing an inverse on the $E^2$-page. This strategy has the advantage of accurately describing the effect of the isomorphism on filtration levels. We describe applications to the symplectic mapping class group, as well as develop Smith theory for the persistence module of a Hamiltonian diffeomorphism $\phi$ on symplectically aspherical symplectic manifolds. We illustrate the latter by giving a new proof of the celebrated no-torsion theorem of Polterovich, and by relating the growth rate of the number of periodic points of the $p^k$-th iteration of $\phi$ and its distance to the identity. Along the way, we prove a sharpening of the classical Smith inequality for actions of $\zp.$
\end{abstract}

\tableofcontents

\section{Introduction}

Equivariant cohomology with respect to the action of the cyclic group of order $p,$ and the resulting Smith-type inequalities (see \cite{Hsiang-Transformation, Borel-Transformation, Bredon-Transformation} for the classical theory) have been obtained and used to great effect in various Floer theories in recent years, see for example \cite{Hend,Hendricksetal,Seidel,SeidelSmith,ManolescuLidman,TreumannLipshitz}. In the case of symplectic fixed point Floer cohomology $HF^*(\phi)$ of a symplectic automorphism $\phi$ of a Liouville manifold, and the $\Z/2\Z$-equivariant cohomology of its second iterate $\phi^2,$ Smith theory was most recently studied by Seidel \cite{Seidel}, after previous work of Hendricks \cite{Hend}. In particular, he proves that under suitable assumptions the following analogue of the Smith inequality holds true for fixed point Floer cohomology:
\begin{equation} \label{eqn:Smith_ineq}
\dim_{\F_2} HF^*(\phi) \leq \dim_{\F_2} HF^*(\phi^2)^{\Z/2\Z} \leq \dim_{\F_2} HF^*(\phi^2).
\end{equation}
The proof uses a remarkable cohomological operation coming from the $\Z/2\Z$ symmetry of the pair of pants with boundary conditions given by $\phi$ on the two input cylinders, and by $\phi^2$ on its output cylinder. 

In this paper, we extend the work of Seidel to all primes $p \geq 2,$  to fixed point Floer cohomology in an generic action window $I = (a,b),$ $a,b \in \R \cup \{\pm \infty \},$ and to local Floer cohomology. In the last case, a Smith-type inequality was obtained by more elementary methods in \cite{CineliGinzburg} during the preparation of this paper. We note that the case of action windows and that of local Floer cohomology are of interest already for Hamiltonian diffeomorphisms of symplectically aspherical symplectic manifolds. 

We record sample applications, proven in Section \ref{sec: applications}. Let $\phi$ be a Hamiltonian diffeomorphism of a symplectically aspherical symplectic manifold, and $\phi^p$ its $p$-th iterate. Assume that ${p}\cdot{a}$ and ${p}\cdot{b}$ are not critical values of the Hamiltonian action functional corresponding to $\phi^p$ on the free homotopy class of contractible loops. Then the Floer cohomology of $\phi$ in the interval $I = (a,b)$ and that of $\phi^p$ in the interval $p\cdot{I} = (p\cdot{a},p\cdot{b})$ are related by the following Smith-type inequality: 
\begin{equation}\label{eq: Smith in interval} \dim_{\F_p} HF^*(\phi)^{I} \leq \dim_{\F_p} (HF^*(\phi^p)^{p\cdot I})^{\zp} \leq \dim_{\F_p} HF^*(\phi^p)^{p\cdot I}.\end{equation} 
We remark that these cohomology groups are defined by perturbing $\phi$ by a sufficiently $C^2$-small Hamiltonian diffeomorphism to $\phi_1,$ and using the fact that the endpoints of the interval are not in the spectrum. In this case, we can choose a perturbation $\phi_1$ so that $\phi_1^p$ is a sufficiently $C^2$-small Hamiltonian perturbation of $\phi^p.$

There are two essentially immediate consequences of the above Smith-type inequalities \eqref{eqn:Smith_ineq} and \eqref{eq: Smith in interval}. Similarly to the results obtained by Hendricks \cite{Hend} and Seidel \cite{Seidel} in the case of $p=2$, one first obtains the following corollary, that is proven as Corollary \ref{app:MCG} in Section \ref{sec: applications}, for the $p$-th iterates of $\phi$ in the symplectic mapping class group.\\

\begin{cor} Given an exact symplectic manifold $W$ which is cylindrical at infinity, and a compactly supported exact symplectomorphism $\phi,$  if $\dim_{\F_p} HF^*(\phi) > \dim_{\F_p} H^*(W),$ then $[\phi^{p^k}] \neq 1$ in the symplectic mapping class group of $W$ for all $k\geq 0.$ 
\end{cor}

Furthermore, with the help of the action-filtered version of the Smith-type inequality \eqref{eq: Smith in interval}, we also provide a new proof of a well-known theorem of Polterovich \cite{Polterovich-groups}, stating that the group of Hamiltonian diffeomorphisms of a closed symplectically aspherical symplectic manifold contains no non-trivial torsion elements. It is stated below as Theorem \ref{thm: Polterovich} together with our new proof. \\

\begin{cor}[Polterovich \cite{Polterovich-groups}]
Let $\phi \in \Ham(M,\om)$ be a Hamiltonian diffeomorphism of a symplectically aspherical symplectic manifold, such that $\phi^k = 1$ for some $k \in \Z_{>1}.$ Then $\phi = \id.$
\end{cor}

Furthermore, we apply \eqref{eq: Smith in interval} together with additional combinatorial arguments, as well as arguments of \cite{SZ-92}, to prove the following result regarding Hamiltonian diffeomorphisms $\phi$ that are not torsion, but whose iterations $\phi^{p^k}$ approach the identity in various natural distances, such as the spectral distance $\gamma(\phi)$ \cite{SchwarzAspherical}, the $C^0$ distance, or the Hofer distance. It is proven as Theorem \ref{thm: growth rate}  in Section \ref{sec: applications}. We recall that a fixed point $x$ of a Hamiltonian diffeomorphism $\phi \in \Ham(M,\om)$ is called {\em contractible} if the loop $\{\phi^t x\}$ for any Hamiltonian isotopy $\{ \phi^t \}$ with $\phi^1 = \phi$ is contractible in $M.$ It is a standard consequence of Floer theory that this property does not depend on the choice of such an isotopy. 

\bs

\begin{cor}\label{cor: gamma small}
Let $\phi \in \Ham(M,\om)$ be a Hamiltonian diffeomorphism of a closed symplectically aspherical symplectic manifold, such that for all $k \geq 0,$ $\phi^{p^k}$ is non-degenerate. Then setting $N(\phi^{p^k})$ for the number of contractible fixed points of $\phi^{p^k}$ we have \[ \liminf_{k\to \infty} N(\phi^{p^k}) \cdot \gamma(\phi^{p^k}) / p^k > 0. \] 
\end{cor}

We note that Corollary \ref{cor: gamma small} implies that if $\liminf_{k \to \infty} \gamma(\phi^{p^k}) = 0,$ then $N(\phi^{p^k})$ grows super-linearly in $p^k.$ The same consequence holds if $\liminf_{k \to \infty} d_{C^0}(\phi^{p^k},\id) = 0,$ or $\liminf_{k \to \infty} d_{\mrm{Hofer}}(\phi^{p^k},\id) = 0.$ It is a well-known conjecture that Hamiltonian diffeomorphisms for which these limits vanish should not exist. Our result is a new step in this direction. 

\bs

To prove our main theorem, we use a cohomological operation coming from a branched cover of a cylinder that has $p$ inputs and $1$ output, and its $\zp$-symmetry, as in \cite{Seidel} for $p=2.$ However, showing that this $\zp$-equivariant product map is an isomorphism on the associated Tate cohomology groups requires substantially more complicated tools. Indeed, for $p=2$ the local contributions can be deduced from a few specific examples, including the period-doubling bifurcation, as discussed in \cite[Section 6]{Seidel}. However, for $p>2$ there is a shortage of such examples, and Seidel has remarked that a more refined approach is necessary. We proceed by providing a local inverse map for the product in terms of an equivariant coproduct operation with $1$ input and $p$ outputs, inspired by the approach briefly outlined for $p=2$ by Seidel in \cite[Remark 6.10]{Seidel}. It is curious to note that the classical Wilson theorem from number theory ultimately plays an important role in the calculation leading to local invertibility. We emphasize that besides the generalization to $p>2,$ our approach differs from the one suggested by Seidel in that we discuss the coproduct in local Floer cohomology, instead of defining the coproduct globally in the aspherical setting. While this situation is slightly more analytically difficult, since there is in general no inverse pair of PSS isomorphisms in local Floer cohomology, proceeding this way simplifies a few topological arguments and has the advantage of applying to Floer cohomology in action windows, which is interesting for Hamiltonian dynamics. It is also more flexible for extensions to the non-aspherical case. To implement this approach, we prove a general crossing energy result to define and discuss local equivariant Floer cohomology. This also allows us to remove certain ``general position" assumptions present in Seidel's paper.  Finally, the local-to-global argument proceeds by the use of the action-filtration spectral sequence.

Besides the approach of local coproduct and product-coproduct operations, that is found in Section \ref{sec:locally invertible}, other technical innovations in this paper include the following. First, in Section \ref{subsec: alg spec sequence} we present new algebraic arguments that allow us to improve on the classical Smith inequality. Second, in Section \ref{subsec: estimates} we give an elementary proof of Proposition \ref{prop: monotonicity}, which is a very general crossing energy argument that clarifies the phenomenon (statements of this kind are usually proved using considerably more advanced techniques such as the target-local Gromov compactness of Fish \cite{Fish-compactness}). It is this result that allows us to define local equivariant Floer cohomology, and to reduce our consideration to individual fixed points, removing extra ``general position" assumptions made in \cite{Seidel}. Finally, Appendix \ref{app:signs and or} contains a discussion of signs and orientations necessary for working with coefficients in $\F_p.$

To give a taste of the algebra involved in the proof, we record our main technical result, Theorem \ref{thm: main}, from which inequality \eqref{eq: Smith in interval} follows purely algebraically. In fact a stronger inequality follows (see Remark \ref{rmk: sharp Smith}). We refer to Sections \ref{sec: group_coho}, \ref{sec:Floer_coho}, \ref{sec:equiv-Floer-coho}, and \ref{sec:prod-coprod} below, for detailed definitions of all the notions involved in the formulation of this theorem. At the moment we just remark that Tate cohomology, on the level of $\zp$ vector spaces, is a cohomology theory that vanishes on free $\zp$ vector spaces. In this paper we use certain more complicated versions of this construction, wherein it corresponds roughly to discarding contributions from simple $p$-periodic points of $\phi$, i.e. the fixed points of $\phi^p$ that are not fixed points of $\phi$.


\bs

\begin{thm}\label{thm: main}
Let $\phi$ be an exact symplectic automorphism of a Liouville domain, or a Hamiltonian diffeomorphism of a closed symplectically aspherical symplectic manifold. For a generic interval $I = (a,b)$ with $a<b,$ $a,b \in \R \cup \{\pm \infty\},$ and a prime $p \geq 2,$ working with coefficients in $\F_p,$ there exists an algebraically defined (quasi-Frobenius) isomorphism of Tate cohomology groups 
\[F: \wh{H}^*(\zp, HF^*(\phi)^I)^{(1)} \to \wh{H}^*(\zp, (CF^*(\phi)^{\otimes p}))^{p\cdot I}\] 
and a Floer-theoretically defined (product) map between $\zp$ group cohomology and the $\zp$-equivariant Floer cohomology group
\[\cP: H^*(\zp, CF^*(\phi)^{\otimes p})^{p\cdot I} \to HF^*_{\zp}(\phi^p)^{p\cdot I}\] 
that becomes an isomorphism of Tate cohomology groups
 \[\cP: \wh{H}^*(\zp, (CF^*(\phi)^{\otimes p}))^{p\cdot I} \to \wh{HF}^*_{\zp}(\phi^p)^{p\cdot I}\] after tensoring with $\F_p((u))$ over $\F_p[[u]].$ Here 
 \[\wh{H}^*(\zp, HF^*(\phi)^I)^{(1)} \cong HF^*(\phi)^I \otimes_{\F_p} \F_p((u))\left< \theta \right>,\] where $u$ and $\theta$ are formal variables of degree $2$ and $1$ respectively, $\F_p((u)) = \F_p[u^{-1},u]]$ denotes the formal Laurent power series in $u,$ $\left< \theta \right>$ denotes an exterior algebra on $\theta,$ and the superscript $^{(1)}$ denotes the Tate twist.
\end{thm}

\bs

We observe that this result has the following immediate corollary. 
\medskip

\begin{cor}\label{cor: PF isomorphism}
The composition \[\cP \circ F: HF^*(\phi)^I \otimes_{\F_p} \F_p((u))\left< \theta \right> \to \wh{HF}^*_{\zp}(\phi^p)^{p\cdot I}\] is an isomorphism of $\F_p((u))$ vector spaces.
\end{cor} 

Studying the dimensions of the vector spaces from Corollary \ref{cor: PF isomorphism}, we deduce \eqref{eq: Smith in interval}.

Since the case of $p=2$ amounts essentially to a repetition, with perhaps a very slight extension, of results of \cite{Seidel}, for brevity we omit the discussion of this case and assume throughout that $p>2.$ However, we note that our result for $p=2,$ due to our use of slightly more complicated algebraic tools is a bit stronger than the one in \cite{Seidel}, in view of the inequality \eqref{eq: summary bound}.

We add that a generalization of a part of the results of this paper to the monotone case, as well as further applications to dynamics, can be found in \cite{S-HZ}.

\section*{Acknowledgements}
We thank Mohammed Abouzaid, Viktor Ginzburg, Basak G\"{u}rel, Kristen Hendricks, Paul Seidel, Leonid Polterovich, Netanel Rubin-Blaier, Nicholas Wilkins, Dingyu Yang, and Jun Zhang for very useful discussions. In particular, Theorem \ref{thm: growth rate} was inspired by numerous discussions with Leonid Polterovich. Part of this project was carried out while both authors were postdoctoral members at the Institute for Advanced Study, supported by NSF grant No. DMS-1128155. We thank the IAS for a great research atmosphere. At the University of Montr\'{e}al E.S. was supported by an NSERC Discovery Grant and by the Fonds de recherche du Qu\'{e}bec - Nature et technologies. At Harvard CMSA, J.Z. was supported by  the Simons Foundation grant $\#385573$, Simons Collaboration on Homological Mirror Symmetry. Finally, we thank the referees for numerous useful comments on the exposition.

\section{Group cohomology and Tate cohomology}\label{sec: group_coho}

\indent In this section, we recall the preliminary definitions of group cohomology and Tate cohomology for vector spaces and cochain complexes endowed with $G=\zp$ actions. \\
\indent For a fixed prime $p,$ we let $\K$ be a field of characteristic $p$ and consider $\Z/2\Z$- or $\Z$-graded vector spaces or cochain complexes $V^*$ defined over $\K.$ In this paper, we mainly consider the case when $\K = \bb F_p.$ 

We add that in this section, and in this paper in general, we make extensive use of the notion of a spectral sequence. Since it is quite standard nowadays, we refer to \cite[Chapter 5]{Weibel} for all the relevant preliminary material, sometimes mentioning specific relevant results for the convenience of the reader.


An action of the cyclic group $G=\zp$ on a (graded) vector space $V$ is given by a (degree-preserving) linear transformation $\sigma \colon V \rightarrow V$ such that $\sigma^p =id.$ Alternatively, the $G$-action on $V$ is equivalent to a (graded) $\K[G]$-module structure on $V.$ Given such a $G$ -action, the $G$-invariants and the $G$-coinvariants are defined as follows
\begin{eqnarray}
&& V^G= \Ker(1-\sigma) =  \{x \in V \mathbin{|} g\cdot x=x \text{ for all } g \in G \} \\
&& V_G=V/\ima(1-\sigma),
\end{eqnarray}
where $\ima(1-\sigma):=\langle 1-\sigma \rangle V$ is the $\K[G]$-submodule generated by $g\cdot x$ for $g=1-\sigma$ and $x \in V.$ The group homology and cohomology can be then defined as the (derived) $G$-coinvariants and $G$-invariants of the $G$-action
\begin{eqnarray}
&& H_i(G; V):=Tor_i^{\K[G]}(\K, V),\\
&& H^i(G; V):=Ext^i_{\K[G]}(\K, V).
\end{eqnarray}
To compute these (co)homology groups explicitly for cyclic groups $G:=\Z/p\Z,$ one takes the free resolution of the ground field $\K$ as a $\K[G]$-module given by
\begin{equation}\label{eqn:res_P}
0 \leftarrow \K \xleftarrow{\epsilon} \K[G] \xleftarrow{1 - \sigma} \K[G] \xleftarrow{N} \K[G] \xleftarrow{1-\sigma} \K[G] \xleftarrow{N} \cdots
\end{equation}
where $\epsilon(\sum_i a_ig_i)=\sum_i a_i$ is the augmentation map, and $N=id+\sigma+\sigma^2+\cdots +\sigma^{p-1}$ is the norm map of the $G$-action. Let $P_{\bullet}:=(\K[G] \xleftarrow{1 - \sigma} \K[G] \xleftarrow{N} \K[G] \xleftarrow{1-\sigma} \K[G] \xleftarrow{N} \cdots).$ The group homology and cohomology can be computed explicitly as the homology of
\begin{equation}
P_{\bullet}\otimes_{\K[G]}V \text{ and } \mrm{Hom}_{\K[G]}(P_{\bullet}, V) \nonumber.
\end{equation}
\indent One the other hand, the Tate invariants and coinvariants are the kernels and cokernels of the norm map $N$ acting on ordinary invariants and coinvariants,
\begin{equation}
\wh{H}_0(G;V)=\Ker(N)/\ima(1-\sigma), \ \ 
\wh{H}^0(G;V)=\Ker(1-\sigma)/\ima(N).
\end{equation}
Extending the previous free resolution \eqref{eqn:res_P} two-periodically, one obtains the Tate ``resolution"
\begin{equation} \label{eqn:res_Q}
Q_{\bullet}:= ( \cdots \xleftarrow{1-\sigma} \underbrace{\K[G]}_{-1} \xleftarrow{N} \underbrace{\K[G]}_{0} \xleftarrow{1-\sigma} \underbrace{\K[G]}_{1} \xleftarrow{N}\underbrace{\K[G]}_{2} \xleftarrow{1-\sigma} \cdots )
\end{equation}


The Tate homology and cohomology are defined to be the (derived) Tate invariants and coinvariants, 
\begin{eqnarray}\label{eq: Tate basic 1}
&& \wh{H}_i(G; V):= H_i(Q_{\bullet} \otimes_{\K[G]}V),\\
\label{eq: Tate basic 2} && \wh{H}^i(G; V):= H_i({\Hom}_{\K[G]}(Q_{\bullet}, V)) \cong H_i( (Q^{\vee})^{\bullet}\otimes_{\K[G]}V),
\end{eqnarray}
where $(Q^{\vee})^{\bullet}:=\Hom_{\K[G]}(Q_{\bullet}, \K[G])$ is the dual resolution.\\


For a cochain complex $(V, d_V)$ of $\bK[G]$-modules, so that the degree of $d_V$ is $1,$ the equivariant (co)homology can be defined as the homology of the equivariant (co)chain complex
\begin{eqnarray}\label{eq: equivariant complexes 1}
{C}_k(\zp; V)=\bigoplus_{i-j=k} P_i \otimes V^{-j}, \ \ {d}=d_{\dagger}+(-1)^id_V,\\ \label{eq: equivariant complexes 2}
{C}^k(\zp; V)=\bigoplus_{i+j=k} (P^{\vee})^i \otimes V^j, \ \ {d}=d_{\dagger}+(-1)^id_V,
\end{eqnarray}
for $(P^{\vee})^{\bullet}:=\Hom_{\K[G]}(P_{\bullet}, \K[G])$ and the Tate (co)homology can be defined as the homology of the Tate (co)chain complex
\begin{eqnarray}\label{eq: Tate complexes 1}
\wh{C}_k(\zp; V)=\bigoplus_{i-j=k} Q_i \otimes V^{-j}, \ \ \wh{d}=d_{\dagger}+(-1)^id_V,\\ \label{eq: Tate complexes 2}
\wh{C}^k(\zp; V)=\bigoplus_{i+j=k} (Q^{\vee})^i \otimes V^j, \ \ \wh{d}=d_{\dagger}+(-1)^id_V,
\end{eqnarray}
where on $\wh{C}_k(\zp; V),$ $d_{\dagger}=d_{0}:=N=1+\sigma+\cdots +\sigma^{p-1}$ if $i$ is even and $d_{\dagger}=d_{1}:=1-\sigma$ if $i$ is odd, whereas on $\wh{C}^k(\zp; V),$ $d_{\dagger}=(d_0)^{\vee}=1-\sigma$ for $i$ even and $d_{\dagger}=(d_1)^{\vee}=N$ for $i$ odd. Note that in \eqref{eq: Tate complexes 1} we look at the {\em chain} complex $V^{-\ast},$ whose differential is of degree $(-1),$ naturally obtained from the cochain complex $V^{\ast}.$ This distinction is important in the $\Z$-graded setting. We remark that the complexes appearing in \eqref{eq: Tate complexes 1} and \eqref{eq: Tate complexes 2} are tensor products of the complexes $Q_{\bullet}$ and $(Q^{\vee})^{\bullet}$ respectively with $V,$ and are hence completely analogous to \eqref{eq: Tate basic 1} and \eqref{eq: Tate basic 2} but in the monoidal category of chain complexes. 

One can therefore rewrite the Tate homology complex with coefficients in a cochain complex $(V, d_V)$ as
\begin{equation}
\big(V \otimes_{\K}\K((u))\langle\theta\rangle, \wh{d} \big),
\end{equation}
where the differential $\wh{d}$ with respect to the splitting $V \otimes_{\K}\K((u)) \otimes_{\K} 1\; \oplus\; V \otimes_{\K}\K((u)) \otimes_{\K} \theta$ is written as follows \begin{equation}
\wh{d}(x\otimes 1)= d_V(x) \otimes 1 + u^{-1} N\, x \otimes \theta,\;\; 
\wh{d}(x\otimes \theta)= -d_V(x) \otimes \theta + (1-\sigma)\, x \otimes 1.
\end{equation}


Here, $u, \theta$ are formal variables of degrees $2$ and $1$ respectively, $\langle \theta \rangle$ is the exterior algebra generated over $\bK$ by $\theta,$ and $\bK((u))$ is the field of Laurent series in $u.$ We observe that strictly speaking rewriting the complex would feature the ring $\bK[u^{-1},u]$ of Laurent polynomials in $u,$ but for our purposes it is convenient to consider its $u$-adic completion, which is $\bK((u)).$

Similarly, the Tate cohomology complex with coefficients in $(V, d_V)$ is defined as the complex $V \otimes_{\K}\K((u))\langle\theta \rangle$ with the differential 
\begin{equation}
\wh{d}(x\otimes 1)= d_V(x) \otimes 1 + (1-\sigma)\, x \otimes \theta,\;\;
\wh{d}(x\otimes \theta)= - d_V(x) \otimes \theta + u N\, x \otimes 1.
\end{equation}


The Tate homology $\wh{H}_*(G; V)$ and cohomology $\wh{H}^*(G; V)$ are the (co)homology of the above complexes (we remind the reader that we have considered the $u$-adically completed versions). We observe that if the $G$-action on $V$ is trivial then \[\wh{H}^*(G; V) \cong H^*(V) \otimes _{\K}\K((u))\langle\theta \rangle.\] In this paper, we mainly use Tate cohomology. We prove its functorial properties in the following lemma.

\medskip

\begin{lma} \label{lem:properties}
	Let $(V, d_V)$ and $(W, d_W)$ be cochain complexes over $\K$ equipped with $G$-actions.
	\begin{enumerate}[label = (\arabic*)]
		\item Suppose that $H^*(V)\cong 0,$ then the group and Tate cohomology groups are also zero:
		\begin{equation}
		H^*(G; V) \cong 0 \text{ and } \wh{H}^*(G; V) \cong 0. \nonumber
		\end{equation}
		\item Suppose that there is an $G$-equivariant chain map $f \colon (V,d_V) \rightarrow (W, d_W)$ that induces a quasi-isomorphism $f_* \colon H^*(V) \xrightarrow{\cong} H^*(W),$ then one has
		$$f_* \colon H^*(G; V) \xrightarrow{\cong} H^*(G; W) \text{ and }f_* \colon \wh{H}^*(G; V) \xrightarrow{\cong} \wh{H}^*(G; W).$$
		\item Given a short exact sequence of cochain complexes and $G$-equivariant maps between them
		$$0 \rightarrow V_1 \rightarrow V_2 \rightarrow V_3 \rightarrow 0,$$
		there is an induced long exact sequence in group or Tate cohomology groups
		$$\cdots \rightarrow H^*(G; V_1) \rightarrow H^*(G; V_2) \rightarrow H^*(G; V_3) \rightarrow H^{*+1}(G; V_1) \rightarrow \cdots$$
		$$  \cdots \rightarrow \wh{H}^*(G; V_1) \rightarrow \wh{H}^*(G; V_2) \rightarrow \wh{H}^*(G; V_3) \rightarrow \wh{H}^{*+1}(G; V_1) \rightarrow \cdots$$
		\item\label{case: Tate} Let $V$ be a free $\bK[G]$-module. Then $\wh{H}^*(G;V) = 0.$
	\end{enumerate}
	
\end{lma}

\begin{proof}
	For (1), one defines the zero chain maps $F \colon C^*(G; V) \rightarrow 0$ and $\wh{F} \colon \wh{C}^*(G;V) \rightarrow 0.$ There are vertical filtrations on $C^*(G; V)$ and $\wh{C}^*(G;V)$ defined by
	\begin{equation}
	F^kC^*(G;V)=\oplus_{i \geq k , j} P^i \otimes_{\bK[G]}V^j \text{ and }\wh{F}^k\wh{C}^*(G;V)=\oplus_{i \geq k , j} Q^i \otimes_{\bK[G]}V^j.
	\end{equation}
	The condition $H^*(V) \cong 0$ implies that the chain maps $F$ and $\wh{F}$ induce quasi-isomorphisms on the associated spectral sequences converging to $H^*(G;V),$ $\wh{H}^*(G;V),$ and the zero group respectively. By comparison of spectral sequences \cite[Theorem 5.5.11]{Weibel}, one obtains the result.\\
	The proof of (2) follows similarly from the fact that the map of spectral sequences associated to the vertical filtrations on $C^*(G; V)$ and $C^*(G;W)$ induces a quasi-isomorphism on $E^1$-pages. The proof of (3) is standard homological algebra. Finally, (4) follows by explicit calculation for $V = \bK[G],$ since $(Q^{\vee})^{\bullet}$ is exact in all degrees.
\end{proof}

\begin{rmk}\label{rmk: Morita for group coho}
Note that Lemma \ref{lem:properties} implies that for a chain complex $V$ over a field $\K,$ $H^*(G,V^{\otimes p}) \cong H^*(G,H(V)^{\otimes p}),$ and $\wh{H}^*(G,V^{\otimes p}) \cong \wh{H}^*(G,H(V)^{\otimes p}),$ where the action on $V^{\otimes p}$ is by cyclically permuting the factors with suitable signs defined precisely in Equation \eqref{eq: sign perm} below, and similarly for $H(V)^{\otimes p}.$ Indeed as $V$ is quasi-isomorphic to $H(V),$ $V^{\otimes p}$ is $G$-equivariantly quasi-isomorphic to $H(V)^{\otimes p}.$
\end{rmk}

\section{Quasi-Frobenius maps}
Let $(V, d)$ be a graded chain complex over a perfect field $\K$ of characteristic $p,$ that is, a field for which the Frobenius automorphism $\K \to \K,$ $k \mapsto k^p$ is invertible. Our main example is $\K = \F_p.$
The Tate twist $V^{(1)}$ of $V$ is defined to be $V$ as an abelian group, but with the following structure of a $\K$-module: 
\begin{equation}
a \colon \K \times V^{(1)} \rightarrow V^{(1)},  \ \ a(k,x)= g(k) \cdot x,
\end{equation}
where $\cdot$ is the original action of $\K$ on $V,$ and $g:\K \to \K$ is the inverse of the Frobenius automorphism. The differential on $V^{(1)}$ is induced by that of $V.$ Finally, we note that if $\K = \F_p$ then $g = \id$ and hence $V^{(1)}$ coincides with $V$ as a vector space.

Having defined the Tate twist $V^{(1)}$, one considers the Tate complexes associated to the trivial $\zp$-action on $V^{(1)}$ and the $\zp$ action $V^{\otimes p}$ given by 
\begin{equation}\label{eq: sign perm}  \sigma \cdot x_0 \otimes  \cdots \otimes x_{p-1}=(-1)^{|x_{p-1}|(|x_0|+\cdots + |x_{p-2}|)}x_{p-1} \otimes x_0 \otimes \cdots \otimes x_{p-2}.\end{equation}
There is a natural map \[V^{(1)} \to V^{\otimes p},\;\; x \mapsto x^{\otimes p}\] which induces the so-called \textit{quasi-Frobenius map} on the associated Tate complexes
\begin{equation}
F \colon \wh{C}^*(\zp, V^{(1)}) \rightarrow \wh{C}^*(\zp, V^{\otimes p}).
\end{equation}
The name of the quasi-Frobenius map has originated in the study of the non-commutative analogue of the Frobenius map $a \mapsto a^p$ for associative algebras \cite[Section 4]{Kaledin}. At a first glance, the morphism $F$ is not a chain map, because it fails to be additive. However, we will prove that this map descends to an isomorphism of $\K$-modules in homology. The following result and its proof are a slightly more explicit version of \cite[Lemma 4.1]{Kaledin} and its proof.\\

\begin{lma} \label{lma:quasi_Frob}
	Let $(V,d)$ be a graded cochain complex over a perfect field $\K$ of characteristic $p.$ Then there is an isomorphism of $\K((u))\langle \theta \rangle$-modules
	$$F \colon \wh{H}^*(\zp, V^{(1)}) \rightarrow \wh{H}^*(\zp, V^{\otimes p})$$ 
\end{lma}

\begin{proof}
	We will prove the case when $(V, d)$ has trivial differential, as the general case then follows from Remark \ref{rmk: Morita for group coho} above. For a graded vector space $V$ over $\K,$ as the $\zp$-action on $V^{(1)}$ is trivial, this amounts to proving that for each degree $i$, there is an isomorphism of $\K$-modules 
	\begin{equation}
	V^{(1)} \cong  \wh{H}^i(\zp, V^{\otimes p}) \nonumber.
	\end{equation}
	Since for the Tate twist $V^{(1)},$ the quasi-Frobenius map $F$ is $\K$-equivariant (that is, $F(k \cdot v) = k \cdot F(v)$ for all $v \in V^{(1)},$ $k \in \K$), it suffices to check that $\ima(F) \subset \Ker(\wh{d})$ and that $F$ descends to an \textit{additive} isomorphism in homology. First, it is clear that $\ima(F) \subset \Ker(\wh{d})$ since $x^{\otimes p} \in \Ker(1-\sigma) \cap \Ker N$ for all $x \in V^{(1)}.$ Now, let us show that $F$ is additive in homology. Set $\underline{n}:=\{0,1, \cdots, n-1\}.$ If a basis $\{ v_j\}_{j=0}^{n-1}$ of $V^{(1)}$ as a $\K$-module is given, then the induced basis of $V^{\otimes p}$ is of the form \[\{ v_f = v_{f(0)} \otimes \cdots \otimes v_{f(p-1)} \mathbin{|}  f\colon \zp \rightarrow \underline{n}\}.\] Denote by $[f]$ the equivalence class of $f$ in the quotient $\Phi(\zp,\ul{n})$ of $\{ f \colon \zp \rightarrow \underline{n}\}$ by the $\zp$-action of cyclically permuting the inputs. We denote by $\Phi(\zp,\ul{n})^{\rm{nc}} \subset \Phi(\zp,\ul{n})$ the subset of classes $[f]$ with $f$ non-constant. For $i$ even, one has that 
	\begin{eqnarray}
	&& \sum_{f \colon \zp \rightarrow \underline{n} } C_f  \cdot v_{f(0)} \otimes \cdots \otimes v_{f(p-1)} \in \Ker(1-\sigma) \Longrightarrow \nonumber \\
	&&  \sum_{\substack{f \colon \zp \rightarrow \underline{n}, \\ f \text{ is non-constant}  }} C_f  \cdot v_{f(0)} \otimes \cdots \otimes v_{f(p-1)}=N \big(\sum_{f \in \Phi(\zp,\ul{n})^{\rm{nc}}} C_f  \cdot v_{f(0)} \otimes \cdots \otimes v_{f(p-1)} \big). \nonumber
	\end{eqnarray}
	The condition that $f$ is nonconstant is required, as $v^{\otimes p} \in \Ker(1-\sigma)$ and $v^{\otimes p} \notin \ima(N).$ To see that this holds similarly for $i$ odd, one notices that $(1-\sigma)^{p-1}=N$ over a field of characteristic $p$. This implies that on $V^{\otimes p}_{\text{non-const}} =  \{v_f \mathbin{|}  f\colon \zp \rightarrow \underline{n}, \text{ is non-constant}\},$  $\Ker(N) \cong\ima(1-\sigma)$ and $\ima(N) \cong \Ker(1-\sigma)$, and the proof when $i$ is odd follows from the case when $i$ is even. This shows that the quasi-Frobenius map $F$ becomes a $\K$-linear map in homology, since for all $x,y \in V^{(1)},$ \[(x+y)^{\otimes p} = x^{\otimes p} + y^{\otimes p} + c(x,y)\] where $c(x,y) \in V^{\otimes p}_{\text{non-const}} \cap \Ker(1-\sigma) \cap \Ker(N).$ Similarly, we obtain that $F: V^{(1)} \to \wh{H}^i(\zp, V^{\otimes p})$ is an isomorphism. Indeed by the above calculation, or by Lemma \ref{lem:properties}, property (4), \[\wh{H}^i(\zp, V^{\otimes p}) \cong V^{\otimes p}_{\text{const}} = \{v_f \mathbin{|}  f\colon \zp \rightarrow \underline{n}, \text{ is constant}\},\] which is also $\ima(F).$ Moreover $\dim_{\bK} V^{\otimes p}_{\text{const}} = \dim_{\bK} V = \dim_{\bK} V^{(1)}.$ This finishes the proof.
\end{proof}

%

\section{Morse functions on classifying spaces} \label{sec:Morse}

Let us consider the Hilbert space
\begin{equation}
\mathcal{H}= L^2(\Z_{\geq 0},\C) =  \{ z=( z_k )_{k \in \Z_{\geq 0}} \mathbin{|} z_k \in \C, \sum_{k \in \Z_{\geq 0}}|z_k|^2 <\infty\}\end{equation} with the standard inner product $\langle z, w \rangle = \sum_{k \in \Z_{\geq 0}} \, z_k \ol{w}_k.$

Let $\C^{\infty} \subset \cl{H}$ denote the (non-closed) subspace \[\C^{\infty} = \{ z = (z_k)_{k\in \Z_{\geq 0}} \in \cl{H}\;|\; z_k = 0   \;\text{for all}\;k\; \text{sufficiently large}\},\] and define $S^{\infty} \subset \C^{\infty}$ by \[S^{\infty}=\{z\in \C^{\infty}\;|\sum_{k \in \Z_{\geq 0}}|z_k|^2=1\}.\] There is a free action of $\zp$ on $S^{\infty}$ given by 
\begin{equation}
(\m \cdot z)_k=e^{2\pi i \m/p}\cdot z_k\text{ for } \m \in \zp.
\end{equation}
Consider the standard Morse-Bott function 
\begin{equation}
{f} \colon S^{\infty} \rightarrow \R,  \;\; {f}(z):=\sum_k k\cdot|z_k|^2.\end{equation}
It is invariant under the $\zp$-action on $S^{\infty}$ and descends to a Morse-Bott function $f_0$ on $B\zp:=S^{\infty}/ \zp$ with critical submanifolds being the $S^1$-fibers of the fibration $\pi\colon B\zp \rightarrow B S^1.$ In fact ${f}$ is invariant under a natural $S^1$-action on $S^{\infty}$ and descends to a Morse function $\ol{f}$ on $\C P^{\infty}= BS^1 = S^{\infty}/S^1.$ Now the critical manifolds of $f_0$ are given as preimages under $\pi$ of the critical points of $\ol{f},$ and those of $f$ are given as the preimages of these critical points by the natural projection $S^{\infty} \to BS^1.$ The latter critical manifolds are described as follows: for each $l \in \Z_{\geq 0}$ there is precisely one critical submanifold $S^1_l$ of $f$ of Morse-Bott index $2l,$ 
\begin{equation}
S^1_l=\{ (z_k) _{k \in \Z_{\geq 0}} \mathbin{|} |z_l|=1, z_k=0, k \neq l \}.
\end{equation} 

Note that the coindex of each critical submanifold is infinite, while its index is finite. There is an embedding $\tau$ of $S^{\infty}$ into itself such that $\tau^*f=f+1$ defined by
\begin{equation}
\tau(z_0, z_1, z_2, \cdots)=(0, z_0, z_1, \cdots).
\end{equation}
This map is compatible with the $\zp$-action on $S^{\infty}$, which yields an automorphism of $B\zp$ that sends the critical submanifold $S^1_l$ to $S^1_{l+1}$.
By choosing a small $\zp$-invariant perturbation of the Morse-Bott function $f$ near each critical submanifold $S^1_l$, one obtains a $\F_p$-perfect Morse function $F$ on $B \zp$. As the symmetry group $\zp$ is discrete, one can lift the perfect Morse function $F$ to a $\zp$-invariant Morse function $\tilde{F}$ on $S^{\infty}$ such that there are exactly $p$ critical points of index $i$ on $S^{\infty}$, denoted by $Z_i^0, \hdots, Z_i^{p-1}$, lying over each critical point of index $i$ of $F$. We also require $\til{F}$ to satisfy $\tau^* \til{F} = \til{F} + 1.$ In fact, we can use the following explicit $\zp$-invariant Morse perturbation $\til{F}$ of $f:$ \[ \til{F}(z) = f(z) + \eps \cdot \sum_k \re(z_k^p) = \sum_k \left(k \cdot |z_k|^2 + \eps\cdot \re(z_k^p)\right),\] for $\eps > 0$ a sufficiently sufficiently small constant. Its critical points of index $2l+1$ are $z \in S^1_l$ with $z_l \in \mu_p$ and those of index $2l$ are $z \in S^1_l$ with $z_l \in \mu_{2p}\setminus \mu_p = - \mu_p,$ where for an integer $d\geq 1,$ $\mu_d \subset S^1 \subset \C$ denotes the set of roots of unity of order $d.$

One can also choose a Riemannian metric $g$ on $B \zp$ and lift it to a $\zp$-invariant metric $\tilde{g}$ on $S^{\infty}$. We choose the metric to satisfy $\tau^* \til{g} = \til{g}.$ Furthermore we require that the multiplication of each coordinate of $S^{\infty} \subset \C^{\infty}$ by $\zeta \in \mu_p$ is an isometry of $\til{g}.$ Under these conditions, it is easy to see that the gradient of $\til{F}$ is tangent to the submanifolds $S^{2l+1} \subset S^{\infty}.$ Furthermore, by standard transversality methods in Morse theory (cf. \cite{Schwarz-book}), applied inductively to unions $\bigcup_{k \in \Z_{\geq 0}} \tau^k S^{2l+1}$ for increasing $N,$ one can choose $\til{g}$ satisfying the above conditions in such a way that $(\til{F},\til{g})$ is a Morse-Smale pair (that is a Morse-Smale pair on each $S^{2l+1} \subset S^{\infty}$). A useful point of view on this situation is provided by its relation to the cascades complex for Morse-Bott functions (see \cite[Section 5]{BanyagaHurtubise-MB}).

Now due to the $\zp$-invariance of the Morse function $F$ and the Riemannian metric $\tilde{g}$, there is a $\zp$-action on the critical points of index $i$, given by $\m \cdot Z_i^{j} \mapsto Z^{j+ \m \  \mathrm{mod} \  p}_i$, making each cochain group of degree $i$ into a free $\zp$-module of rank $1$. Under this identification, the Morse cochain complex of $\til{F}$ can be written as 
\begin{equation} \F_p[\zp] \xrightarrow{1- \sigma} \F_p[\zp] \xrightarrow{N} \F_p[\zp] \xrightarrow{1- \sigma} \F_p[\zp] \xrightarrow{N} \cdots
\end{equation}
where $\F_p[\zp]$ is the group ring of $\zp$ with coefficients in $\bK = \F_p,$ and $\sigma$ is the action of $1 \in \zp.$ We recall that $N$ is defined as $N=id+\sigma+\sigma^2+\cdots +\sigma^{p-1}.$ The homology of this complex is hence $\F_p$ in degree $0,$ and vanishes in all other degrees. The Morse complex of $F$ is given by tensoring by $\F_p$ over $\F_p[\zp]:$

\begin{equation}
\F_p \xrightarrow{0} \F_p \xrightarrow{0} \F_p \xrightarrow{0} \F_p \xrightarrow{0} \cdots
\end{equation}

The cohomology ring, for $p>2,$ can be identified with $\F_p[u]\langle \theta \rangle$ for a formal degree $2$ variable $u$ and formal degree $1$ variable $\theta,$ so $\theta^2 = 0.$ Here $u$ corresponds to the generator of $H^*( BS^1,\F_p)$ under $\pi ^*,$ for the natural projection $\pi: B \zp \to B S^1,$ and $\theta$ is the preimage of $u \in H^2(B\zp; \F_p)$ under the Bockstein isomorphism $H^1(B\zp;\F_p) \xrightarrow{\sim} H^2(B\zp;\F_p).$ The class $\theta$ evaluates to $1$ on the $\F_p$-homology class of the fiber of $B\zp \to B S^1.$ However, as before, we prefer to complete all complexes $u$-adically, and in this case the homology can be identified as \[R_p = \F_p[[u]]\left< \theta \right>,\] a notation that we keep for the rest of the paper. Furthermore, we set \[\hrp = \F_p((u)) \left<\theta \right> = \rp \otimes_{\F_p[[u]]} \F_p((u))\] for its version with $u$ inverted. This corresponds to $\rp$ being the (completed) $\F_p$ group  cohomology of $\zp,$ and $\hrp$ being the (completed) $\F_p$ Tate cohomology of $\zp.$ Note that $\hrp$ is a vector space of dimension $2$ over $\cK = \F_p((u)).$ Finally, for $p=2,$ the cohomology ring becomes $\F_2[h]$ for a formal variable $h$ of degree $1,$ its completion is $R_2 = \F_p[[h]],$ and its Tate version is $\wh{R}_2 = \F_2((h)).$


\indent We denote by $\sP^{i, \m}_0$ the moduli space of parametrized flow lines $w \colon \R \rightarrow S^{\infty}$ satisfying 
\begin{equation}\label{eqn:grad_flow}
\partial_sw(s)+ \nabla \tilde{F}(w)=0
\end{equation}
\begin{equation}
\displaystyle\lim_{s \rightarrow -\infty} w(s)=Z_i^{\m}, \text{ and }\displaystyle\lim_{s \rightarrow \infty} w(s)=Z_0^0. \nonumber 
\end{equation}
Similarly, we denote by $\sP^{i, \m}_1$ the parametrized flow lines that satisfy \eqref{eqn:grad_flow} and have asymptotic behaviors $\displaystyle\lim_{s \rightarrow -\infty} w(s)=Z_i^{\m}$ and $\displaystyle\lim_{s \rightarrow \infty} w(s)=Z_1^0$. There are free $\R$ actions on $\sP^{i, \m}_0$, $\sP^{i, \m}_1$ defined by translations 
\begin{equation}
r \cdot w(s) \mapsto w(s+r).
\end{equation} 
Their quotients are defined as $\sQ^{i, m}_0=\sP^{i, \m}_0/\R$ and $\sQ_1^{i, \m}=\sP^{i, \m}_1/\R$ respectively. Below, we shall use the notations $\sP^{i, \m}_{\al}$ and $\sQ^{i, m}_\al$ for $\al \in \{0,1\}.$

The standard compactification of the moduli space of unparametrized Morse flow lines, together with the $\tau$-invariance, provides that 
\begin{equation}
\overline{\sQ}^{i, \m}_{\alpha_0} \cong \bigsqcup \sQ^{i_1, \m_1}_{\alpha_1} \times \sQ^{i_2, \m_2}_{\alpha_2} \times \cdots \times \sQ^{i_n, \m_n}_{\alpha_n},
\end{equation}

where the union is taken over all tuples of triples $(i_1,\alpha_1,\m_1),\ldots,(i_n,\alpha_n,\m_n),$ where $\alpha_j \in \{0,1\}$ for all $1 \leq j \leq n,$ such that $\m_1 + \m_2 + \ldots + \m_n=\m$ in $\zp,$ $\alpha_1 =  \alpha_0,$ $i_n = i \;(\mrm{ mod } \ 2),$ and $\sum_{j=1}^n (i_j - \alpha_j) =i - \alpha_0.$ Indeed, for two critical points $Z,Z'$ of $\til{F},$ denote by $\sP(Z',Z)$ the space of parametrized flow lines, and by  $\sQ(Z',Z) = \sP(Z',Z)/\R$ the space of unparametrized flow lines $w(s)$ with $\displaystyle\lim_{s \rightarrow -\infty} w(s)=Z$ and $\displaystyle\lim_{s \rightarrow \infty} w(s)=Z'.$ Then, by standard Morse theory \[\overline{\sQ}^{i, \m}_{\alpha_0} = \bigsqcup \sQ(Z^{[1]}, Z^{[2]}) \times \ldots \times \sQ(Z^{[n]}, Z^{[n+1]})\] the union running over $(n+1)$-tuples of critical points $\{ Z^{[1]},\ldots,Z^{[n+1]} \}$ of $\til{F},$ with $Z^{[n+1]} = Z^{m}_i,$ $Z^{[1]} = Z^0_{\alpha_0}$, and \[\til{F}(Z^{[1]}) < \ldots < \til{F}(Z^{[n+1]}),\] \[\lambda(Z^{[1]}) < \ldots < \lambda(Z^{[n+1]}),\] where $\lambda$ denotes the Morse index. Now write $Z^{[j]} = Z^{m'_j}_{i'_j}$ for $1 \leq j \leq n+1,$ so that $i'_j = \lambda(Z^{[j]}),$ and observe that \[\sQ(Z^{[j]},Z^{[j+1]}) = \sQ(Z^{m'_{j}}_{i'_{j}},Z^{m'_{j+1}}_{i'_{j+1}}) \cong \sQ(Z^{0}_{\alpha_{j}},Z^{m'_{j+1}-m'_{j}}_{i_j}) = \sQ^{i_j,m_j}_{\alpha_j},\] where $m_j = m'_{j+1}-m'_{j}$ in $\zp,$ $i_j - \alpha_j = i'_{j+1} - i'_j,$ and $i_j = i'_{j+1}\;(\mrm{ mod } \ 2),$ by the $\tau$ translation invariance and the $\zp$-invariance of the Morse-Smale data. Observe that $m_1 + \ldots + m_n = m$ in $\zp,$ $\alpha_1 = i'_1 =\alpha_0,$ $i'_{n+1} = i,$ $i_n = i\;(\mrm{ mod } \ 2),$ and $\sum_{j=1}^n (i_j - \alpha_j) = \sum_{j=1}^{n} i'_{j+1} - i'_j = i'_{n+1} - i'_1 = i - \alpha_0.$ Vice versa, given an $n$-tuple of triples $\{(i_j,\alpha_j,m_j)\}_{1 \leq j \leq n}$ as above, we can reconstruct the $(n+1)$-tuple $\{ Z^{[j]} \}_{1 \leq j \leq n+1}.$


\indent A parametrized flow line of the Morse function $\tilde{F}$ on $S^{\infty}$ is a unparametrized flow line of the Morse function $\Psi+ \tilde{F}$ on $\R \times S^{\infty}$, where $\Psi: \R \rightarrow \R$ is a Morse function which has a unique maximum at $r=1$ and a unique minimum at $r=0$. Then compactifying the space of Morse flow lines on $\R \times S^{\infty}$ between $(1, Z^{\m}_i)$ and $(0, Z^{0}_0)$ yields a compactification of the moduli space of parametrized flow lines
\begin{equation}\label{eqn:para_comp}
\overline{\sP}^{i, \m}_{\alpha_0}=\bigsqcup \sQ^{i_1, \m_1}_{\alpha_1} \times \cdots \times\sQ^{i_{d-1}, \m_{d-1}}_{\alpha_{d-1}} \times \sP^{i_{d}, \m_{d}}_{\alpha_{d}} \times \sQ^{i_{d+1}, \m_{d+1}}_{\alpha_{d+1}} \times \cdots \times \sQ^{i_n, \m_n}_{\alpha_n},    
\end{equation}
where the union is taken over the same indexing set of triples $\{(i_j,\alpha_j,\m_j)\}.$ In the subsequent sections, we will use these geometric moduli spaces to define the $\zp$-equivariant Floer cohomology and the $\zp$-equivariant product and coproduct maps correspondingly.

\section{Fixed point Floer cohomology} \label{sec:Floer_coho}

Given a symplectomorphism $\phi$ of a symplectic manifold $(M, \omega)$, we recall a few equivalent definitions of its fixed point Floer cohomology. While these definitions are equivalent, each one highlights different aspects of the theory, which turns out to be useful. A few references for this section are \cite{DostoglouSalamon, DostoglouSalamon-corr, SeidelThesis,Seidel-4d} and \cite{Ono-ICM} for a more general setup.

We assume throughout that out symplectic manifold $(M,\omega)$ is exact or symplectically aspherical. In the case when $(M,\omega)$ is exact we assume that it is a Liouville domain. We consider a suitable class of symplectomorphisms in each case such that the conditions on the symplectic manifold and the symplectomorphism imply that we can work over a ground field $\bK,$ without the presence of a Novikov field. We make no assumption on grading, since our main isomorphisms are essentially those of ungraded filtered homologies.

Similar definitions work in the case when $(M,\omega)$ is closed or tame at infinity and weakly monotone, and symplectomorphisms are monotone, as in \cite{SeidelMCG}. However, these cases necessitate the introduction of Novikov coefficients, which is not the focus herein. We simply note that if $M$ is monotone and simply-connected, then all symplectomorphisms of $M$ are automatically monotone. 

For a time-dependent Hamiltonian $H \in \cl{H} = \sm{\R/\Z \times M, \R}$ we denote by $\phi^t_H$ the time-$t$ Hamiltonian diffeomorphism generated by the time-dependent Hamiltonian vector field $X^t_H,$ produced by the Hamiltonian construction \[ \iota_{X^t_H} \om = - d H_t,\] for $H_t(-) = H(t,-).$ As long as $X^t_H$ is integrable as a time-dependent vector field, this construction is in fact defined for all $t \in \R$ by considering $H(t,x)$ to be a smooth function on $\R \times M,$ $1$-periodic in the $t$ variable. 

%

We recall in detail the hypotheses on $(M,\om)$ and $\phi \in \Symp(M,\om)$ that we consider.

\begin{enumerate}[label=(\alph*)] 
	\item \label{case: exact}(Exact) In this case the symplectic form is $\omega=d\thetam,$ and the Liouville vector field $Z,$ defined by $i_{Z}\omega=\thetam,$ points strictly outwards along the boundary $\partial M,$ which is equivalent to the condition that $\alpha:=\thetam|_{\partial M}$ is a contact form. The symplectomorphism $\phi \colon M \rightarrow M$ is {\em exact}, that is, one has that $\phi^*\thetam=\thetam+ dG_{\phi}$ for some $G_{\phi} \in C^{\infty}_c(M).$ For a small Hamiltonian perturbation $H_{\epsilon},$ which we shall typically choose to be of the form $\epsilon \cdot r$ near $\partial M,$ wherein $r$ is the Liouville radial coordinate, one can always assume that $\phi_{H_{\epsilon}}^1 \circ \phi$ has nondegenerate isolated fixed points. 
	

\item \label{case: symp asph}(Symplectically aspherical) In this case $M$ is closed and $\omega(A) = 0$ for all classes $A \in H_2^S(M;\Z)$ in the image of the Hurewicz homomorphism $\pi_2(M) \to H_2^S(M;\Z).$ In this case we consider Hamiltonian symplectomorphisms $\phi$ of $M,$ that is, the time-one maps $\phi^1_H$ all Hamiltonian isotopies generated by $H \in \cl{H}.$ Moreover, we pick the normalization condition $\int_M H(t,-) \,\om^n = 0$ for all $t \in [0,1]$ on $H.$ Furthermore, here we work with Floer cohomology in the free homotopy class of contractible loops. 


\end{enumerate}
\bs

\begin{rmk}
In case \ref{case: symp asph} one may define fixed point Floer cohomology over $\bK$ under less stringent assumptions. For example one may define it for those $\phi \in \Symp(M,\om)$ for which $\int_C \om = 0$ for all cylinders $C:S^1 \times [0,1] \to M$ with $C(s,0) = \phi(C(s,1))$ for all $s \in S^1$ and $C(0,t) = x_0$ for all $t \in [0,1]$ (such a cylinder represents a loop in the twisted loop space $\cl L_{\phi} M$ described below based at a fixed point $x_0$ of $\phi$). Whenever both $\phi$ and $\phi^p$ satisfy such a condition, our main result Theorem \ref{thm: main}, and its corollary \eqref{eq: Smith in interval} apply. We note, however, that in general this condition requires the symplectic manifold to be symplectically atoroidal, that is, $\om(A) = 0$ for all $A$ represented by continuous maps from $T^2$ to $M,$ at least those representing loops in certain free homotopy classes of loops, and is furthermore not in general preserved under iteration. 

\end{rmk}

	

\bs

For a pair constisting of a symplectic manifold $(M, \omega)$ and symplectorphism $\phi \in \Symp(M,\omega)$, as above, we will describe the fixed point Floer cohomology $HF^*(\phi)$ in the following three ways. In case \ref{case: symp asph}, the cohomology $HF^*(\phi)$ is isomorphic to $H^*(M;\bK),$ and the main interest of our results lies in the associated {\em filtered} cohomology theory.

\subsection{Twisted loop space}\label{subsec: FH twisted}



\indent First consider an exact symplectic manifold $(M, \omega)$ and $\phi \in \mathrm{Symp}(M)$, as in (a) above. The flow of the Liouville vector field $Z$ near $\partial M$ gives rise to a trivialization of the collar neighborhood of the boundary 
\begin{align}
\Psi\colon (-\epsilon, 0] \times& \partial M \rightarrow M,\\ \nonumber (r,& y)\mapsto \phi^r_Z(y).
\end{align}
This implies that for $R=e^r$ one has $R|_{\partial M}=1$ and $Z\cdot R= R$ near $\partial M$. On the collar neighborhood $(-\epsilon,0] \times \partial M,$ the symplectomorphism $\phi$ satisfies
\begin{equation}
\phi^*\thetam-\thetam=dG_{\phi},
\end{equation}
where $G_{\phi}$ is a smooth function on $M$ which vanishes near the boundary $\partial M$, which ensures that near the boundary $\partial M$ we have $\phi^*R=R$. Consider time-dependent $\om$-compatible almost complex structures $J_t$ for $t \in \R$ satisfying \begin{equation}\label{eqn:J invariant} J_t=\phi_*J_{t+1}.\end{equation} Observe that the condition
\begin{equation}\label{eqn:adapted}
dR\circ J_t=-\thetam \text{ on } (-\epsilon, 0] \times \partial M
\end{equation}
is preserved under replacing $J_t$ by its push-forward $\phi_{\ast} J_t = (\phi^{-1})^* J_t$ by $\phi.$ We denote the space of almost complex structures satisfying \eqref{eqn:J invariant} and \eqref{eqn:adapted} for all $t \in \R$ by $\mathcal{J}_{\phi}$.

Given such a symplectomorphism $\phi \in \mathrm{Symp}(M)$, the mapping torus of $\phi$ is defined by
\begin{equation}
M_{\phi}:=\R \times M/(t,\phi(x)) \sim (t+1, x).
\end{equation} 
By construction, there is a natural projection map $\pi \colon M_{\phi} \rightarrow S^1.$ The twisted loop space is defined as
\begin{equation}\label{eq: twisted loop space}
\mathscr{L}_{\phi}M:=\{ x \in C^{\infty}(\R,M) \mathbin{|} x(t)=\phi(x(t+1))\}.
\end{equation}
Furthermore, $\mathscr{L}_{\phi}M$ is identified with the space of smooth sections of $\pi.$

Given $\phi \in \mathrm{Symp}(M),$ we can associate a $1$-form $\alpha_{\phi}$ on the twisted loop space 
\begin{equation} \label{eqn:action_formula}
\alpha_{\phi}(x)(\xi)= -\int_0^1 \omega(\xi(t), \frac{\partial x}{\partial t})\,dt.
\end{equation}
In our exact case, this one-form is given as the differential of the action functional \[\mrm{A}_{\phi}:\cl L_{\phi} \to \R,\] \[\mrm A_{\phi}(x) = -\int_0^1 x^*\thetam - G_{\phi}(x(1)).\]

The critical points of the action functional $\mrm{A}_{\phi}$ are therefore the constant paths in $\cl L_\phi$ at fixed points of $\phi.$ For this reason we shall identity these paths with $\fix(\phi).$ For tracking action values of fixed points, it will be convenient for us to use the functional \[\cl{A}_{\phi} = - \mrm{A}_{\phi}.\] We call the set of critical values of $\mathcal{A}_{\phi}(\fix(\phi))$ the spectrum $\spec(\ul \phi)$ of $\ul \phi = (\phi, G_{\phi}).$ While $G_{\phi}$ is uniquely determined by $\phi$ in our situation, we prefer to keep it in the notation.

Each time-dependent $J_t$ in $\mathcal{J}_{\phi}$ defines a $L^2$-metric on the twisted free loop space $\mathscr{L}_{\phi}M$. With respect to this metric, negative gradient flow lines of $\mrm{A}_{\phi}$ which are asymptotic to fixed points $x_0, x_1$ are in bijection with solutions $u: \R^2 \rightarrow M$ to Floer's equation
\begin{eqnarray} \label{eqn:floer}
&& \partial_s u+ J_t \partial_t u=0;\\
&& u(s,t)=\phi(u(s,t+1));\\
&& \displaystyle\lim_{s\rightarrow -\infty} u(s,t)=x_0(t),\ \ \displaystyle\lim_{s\rightarrow +\infty} u(s,t)=x_1(t).
\end{eqnarray}
The maximum principle applied to the subharmonic function $R(u)$ ensures that no solutions of \eqref{eqn:floer} reach the boundary $\partial M$. We denote the solutions to \eqref{eqn:floer} up to translations in the $s$-direction by $\cM(x_0, x_1)$. For generic choice of $J_t$, the moduli space $\cM(x_0,x_1)$ is a smooth finite dimensional manifold of dimension
\begin{equation}
\dim \cM(x_0,x_1)=|x_0|-|x_1|-1
\end{equation}
The Floer cochain complex is now defined by
\begin{equation}
CF^i(\phi):= \bigoplus_{|x|=i} \K \langle o_x \rangle,
\end{equation}
with differential given by \[d_{\phi}(x_1) = \sum_{|x_0|=|x_1|+1} \# \cl M(x_0,x_1) \, x_0.\] Here the notation $o_x$ is the orientation line associated to the generator $x \in \mathrm{Fix}(\phi)$ defined in Appendix \ref{app:signs and or}, which we shall henceforth omit from the notation. For questions regarding signs in the differential $d,$ we also refer the reader to Appendix \ref{app:signs and or}. Since we work with cohomology, the fixed point $x_1$ should be considered to be the ``input", and the fixed point $x_0$ shall be considered to be the ``output". In fact, throughout the paper we adopt the general convention that postitive cylindrical ends correspond to inputs and negative cylindrical ends correspond to outputs.

It is readily verified that \[\cl A_{\phi}(d_{\phi}x) < \cl A_{\phi}(x),\] where \[\displaystyle \cl A_{\phi} \left(\sum a_j x_j\right) = \max \{\cl A_{\phi}(x_j)\,|\, a_j \neq 0\},\] \[ \cl{A}_{\phi}(0) = -\infty.\]



In the case when the manifold $(M,\omega)$ is closed and symplectically aspherical, and $\phi \in \Ham(M,\om),$ writing $\phi = \phi^1_H$ for $H \in \cl H,$ the twisted loop space $\cl L_{\phi} M$ is identified with the usual free loop space $\cl LM$ by the map $D_H:\cl L_{\phi} M \to \cl LM,$ $z(t) \mapsto \phi^t_H z(t).$ It is easy to see that \[(D_H^{-1})^* \mrm A_{\phi} = \cl A_H,\] \[(D_H^{-1})^* \cl A_{\phi} = -\cl A_H,\] for \[\cl A_H: \cl LM \to \R,\] \[\cl A_H (x) = \int_0^1 H(t,x(t))\,dt - \int_{\overline{x}} \om,\] where $\overline{x}: \D \to M$ is a map with boundary values $\overline{x}(e^{2\pi i t}) = x(t).$ Note that the differential in Floer {\em cohomology} for $H$ decreases $(-\cl{A}_H),$ as it increases $\cl{A}_H.$


In both the exact and the sympectically aspherical case, we denote by $\ul \phi$ the tuple consisting of $\phi$ and the data required to define the action functional $\cl A_{\phi},$ and call it a filtered symplectic brane. In the first case, this means the primitive $G_{\phi}$ of $\phi^*\thetam - \thetam.$ In the second case it can be considered to be either (i) the choice of a base-point of the connected component of $\cl L_{\phi}M$ corresponding to the component $\cl L_{pt} M$ of contractible loops in $\cl LM,$ or (ii) in the $\cl LM$ description: a choice of a Hamiltonian $H \in \cl H$ generating $\phi.$ Observe that for each contractible Hamiltonian loop $\eta:S^1 \to \Ham(M,\om),$ $\eta^0 = \id,$ the map $D_{\eta}:\cl LM \to \cl LM$ given by $z(t) \mapsto \eta^t z(t)$ satisfies $ \cl A_{H} = D_{\eta}^{*} \cl A_{H \# K},$ where $K \in \cl H$ is the normalized Hamiltonian generating the loop $\eta.$ 

Finally, let $\phi$ be non-degenerate, that is, $\ker(D(\phi)_x -\id) = 0$ for all $x \in \fix(\phi).$ Let $I$ be an admissible action window, that is, $I = (a,b)$ with $a<b,$ such that $a,b \in (\R\setminus \spec(\ul \phi)) \cup \{\pm \infty\}.$ We define $HF^*(\phi)^I$ as the homology of the quotient complex \[CF^*(\phi)^I = CF^*(\phi)^{<b}/ CF^*(\phi)^{<a},\] where $CF^*(\phi)^{<c}$ is the subcomplex spanned by generators $x$ of action value $\cA(x) < c.$ These chain complexes and homologies admit natural comparison maps $CF^*(\phi)^{I_1} \to CF^*(\phi)^{I_2}$ for $I_1 = (a_1,b_1), I_2=(a_2,b_2)$ with $a_1 \leq a_2,$ $b_1 \leq b_2.$ When $\phi$ is degenerate, then for $a,b$ as above, we perturb it slightly to a non-degenerate symplectomorphism $\phi_{1,G} = \phi^1_{G} \phi,$ where $\phi^1_G$ is the time-$1$ map of a sufficiently $C^2$-small Hamiltonian $G.$ Then it still satisfies $a, b \in (\R\setminus \spec(\ul \phi)) \cup \{\pm \infty\},$ and for all $G$ sufficiently $C^2$-small the homologies $HF^*(\phi_{1,G})^I$ are canonically isomorphic, whence we define $HF^*(\phi)^I$ as the colimit of the associated indiscrete groupoid\footnote{This notion is known under different names in the literature: ``simple connected system", for example. It comprises a category with precisely one morphism between each two objects. These morphisms are all isomorphisms.}.




Considering Hamiltonian isotopies $\{\phi^t_{G}\}$ induced by Hamiltonians $G,$ it is classical to show that $HF^*(\phi^1_G \circ \phi)$ does not depend on $G.$  This is due to the fact that each symplectic isotopy $\{ \psi^t \}$ of $M$ generated by $1$-forms $b_t$ such that $b_{t+1}=\phi^*(b_t)$ and whose flux $\int_0^1 b_t dt$ satisfies
\begin{equation} \label{eqn:flux_cond}
\int_0^1 b_t dt \in \mathrm{Im}(\phi^*-id) \subset H^1(M; \mathbb{R})
\end{equation}
induces a canonical isomorphism between fixed point Floer cohomologies $HF^*(\psi^1 \circ \phi) \cong HF^*(\phi)$ (see \cite{SeidelMCG}). When $\psi^1 =\psi^1_G$ for some Hamiltonian perturbation, this allows us to define $HF^*(\phi).$ 



\subsection{Symplectic fibrations}\label{subsec: fibrations}

Given a symplectic manifold $(M,\om)$ we let a symplectic fibration $\pi: E \to B$ over a base manifold $B$ with fiber $(M,\om)$ be a smooth fibration with a closed two-form $\Om$ on $E,$ such that for all $z \in B,$ setting $E_z = \pi^{-1}(z),$ $(E_z, \Om|_{E_z})$ is a symplectic manifold symplectomorphic to $(M,\om).$ This is not quite the standard terminology: for example in \cite[Chapter 8]{McDuffSalamon-BIG} symplectic fibrations satisfy a more general condition, while the fibrations we consider are called locally Hamiltonian. Denote by $Vert = \ker(D\pi) \subset TE$ the vertical subbundle of vectors tangent to the fibers. Furthermore, let \[Hor = \{v \in TE\,:\, \iota_v \Omega|_{Vert} = 0\}\] be the horizontal subbundle of $TE.$ It is transverse to $Vert$ and isomorphic to $\pi^*(TB)$ via $D\pi,$ whence it induces an Ehresmann connection on $E \to B.$ An important feature of such fibrations is that the holonomy of this connection over each loop in the base is a symplectomorphism of the fiber, which is Hamiltonian if the loop is contractible. Finally, we let $\Pi: TE \to Vert$ be the projection parallel to $Hor,$ and define the vertical two-form $\Om^{v}$ on $E$ by $\Om^v(\xi,\eta) = \Om(\Pi(\xi),\Pi(\eta))$ for $\xi, \eta \in T_e E.$

When the base $B=S$ is a surface endowed with a complex structure $j_S,$ we call an almost complex structure $J$ on $E$ {\it compatible with the fibration,} and more specifically with $\Om,$ if $J(\ker d\pi) = \ker d\pi$ and $d\pi \circ J = j_S \circ d\pi.$ That is, $J$ preserves the fibers and makes the projection map holomorphic with respect to the complex structure $j_S$ on $S,$ and furthermore $J|_{\ker d\pi}$ is compatible with the symplectic form $\Om|_{\ker d\pi}.$

We note that the mapping torus $M_{\phi} \to S^1$ of a symplectomorphism $\phi \in \Symp(M,\om)$ is a symplectic fibration over $S^1$ with fiber $(M_\phi)_t = M$ over each $t\in S^1$: the form $\om$ on $M$ naturally extends to a closed form $\Om = \om_{\phi}$ on $M_{\phi}$ with $\Om|_{(M_\phi)_t} = \om$ for all $t \in S^1.$  Furthermore, $\fix(\phi)$ is in bijective correspondence with the flat sections $\mathscr{P}_{\phi}$ of $M_{\phi} \to S^1.$ We note that almost complex structures from Section \ref{subsec: FH twisted} satisfying condition \eqref{eqn:J invariant} induce $\Om$-compatible almost complex structures on $\R \times M_{\phi}.$ We shall denote the latter space of almost complex structures $\cl{J}_{M_{\phi}}.$

Following \cite{SeidelThesis}, we can define \[\displaystyle CF^i(\phi) = \bigoplus_{|x|=i} \bK \langle o_x \rangle,\] where now $x$ ranges over the set $\mathscr{P}_{\phi}$ of flat sections of $M_{\phi} \to S^1.$ For the differential, set $Z = \R \times S^1$ with the standard complex structure, and let $\pi_{E_{\phi}}: E_{\phi} = \R \times M_{\phi} \to Z$ be the pullback symplectic fibration of $M_{\phi} \to S^1$ by the natural projection $Z \to S^1.$ Then the differential counts isolated solutions (modulo the $\R$-action by translation) of finite energy $\int_Z u^*\Om^v$ to the equation \begin{align}\label{eq: differential with fibrations}
&u: Z \to E_{\phi},\; \pi_{E_{\phi}} \circ u = \id_Z\\ 
& (du)^{(0,1)} = 0 \nonumber\\
& u(s,t) \xrightarrow{s \to \pm \infty} \sigma_{x_{\pm}}(s,t),\nonumber
\end{align}

where $\sigma_x$ for $x$ a flat section of $M_{\phi} \to S^1$ is the induced flat section of $E_{\phi}.$ Here the convergence is exponential in suitable trivializations over the ends and the $(0,1)$-part is taken with respect to an $\Om$-compatible almost complex structure. The comparison between this definition and the one in Section \ref{subsec: FH twisted} is rather straightforward: it essentially amounts to the well-known Gromov graph trick \cite[Chapter 8]{McDuffSalamon-BIG}. 


\indent Similarly to the previous section, one can consider a class of perturbations to the symplectic connection associated to the symplectic fibration $\pi_{E_{\phi}},$ which is induced by smooth families of $1$-forms $b_t$ satisfying  $b_{t+1}=\phi^*b_t$ and the condition \eqref{eqn:flux_cond}. Every such family $b_t$ is equivalent to an exact $2$-form $B$ on $M_{\phi}$ that pulls back to $dt \wedge b_t$ on $\R \times M.$ Now in the case that $\phi$ has degenerate fixed points, there is an open dense subset $\mathscr{B}_{reg}$ of such $b_t$ considered above such that if we define the new symplectic fibration to be
\begin{equation}
(M_{\phi}, \ \ \tilde{\Omega}=\Omega+B), \nonumber
\end{equation}
then the condition \eqref{eqn:flux_cond} implies that the monodromy of this symplectic fibration becomes $\psi^1 \circ \phi.$ Then the map \begin{align} \Psi \colon \mathscr{P}_{\phi} &\rightarrow  \mathscr{P}_{\psi^1 \circ \phi}, \nonumber \\ \Psi(\gamma)(t)& =\psi^t \circ \gamma(t) \nonumber \end{align} gives rise to the canonical isomorphism $HF^*(\phi) \cong HF^*(\psi^1 \circ \phi).$ Again, it suffices in our case to choose some Hamiltonian isotopy $\psi^t:=\phi_K^t$ generated by some $K \in \sm{[0,1] \times M, \R}$ to ensure the non-degeneracy, and set $b_t = -d(K_t).$


\subsection{Lagrangian graph construction}

Alternatively, we may consider the symplectic manifold $M^{-} \times M = (M \times M, - \om \oplus \om)$ and let \[HF^*(\phi) = HF(\mrm{graph}(\phi),\Delta),\] where \[\mrm{graph}(\phi) = \{(\phi (x), x)\;|\; x \in M \} \subset M^- \times M,\]  \[\Delta = \mrm{graph}(\id_M) \subset M^- \times M\] are (weakly) exact Lagrangian submanifolds. The comparison between this approach and the one in Section \ref{subsec: FH twisted} is again straightforward, and has to do with choosing product type almost complex structures on $M^- \times M.$ We refer for example to \cite{Hend} for details, remarking that there $\mrm{graph}(\phi)$ is defined as $\mrm{graph}(\phi) = \{(x, \phi(x))\;|\; x \in M \} \subset M \times M^{-}.$

For the above Floer cohomology to be well-defined, one needs that the graph of $\phi$ intersect the diagonal Lagrangian transversally. This can always be achieved for example by adding a small Hamiltonian perturbation to $\phi,$ or by introducing a small Hamiltonian perturbation into the Floer equation for the differential. Indeed the non-degeneracy of the fixed points of $\psi_K^1 \circ \phi$ is equivalent to the graph of the perturbed symplectomorphism $\psi_K^1 \circ \phi$ being transverse to the diagonal Lagrangian $\Delta$ in $M^- \times M.$

\subsection{The $\Z/k\Z$-action on $HF(\phi^k)$}

For any integer $k$, we define the fixed point Floer cohomology associated to $\phi^k$ following \cite{Seidel}. Given the function $G_{\phi}$ such that $\phi^* \thetam-\thetam=dG_{\phi}$, one can choose the corresponding function for $\phi^k$ to be
\begin{equation}
G_{\phi^k}=(\phi^*)^{k-1}G_{\phi}+\cdots + \phi^* G_{\phi}  +G_{\phi}.
\end{equation}
so that $(\phi^*)^{k} \thetam-\thetam=dG_{\phi^k}$. Consider the twisted loop space of period $k$ for $\phi,$
\begin{equation}
\mathscr{L}_{\phi,k}M:=\{ x \in C^{\infty}(\R,M) \mathbin{|} x(t)=\phi^k(x(t+k)) \}.
\end{equation}

Note that this space is not the same as $\cl{L}_{\phi^k}$ as defined in \eqref{eq: twisted loop space} in Section \ref{subsec: FH twisted}. However, there is a natural diffeomorphism $\cl{L}_{\phi^k} \to \cl{L}_{\phi,k}$ given by $x(t) \mapsto x(t/k)$ with inverse $\cl{L}_{\phi,k} \to \cl{L}_{\phi^k}$ given by $y(t) \mapsto y(kt).$

On this space one defines the action functional $\mathrm{A}_{\phi,k}: \mathscr{L}_{\phi,k}(M) \rightarrow \R$ given by \[\mrm A_{\phi,k}(x) = - \int_0^k x^*\thetam - G_{\phi^k}(x(1)).\] Set \[\cl A_{\phi,k} = - \mrm{A}_{\phi,k}.\]

There is a $\Z/k\Z$ action on $\mathscr{L}_{\phi,k}(M)$ defined by
\begin{align}
\Z/k\Z &\times \mathscr{L}_{\phi,k}(M) \rightarrow \mathscr{L}_{\phi,k}(M),\\ \nonumber (\m&, x(t)) \mapsto \phi^{\m}(x(t+\m)).
\end{align}
The fixed point set of the $\Z/k\Z$-action is precisely $\mathscr{L}_{\phi} M$ seen as a subset of $\mathscr{L}_{\phi,k} M.$ For $x(t) \in \mathscr{L}_{\phi} M$, we have 
\begin{equation}
\mathcal{A}_{\phi,k}(x(t))=k\,\mathcal{A}_{\phi}(x(t)).
\end{equation}
We denote the space of compatible almost complex structures corresponding to $\phi^k$ by 
\begin{equation}
\mathcal{J}_{\phi,k}=\{(J_t)_{t \in \R} \mathbin{|}J_t \text{ is } \omega\text{-compatible and } J_t=\phi^k(J_{t+k})\}.
\end{equation}
For any integer $k \in \Z,$ there is an action of the finite group $\Z/k\Z$ on $\cJ_{\phi,k}$ defined as follows
\begin{align}\label{eq: action-acs}
\til{\rho} \colon \Z/k\Z &\times \cJ_{\phi,k} \rightarrow \cJ_{\phi,k},  \\ \nonumber (\m&, J_t) \mapsto (\phi^{\m})_*(J_{t+\m}), \forall \m \in \Z.
\end{align}

We call the almost complex structure $J$ in $\mathcal{J}_{\phi,k}$ {\em symmetric} if it is invariant under the $\Z/k\Z$-action defined above, and set for $J \in \cJ_{\phi,k},$ \begin{equation}\label{eq: rho star}\rho_* J = \til{\rho}(1,J).\end{equation}


Next we describe a $\Z/k\Z$-action in Floer cohomology $HF^*(\phi^k)$ for each $k \in \Z.$ Similar to the $\Z/2\Z$ case in \cite[Section 4]{Seidel} and the $\Z/k\Z$ case considered in \cite{PolterovichS},\cite{PSS-stabilization},\cite{Zhang-roots}, and in \cite{Tonkonog-commuting}, \cite[Section 3]{CineliGinzburg}, where the signs and orientations were made explicit, the chain map which induces the generator of the $\Z/k\Z$-action in cohomology is the following composition of maps for $J \in \cl{J}_{\phi,k}$
\begin{equation}
\sigma \colon (CF^*(\phi^k), d_{J}) \xrightarrow{c}  (CF^*(\phi^k), d_{\rho_*J}) \xrightarrow{\rho^*} (CF^*(\phi^k)), d_{J}),  \nonumber
\end{equation}
where $d_{J}$ and $d_{\rho_*J}$ denotes the differentials of the fixed point Floer cochain complex defined using the almost complex structure $J_{\phi^k} \in \mathcal{J}_{\phi^k}$ and its image $\rho_*J_{\phi^k}$ under the action of $1 \in \Z/k\Z.$ The first cochain map $c \colon (CF^*(\phi^k), d_{J_{\phi^k}}) \rightarrow  (CF^*(\phi^k), d_{\rho_*J_{\phi^k}})$ is given by a standard continuation map for the almost complex structures $J_{\phi^k}$ and $\rho_*J_{\phi^k}$, and the second map is induced by applying $\rho_*$ in \eqref{eq: rho star} to the Floer data defining $(CF^*(\phi^k), d_{J_{\phi^k}})$ and identifying the generators with those of $(CF^*(\phi^k), d_{\rho_*J_{\phi^k}}).$ Since we work in cohomology, the map goes in the specified direction. It can be checked that both of these maps are quasi-isomorphisms, and so is their composition. We denote the induced map in cohomology by
\begin{equation}
\sigma_* \colon HF^*(\phi^k) \rightarrow HF^*(\phi^k). \nonumber
\end{equation}
It is straightforward to see that it satisfies $(\sigma_*)^k=id.$ The same definition goes through for cohomology groups in all admissible action windows $I,$ giving a map \begin{equation}
\sigma_* \colon HF^*(\phi^k)^I \rightarrow HF^*(\phi^k)^I,\nonumber
\end{equation}
satisfying $(\sigma_*)^k=id$ and commuting with the interval comparison maps.


In fact, the following alternative cochain-level description of the map $\sigma_{\ast},$ which we still denote by $\sigma,$ will define a homotopical action of $\Z/k\Z$ on $CF^*(\phi^k),$ which is more closely related to equivariant Floer cohomology, as considered in this paper (see Section \ref{sec:equiv-Floer-coho}). One first considers the interpolation of almost complex structures to define the continuation map $\sigma.$ For $J \in \cl{J}_{\phi,k}$ set
\begin{equation}\label{eq: alt descr of zp action}
J_{s,t}=\begin{cases}
J_t, \ \ \ \ \ \  s \gg 1\\
\rho_* J_t, \ \ \ s \ll -1
\end{cases} \nonumber
\end{equation}


Define $M_{\phi,k}$ to be the period $k$ mapping torus of $\phi,$ \[ M_{\phi,k} = (\R \times M)/ (t,\phi^k(x)) \sim (t+k,x).\] This space differs from $M_{\phi^k}$ as defined above, but is diffeomorphic to it. Observe that $J_{s,t}$ gives an almost complex structure $\ol{J} \in \cl{J}_{M_{\phi,k}}.$ That is $\ol{J}$ is compatible with the structure of symplectic fibration on $\R \times M_{\phi,k}$ over $Z_k = \R \times \R/k\Z.$ 

Then one considers section solutions $u \colon Z_k \rightarrow \R \times M_{\phi,k}$ to the following continuation equation:
\begin{equation} \label{eqn:shift_moduli}
\begin{cases}
(du)^{(0,1)}_{J_{s,t}}=0 \\
\displaystyle \lim_{s \rightarrow \infty}u(s,t)= x_0(t) \;\text{ and } \displaystyle \lim_{s \rightarrow -\infty}u(s,t)= \phi(x_1(t+1)).
\end{cases}
\end{equation}
When $|x_0|=|x_1|$, one defines the map $\sigma$ by counting rigid solutions to the above equation. This replaces the continuation map $c$ in the first definition above by the slightly more complicated asymptotic condition at $s \to -\infty.$ See Section \ref{sec:equiv-Floer-coho} for further discussion of related notions.

Let $\gamma$ be a flat section of $M_{\phi}.$ Denote $\gamma$ traversed $k$ times by $\gamma^k.$ Let $\gamma^{(k)}$ be the natural lift of $\gamma^k$ to a flat section of $M_{\phi,k}.$ Now suppose that $x_0 = x_1 = \gamma^{(k)}$ for a flat section $\gamma$ of $M_{\phi}.$ It is convenient to think of such flat sections $\gamma$ of $M_{\phi}$ and $\gamma^{(k)}$ of $M_{\phi,k}$ as Reeb orbits of the Reeb vector field of a natural stable Hamiltonian structure $(\om, dt)$ on $M_{\phi}$ and respectively $M_{\phi,k}.$ Note that $\gamma^k$ is also a Reeb orbit on $M_{\phi}.$ We discuss this further in Section \ref{subsec: def of FH eq}. It has been shown in  \cite[Lemma 6.7]{Zhao} that depending on the parity of the index difference $\mu_{CZ}(\gamma^k)-\mu_{CZ}(\gamma)$ and parity of $k$, the induced map in cohomology $\sigma_*$ can have either trivial or non-trivial signs. Specifically, the linearization of the equation \eqref{eqn:shift_moduli} is given by the following operator 
\begin{equation}
T\colon W^{1,p}(Z_k, \R^{2n+2}) \rightarrow L^p(Z_k, \R^{2n+2}), \ \ \partial_s +J_0\partial_t +S_k(t+\beta(s)), \nonumber
\end{equation} 
where $S_1(t)$ is the loop of symmetric matrices associated to the orbit $\gamma$ defined in equation \eqref{eqn:symmetric_matrix} in the Appendix, $S_k(t)=S_1(kt)$, and the function $\beta \colon \R \rightarrow [0,1]$ is a smooth cut-off function such that 
\begin{equation}
\displaystyle\lim_{s \rightarrow - \infty} \beta(s)=1, \ \ \displaystyle\lim_{s \rightarrow \infty} \nonumber \beta(s)=0.
\end{equation}
Now it is shown in \cite[Lemma 6.7]{Zhao} that this operator acts as an orientation reversing map on $\Det(\mathscr{O}(S_k, S))$\footnote{The definition of the determinant line associated to the Fredholm operator $T$ is in Appendix \ref{app:signs and or}.} if and only if $k$ is even and the difference of Conley-Zehnder indices $\mu_{CZ}(\gamma^k)-\mu_{CZ}(\gamma)$ is odd. In this case the Reeb orbit $\gamma$ is called bad in the literature.

\section{The $\zp$-equivariant Floer cohomology}\label{sec:equiv-Floer-coho}

\subsection{Definitions}\label{subsec: def of FH eq}

\indent The $\Z/2\Z$-equivariant Floer cohomology $HF^*_{\Z/2\Z}(\phi^2)$ has been considered in \cite{Seidel}. In this section, we fix a prime number $p$ and consider the analogous constructions for the $\zp$-equivariant Floer cohomology $HF^*_{\zp}(\phi^p)$.

\indent Let $(M, \phi)$ be a pair satisfying conditions (a) or (b) listed in Section \ref{sec:Floer_coho}. As in Section \ref{subsec: FH twisted} that the mapping torus of $\phi$ is defined as 
\begin{equation}
M_{\phi}:=\R \times M/(t, \phi(x)) \sim (t+1, x). \nonumber
\end{equation}
The symplectic form $\omega$ on $M$ induces a two-form of maximal rank on $M_{\phi}$, which we now denote by $\omega$ by a slight abuse of notation. We consider $\R \times M_{\phi}$ as a locally Hamiltonian fibration over $Z = \R \times S^1,$ as above, and discuss holomorphic sections. However, we can think of this situation in a slightly different way, which is of use for intuition. The pair $(\omega, dt)$ defines a stable Hamiltonian structure on the mapping torus $M_{\phi},$ which is a pair $(\lambda,\omega)$ of a one-form $\lambda,$ and a two-form $\omega,$ such that $\omega$ is of maximal rank, $\lambda \wedge \omega^{n} > 0,$ and $\ker \omega \subset \ker(d\lambda).$ We refer to \cite{CV stab Ham top} and references therein for further discussion of this notion. The symplectization of this stable Hamiltonian structure, which is a symplectic manifold in its own right, is 
\begin{equation}
(\R \times M_{\phi}, \ \ \tilde{\omega}=\omega+ds \wedge dt). 
\end{equation}
As before any almost complex structure $J_t \in \cl{J}_{\phi},$ extends to an $\om$-compatible almost complex structure $\hat{J}_{t}$ on the horizontal fibration $\ker(dt)$ of $M_{\phi} \to \R/\Z.$ In turn, it extends to an $\til{\om}$-compatible almost complex structure $\tilde{J}_{t}$ on $\R \times M_{\phi}$ canonically by seeing the symplectization as
\begin{equation}
\R \times M_{\phi}=\C \times M/(s,t,\phi(x)) \sim (s,t+1,x).
\end{equation}

We let $\cJ(M_{\phi})$ be the space of the former almost complex structures $\hat{J}_t$ and $\cJ(\R \times M_{\phi})$ be the space of the latter $\tilde{\omega}$-compatible almost complex structures $\tilde{J}_t$ obtained by such an extension. 

Given a fixed prime number $p,$ we consider as before the following model of the mapping torus $M_{\phi^p}$: \begin{equation}
M_{\phi,p} =\R \times M/(t, \phi^p(x)) \sim (t+p, x). \nonumber
\end{equation}
This is the pull-back of $M_{\phi} \to S^1$ under the obvious covering map \[S^1_p = \R/p\Z \to S^1 = \R/\Z.\] It is also endowed with a natural stable Hamiltonian structure. Moreover, each $J_t \in \cl{J}_{\phi,p}$ extends to an almost complex structure $\hat{J}_{t}$ on the horizontal fibration $\ker(dt)$ of $M_{\phi,p} \to \R/p\Z,$ and to an almost complex structure $\tilde{J}_{t}$ on the symplectization $\R \times M_{\phi,p}.$ We let $\cl{J}(M_{\phi,p})$ be the space of the former almost complex structures, and $\cl{J}(\R \times M_{\phi,p})$ be the space of the latter almost complex structures.

To define the equivariant differential, we define the following class of almost complex structures on $\R \times M_{\phi,p}$ parametrized by $S^{\infty}.$ Let $\tilde{J}_t \in \cJ(\R \times M_{\phi,p})$ be an almost complex structure on $\R \times M_{\phi,p}.$ We extend $\tilde{J}_t$ to an almost complex structure $\tilde{J}_{t,z}=\tilde{J}_{s,t,x,z}$ (depending trivially on $s \in \R$) on $\R\times {M}_{\phi,p}$ parametrized by $z \in S^{\infty},$ satisfying the following properties
\begin{enumerate}
	\item[$\bullet$](Locally constant at critical points): $\tilde{J}_{t, z}=\tilde{J}_{t} \text{ for $z$ in a neighbourhood of }Z_i^0$ $\text{ for each } i.$
	\item[$\bullet$]($\zp$-equivariance): for all $\m \in \zp$ and $z \in S^{\infty},$ one has that
	$\tilde{J}_{t, \m \cdot z}=\phi^{\m}_*\tilde{J}_{t+\m,z}.$ In particular $\tilde{J}_{t, z}=\phi^{\m}_*\tilde{J}_{t+\m}\text{ for $z$ in a neighbourhood of }Z_i^\m$  for each $i$ and $\m$ in $\zp.$ 
	
	\item[$\bullet$](Invariance under shift $\tau$): $\tilde{J}_{t,z}=\tilde{J}_{t,\tau(z)} \text{ for all $z$ in } S^{\infty}.$
\end{enumerate}

We denote by $\mathcal{J}^{\zp}(\R \times M_{\phi,p})$ the set of all almost complex structures on $\R\times M_{\phi,p}$ parametrized by $S^{\infty}$ satisfying these properties.

Recall that fixed points of $\phi^p$ are in bijection with flat sections $\sP_{\phi,p}$ of $M_{\phi,p} \to S^1,$ which are the Reeb orbits $\gamma$ of the stable Hamiltonian structure $M_{\phi,p}$ with ``period" $\int_{\gamma} dt = p.$ 

For fixed a non-negative integer $i$ and a group element $\m \in \zp.$ The moduli space $\widetilde{\cM}^{i,\m}_{\alpha}(x_0,x_1),$ where $\alpha \in \{0,1\},$ consists of solutions $u\colon \R\times \R/p\Z \rightarrow \R \times M_{\phi,p}$ with $\pi_{\phi,p} \circ u = \id,$ for $\pi_{\phi,p}:  \R \times M_{\phi,p} \to \R \times S^1$ the natural projection, and $w\colon \R \rightarrow S^{\infty}$ to the following parametrized $J$-holomorphic equations 
\begin{equation}  \label{eqn:para}
\begin{cases}
du \circ j=J_{t,w(s)} \circ du,\\
\partial_sw(s)+\nabla \tilde{F}(w)=0,
\end{cases}
\end{equation}

with asymptotic behavior 
\begin{equation}
\displaystyle\lim_{s \rightarrow - \infty}(u(s, t), w(s)) = (\phi^{\m}(x_0(t+m)), Z_i^{\m}), \ \ \ \displaystyle\lim_{s \rightarrow \infty}(u(s, t), w(s)) = (x_1(t), Z_{\alpha}^0).
\end{equation}

Here the input orbit is $x_1 \in \cl{P}_{\phi,p}$ and the output orbit is $\phi^{\m}(x_0(t+m)) \in \cl{P}_{\phi,p}.$ As we consider the fixed points as flat sections, or Reeb orbits in the mapping torus, the latter notion requires a definition. In terms of fixed points $x_0 \in \cl{P}_{\phi,p}$ corresponds to a fixed point $x_0(0) \in M \times \{0\} \subset M_{\phi,p} \cong M$ of $\phi^p.$ Then $\phi^{\m}(x_0(t+m)) \in \cl{P}_{\phi,p}$ corresponds to the fixed point $\phi^{\m}(x_0(0)) \in M$ of $\phi^p.$ In other words, in terms of the twisted loop space, given the twisted loop $x_0(t) \in \cl{L}_{\phi,p} M,$ we have the twisted loop $\phi^{\m}(x_0(t+m)) \in \cl{L}_{\phi,p} M.$ We are using a slight abuse of notation: for a twisted loop $x_0(t)$ we denote the section of $M_{\phi,p}$ that it induces again by $x_0(t),$ while in fact it is given at $t \in \R/p\Z$ by $[(t,x_0(t))] \in M_{\phi,p}.$ Similarly, the section obtained from $\phi^{\m}(x_0(t+m))$ is given at $t \in \R/p\Z$ by $[(t,\phi^{\m}(x_0(t+m)))] \in M_{\phi,p}.$


%


There is a free $\R$-action on $\widetilde{\cM}^{i,\m}_{\alpha}(x_0,x_1)$ given for $r \in \R$ by 
\begin{equation}
r \cdot (u(s,t),w(s)) \mapsto (u(s+r,t),w(s+r)).
\end{equation}
We denote the quotient space by this action by $\cM^{i,\m}_{\alpha}(x_0,x_1):=\widetilde{\cM}^{i,\m}_{\alpha}(x_0,x_1)/\R$. If $\phi^p$ is non-degenerate, that is if $\Ker(D(\phi^p)_x - \id) = 0$ for each fixed point $x$ of $\phi^p,$ then for generic choice of almost complex structure $\til{J}_{t,z} \in \cl{J}^{\zp}(\R \times M_{\phi,p})$ this moduli space is a smooth finite dimensional manifold of dimension
\begin{equation} \label{eqn:index}
\mathrm{dim} \cM^{i,\m}_{\alpha}(x_0,x_1)=|x_0|-|x_1|+i-\alpha -1 \text{ for all }\alpha.
\end{equation}
For $|x_0|=|x_1|-i+\alpha+1$,  one can define $d_{\alpha}^{i,\m}: CF^*(\phi^p)\rightarrow CF^{*+1-i+\alpha}(\phi^p)$ by 
\begin{eqnarray}
&& d_0^{i,\m}(x_1)=\sum_{x_0: |x_1|=|x_0|-i+1} \#\cM^{i,\m}_0(x_0,x_1)x_0, \\
&& d_1^{i,\m}(x_1)=\sum_{x_0: |x_1|=|x_0|-i+2} \#\cM^{i,\m}_1(x_0,x_1)x_0.
\end{eqnarray}

Let $u$ be a formal variable of degree $2$ and $\theta$ be a formal variable of degree $1$ so that $\theta^2=0$ as in Section \ref{sec: group_coho}. We set $d_{\alpha}^i=d_{\alpha}^{i,0}+d_{\alpha}^{i,1}+\cdots+d_{\alpha}^{i,p-1}$. The equivariant differential \begin{equation}\label{eq: equivariant differential} d^{\zp} = d^{\zp}_{\phi^p} \colon CF^*(\phi^p)[[u]]\langle \theta \rangle \rightarrow CF^{*+1}(\phi^p)[[u]]\langle \theta \rangle \end{equation} can be written as 
\begin{eqnarray}
d^{\zp}(x\otimes 1) & = & d_0^{0}(x) \otimes 1 + u\, d_0^{2}(x) \otimes 1+ u^2\, d_0^{4}(x) \otimes 1 + \ldots \nonumber\\
& +& d_0^{1}(x) \otimes \theta + u\, d_0^3(x) \otimes \theta + u^2\, d_0^5(x) \otimes \theta + \ldots \nonumber\\
 d^{\zp}(x\otimes \theta) & = & d^{1}_1(x) \otimes \theta + u d^{3}_1 (x) \otimes \theta+ u^2 d^{5}_1(x) \otimes \theta + \ldots \nonumber\\
 & +& u\, d_{1}^2(x) \otimes 1 + u^2\, d_1^4(x) \otimes 1 + u^3\, d_1^6(x) \otimes 1 + \ldots \nonumber
\end{eqnarray}

We abbreviate this to \[ d^{\zp}(x \otimes 1) = D^{1}_{1}(x) \otimes 1 + D^{\theta}_{1}(x) \otimes \theta\] \[ d^{\zp}(x \otimes \theta) = D^{1}_{\theta}(x) \otimes 1 + D^{\theta}_{\theta}(x) \otimes \theta.\]   

There is a natural projection map
\begin{equation}
\cM^{i,\m}_{\alpha}(x_0,x_i) \rightarrow \sQ^{i, \m}_{\alpha} \text{ induced by } (w, u) \mapsto w.
\end{equation} 
Considering the codimension $1$ strata in the compactifications of the unparametrized moduli spaces ${\sQ}^{i, \m}_{\alpha}$ described in \eqref{eqn:para_comp}, gives, by standard Gromov-Floer compactness, transversality, and gluing arguments, that $(d^{\zp})^2=0$ (for a discussion of the analytic issues see \cite[Section 4.2]{Seidel}). 

We call $CF^*_{\zp}(\phi^p) = CF^*(\phi^p) \otimes \rp$ with the differential $d^{\zp}$ the $\zp$-equivariant cochain complex of $\phi^p,$ and its homology $HF^*_{\zp}(\phi^p)$ the $\zp$-equivariant cohomology of $\phi^p.$ Furthermore, we observe that $CF^*_{\zp}(\phi^p)$ and $HF^*_{\zp}(\phi^p)$ are modules over $\rp.$ We call the homology $\wh{HF}^*_{\zp}(\phi^p)$ of $\wh{CF}^*_{\zp}(\phi^p) = CF^*_{\zp}(\phi^p) \otimes_{\rp} \hrp$ with the differential $\wh{d}^{\zp} = d^{\zp} \otimes \id$ the $\zp$-equivariant Tate homology of $\phi^p.$


\medskip

We outline a selected technical aspect.

\medskip

\begin{lma}\label{lma: trans fixed}
	For generic choice of almost complex structures, all elements in $\widetilde{\cM}^{i, \m}_{\alpha}(x_0, x_1)$ are regular for all $i, \m$ and $\alpha$.
\end{lma}
\begin{proof}
	All solutions to equation \eqref{eqn:para} except for the constant solutions can be made regular by choosing generic almost complex structure as in \cite[Proposition 6.7.7]{MSbook}. For an energy zero solution $w=(a,u) \colon \R \times \R/p\Z \rightarrow \R \times M_{\phi,p},$ after composing with the projection to $M,$ the map $\pi_{M}(w)=u$ is the constant map whose image is some fixed point $x$ of $\phi^p$. The linearization of the Cauchy-Riemann equation of $u$ at $x \in M$ is of the form
	\begin{equation} \label{eqn:linear_para}
	D_u \colon W^{k,p}(\R^2, T_x M) \rightarrow W^{k-1,p}(\R^2, T_x M),\ \ 
	D_u(\xi)=\partial_s \xi + J_{t, w(s)} \partial_t \xi
	\end{equation}
	for some $\xi \colon \R \rightarrow T_xM$ satisfying $\xi(s,t)=D\phi_x^p(\xi(s,t+p))$. As $x$ is a non-degenerate fixed point of $\phi^p$, the linear operator $D_u$ is Fredholm. The index of $D_u$ is 
	\begin{equation}
	\dim{\Ker(D_u)}-\dim{\coker(D_u)}+i=|x|-|x|+i=i.
	\end{equation}
	To achieve surjectivity of $D_u$, it hence suffices to show that the linearization $D_u$ is injective for any almost complex structure $J_{t, w(s)}$. Suppose $D_u \xi=0.$ In particular, we have $||D_u \xi||_{W^{1,2}(\R \times [0,1])}=0$ for $(k,p)=(2,2)$ in \eqref{eqn:linear_para}. With respect to the metric on $T_xM$ defined by $J_{t, w(s)}$ that we chose, one computes
	\begin{equation}
	0=\int_{\R \times [0,1]}|D_u \xi|^2+ 2\int_{\R \times [0,1]} \xi^*\omega_{x} =\int_{\R \times [0,1]} |\partial_s \xi|^2+ |\partial_t \xi|^2.
	\end{equation}
	The second term $\int_{\R \times [0,1]} \xi^* \omega_{x}$ on the left-hand side vanishes by Stokes' Theorem. Hence we can conclude that $\xi=0$ by above equation, which completes the proof.
\end{proof}





By \eqref{eq: alt descr of zp action} and unwinding the definition of $d^{\zp},$ the following statement is evident.

\medskip

\begin{lma}\label{lma: d10 and d21}
The differentials $d^1_0, d^2_1: CF^*(\phi^p) \to CF^*(\phi^p)$ are chain maps, and induce \[[d^1_0] = (1-\sigma_*): HF^*(\phi^p) \to HF^*(\phi^p),\] \[[d^2_1] = N_* = (1+\sigma_* + \ldots + \sigma_*^{p-1}): HF^*(\phi^p) \to HF^*(\phi^p)\] on cohomology. 
\end{lma}




Finally, all constructions and statements in this section have analogues for admissible action windows $I$ for $\phi^p.$ Indeed, by an index argument and the finiteness of $\fix(\phi),$ in the non-degenerate case, $d^{\zp}$ contains only a finite number of terms. Hence for each admissible window $I = (a,b)$ we can choose the perturbation data for the equivariant differential in such a way as to have $CF^*_{\zp}(\phi^p)^{<a}$ and $CF^*_{\zp}(\phi^p)^{<b}$ be subcomplexes with respect to $d^{\zp},$ and then set $HF^*_{\zp}(\phi^p)^I$ as the homology of the quotient complex  \[CF^*_{\zp}(\phi^p)^I = CF^*_{\zp}(\phi^p)^{<a}/CF^*_{\zp}(\phi^p)^{<b}.\] One shows that this does not depend on the choice of perturbations, provided that they are sufficiently small. Finally, for $\phi^p$ degenerate, we perturb it as $\phi_1^p,$ so that $\phi_1, \phi_1^p$ are non-degenerate, and sufficiently $C^2$ close to $\phi, \phi^p.$ Then the interval stays admissible, and making the perturbations for the differential sufficiently small, we can show that $HF_{\zp}(\phi_1^p)^I$ does not depend, up to canonical isomorphism, on $\phi_1$ provided that it is sufficiently close to $\phi.$ The same applies to $\wh{HF}_{\zp}(\phi_1^p)^I.$

%


\subsection{The algebraic spectral sequence}\label{subsec: alg spec sequence}

For the purposes of this section, consider the following grading on $CF^*(\phi^p)\otimes \rp$ that we call the {\em algebraic degree}. Recall that $\rp = \F_p[[u]] \langle \theta \rangle.$ The algebraic degree of $1 \in \rp$ is $0,$ that of $u$ is $2,$ that of $\theta$ is $1,$ and elements of $CF^*(\phi^p)$ have degree zero. Requiring that degrees of non-zero products add up, we extend this degree to $CF^*(\phi^p)\otimes \rp.$ 



Consider the decreasing filtration $\cF^k_{\mrm{alg}} = \cF^k_{\mrm{alg}} ( CF^*(\phi^p)\otimes \rp)$ generated as an $\rp$-module by the elements of $CF^*(\phi^p)\otimes \rp$ of algebraic degree at least $k.$ It is easy to check that for all $k \geq 0,$ \[d^{\zp}(\cF^k_{\mrm{alg}}) \subset \cF^k_{\mrm{alg}},\] and hence $\cF^k_{\mrm{alg}} $ forms a decreasing filtration on $(CF^*(\phi^p)\otimes \rp, d^{\zp}),$ which is complete. We call it the algebraic filtration. The same applies when we fix an admissible action window $I = (a,b),$ $-\infty \leq a < b \leq \infty,$ and consider the algebraic filtration on the complex $(CF^*(\phi^p)^{I}\otimes \rp, d^{\zp})$ in action window $I.$ In each case, this filtration, being complete and exhaustive, gives a regular spectral sequence (see \cite[Definition 5.2.10]{Weibel}) which converges to $HF_{\zp}(\phi^p)^{I},$ as $\rp$ is already complete with respect to this filtration. We start by describing the $E^1$-page of this {\em algebraic spectral sequence}. The only terms in $d^{\zp}$ that do not increase filtration are $d^0_0$ and $d^1_1,$ both corresponding to the usual Floer differential on $CF(\phi^p).$ Hence the $E_1$-page is given by \[E_1 = HF(\phi^p) \otimes R_p.\] Furthermore, the differential in this page is given by \begin{eqnarray}d_{E_1}(x \otimes 1) &= [d^1_0](x) \otimes \theta,\\ d_{E_1}(x \otimes \theta) &= u [d^2_1](x) \otimes 1,
\end{eqnarray}
observing that $(d^{\zp})^2 = 0$ implies that $d^1_0$ and $d^2_1$ are in fact chain maps. Furthermore, by Lemma \ref{lma: d10 and d21}, $[d^1_0] = 1 -\sigma_*,$ and $[d^2_1] = 1 + \sigma_* + \ldots + (\sigma_*)^{p-1},$ for $\sigma_*: HF(\phi^p) \to HF(\phi^p)$ the generator of the $\zp$-action. This means that the $E_2$-page of the spectral sequence is given by \[E_2 = H^*(\zp,HF(\phi^p))\] as an $\rp$-module. 
We claim that there exists a differential $d_{2,\infty}$ of $\rp$-modules on $H^*(\zp,HF(\phi^p))$  whose homology calculates $H^*_{\zp}(\phi^p).$ Indeed, by an application of the homological perturbation lemma \cite[BPL]{Markl-ideal} to the initial complex, making the homology subspaces and projection operators invariant with respect to multiplication by $u,$ there exists a differential $d_{1,\infty}$ on $E_1$ of the form \begin{equation}\label{eq: perturbed differential on E_1}d_{1,\infty} = d_{E_1} + D_1,\end{equation} where $D_1$ consists of maps of order $\geq 2$ in the algebraic filtration. In a similar way, a second application of the homological perturbation lemma with respect to the splitting \eqref{eq: perturbed differential on E_1}, produces the required differential. Consequently, tensoring with $\bK((u))$ over $\bK[[u]],$ we obtain a differential $\wh{d}_{2,\infty}$ of $\hrp$-modules on $\wh{H}^*(\zp,HF(\phi^p)),$ whose homology is $\wh{H}^*_{\zp}(\phi^p).$ In particular, we obtain that \begin{equation}\label{eq: dim Tate bound} \dim_{\bK((u))}  \wh{H}^*_{\zp}(\phi^p) \leq \dim_{\bK((u))} \wh{H}^*(\zp,HF(\phi^p)).\end{equation}

Furthermore, the same applies to cohomology in all admissible action windows.

\section{Action and energy estimates}\label{subsec: estimates}

In this section we collect two classical results on action and energy of solutions to the Floer equation that we resort to throughout the paper.

Let $\bar{S}$ be a closed Riemann surface with complex structure $j_S,$ and with $k_{-} + k_{+}$ marked points $\Gamma: I \to \bar{S},$ where $I = I_{-} \sqcup I_{+},$ $I_{-} = \{1,\ldots,k_{-}\},$ $I_{+} = \{1,\ldots,k_{+}\}.$ We denote by $S = \overline{S} \setminus \Gamma(I)$ the associated punctured Riemann surface. Let $S$ be equipped with holomorphically embedded cylindrical ends $\epsilon^{-}_j: (-\infty,-1] \times S^1 \to S$ around $\Gamma(j),$ $j\in I_{-},$ and $\epsilon^{+}_j: [1,\infty) \times S^1 \to S$ around $\Gamma(j),$ $j\in I_{+}.$ 


Let $u: S \to E$ be a section of a locally Hamiltonian symplectic fibration $\pi: E \to S,$ with cylindrical trivializations over the cylindrical ends of $E$ and connection form $\Om.$ 

Assume that $u$ converges to sections $x^{\pm}_j$ of $E_{j^{\pm}} \to S^1$ for $j^{\pm} \in I_{\pm}$ in the cylindrical ends corresponding to $\Gamma(j).$ Assume that $u$ satisfies the Floer equation $(du)^{(0,1)} = 0$ with respect to an $\Om$-compatible almost complex structure on $E.$ Then the energy $E(u) = \int_S u^* \Om^{v}$  satisfies \[E(u) = \int_S u^* \Om - \int_S u^* R(\Om),\] where $R(\Om)$ is the curvature form of the locally Hamiltonian fibration (cf. \cite[Lemma 8.2.9]{McDuffSalamon-BIG}). Furthermore, if we equip $E_{j^{\pm}} \to S^1$ for $j^{\pm} \in I_{\pm}$ with the structure of a filtered symplectic brane, in a way that is compatible with $(E,\Om),$ so that \[\int_S u^* \Om = \sum_{j \in I_{-}} \cA_{j}(x^-_{j}) - \sum_{j \in I_{+}} \cA_{j}(x^+_{j}),\] we obtain the action identity 



\begin{equation}\label{eq: main action estimate} 
\sum_{j \in I_{-}} \cA_{j}(x^-_{j}) - \sum_{j \in I_{+}} \cA_{j}(x^+_{j}) = E(u) + \int_S u^* R(\Om).\end{equation}

 This way, lower bounds on $E(u),$ the most elementary of which is of course $E(u) \geq 0,$ combined with lower bounds on the curvature term $\int_S u^* R(\Om)$ uniform in $u,$ yield upper bounds on the action shift $\sum_{j \in I_{+}} \cA_{j}(x^+_{j}) - \sum_{j \in I_{-}} \cA_{j}(x^-_{j}).$ Typically, we will be able to make our curvature terms arbitrarily small.

Secondly, to describe a slightly more sophisticated lower bound on $E(u),$ we adapt a well-known monotonicity lemma \cite[Proposition 4.3.1.(ii)]{Sikorav} to show the following quite general ``crossing energy" type argument. One may prove the same statement in a different way (following \cite{GG-revisited,GG-hyperbolic,McLean-geodesics}), applying the target-local Gromov compactness of Fish \cite{Fish-compactness}, but we choose to present a more elementary argument, which is sufficient for our purposes.


Suppose that our Riemann surface $S,$ as well as the (locally Hamiltonian) symplectic fibration $E$ over it (see Section \ref{subsec: fibrations}), is obtained by a branched cover from an $s$-invariant fibration over the standard cylinder $Z = \R \times S^1.$ By $s$-invariance we mean that for a symplectic fibration $E_0 \to S^1,$ the fibration $E_1 \to Z$ satisfies $E_1 = \pi_{S^1}^* E_0,$ for the projection $\pi_{S^1}: Z \to S^1$ to the $S^1$-factor. Furthermore, we assume that the Hamiltonian perturbation term is sufficiently small. Then, if the asymptotic conditions of a solution $u$ are distinct, this solution must satisfy an energy monotonicity statement: its energy is bounded from below by $\eps > 0,$ depending only on $J,$ the asymptotic Floer data, and on the isolating neighborhoods of the asymptotics. These estimates shall be used later in defining various flavors of local Floer cohomology and related operations. 

In the statement below, we consider Floer data $(J,K)$ where $K\in \Om^2(S,\sm{E,\R})$ is exact, and $J$ is compatible with $\Om_K = \Om + K.$

\medskip

\begin{prop}\label{prop: monotonicity}

Let $p: S \to Z = \R \times S^1$ be a branched covering with finite branch set. 
Let $(E,\Om)$ be a symplectic fibration over $S,$ obtained by base-change by $p$ from
a symplectic fibration over $Z$ of the form $(E_1,\Om_1) = \pi_{S^1}^* (E_0,\Om_0),$ where
$(E_0,\Om_0)$ is a symplectic fibration over $S^1.$ Let $J_0$ be an almost complex structure compatible with the fibration $p: S \to Z,$ and let $x,y: S^1 \to E_0$ be two different flat sections of $E_0.$ 

Then there exist $\eps > 0$ depending only on $x,y,$ and $(E,\Om)$ such that for all Floer data $(J,K)$ sufficiently close to $(J_0,{0})$ in the $C^{\infty}$ topology the following holds. Each section $u: S \to E$ of $E \to S$  which is a solution to the Floer equation $du^{(0,1)} = 0$ with respect to $(J,K)$ with horizontal asymptotics at the cylindrical ends of $S,$ at least two of which are $x,y,$ satisfies \[E(u) \geq \epsilon > 0.\] Furthermore, there exists a small neighborhood $U_x$ of $\sigma_x,$ such that if all the asymptotics of $u$ are at $x,$ then the image of $u$ is contained in $U_x.$
\end{prop}


\begin{proof}
Let us first deal with the case of at least two different asymptotics $x,y.$ To begin, we assume that $(J,K) = (J_0,0).$ Since our fibration is obtained from $(E_1,\Om_1) = \pi_{S^1}^* (E_0,\Om_0),$ the sections $x,y$ of $(E_0,\Om_0)$ give us flat sections $\sigma_x,\sigma_y : S \to E.$ The asymptotic conditions on $u$ yield that $u$ is asymptotic to $\sigma_x$ on a cylindrical end $\epsilon_x$ of $S,$ and to $\sigma_y$ on cylindrical end $\epsilon_y$ of $S.$ Clearly, as $x,y$ are flat and different, they have disjoint images and so do $\sigma_x,\sigma_y.$ Consider a small tubular $2\delta$-neighborhood $U_x(2\delta)$ of $\ima(\sigma_x)$ inside $E$ that is disjoint from $\ima(\sigma_y),$ and admits a trivialization $\Phi: U_x(2\delta)  \xrightarrow{\sim} D^{2n}(2\delta) \times S$ as a symplectic fibration, with $\Phi(\sigma_x(z)) = (0,z)$ and $D\Phi_{\sigma_x(z)} (Vert_{\sigma_x(z)}) = T_0(D^{2n}(2\delta)) \oplus 0 = Vert^{triv}_{(0,z)}$ for all $z \in S.$ Let $S_x(\delta) = \overline{U}_x(\frac{3}{2} \delta)\setminus U_x(\frac{1}{2} \delta) \subset U_x(2\delta)$ be a shell around $\ima(\sigma_x).$ Note that $S_x(\delta) = \Phi^{-1}(S^{2n}(\delta) \times S),$ where $S^{2n}(\delta) = \overline{D}^{2n}(\frac{3}{2} \delta) \setminus {D}^{2n}(\frac{1}{2} \delta).$ We claim that there exists $\delta > 0,$ and $\epsilon > 0,$ as in the formulation of Proposition \ref{prop: monotonicity}, such that \[ E(u) \geq \int_{u^{-1}(S_x(\delta))} u^* \Om^v \geq \epsilon >0. \] We start with the obvious observation that \[\int_{u^{-1}(S_x(\delta))} u^* \Om^v = \int_{(\Phi \circ u)^{-1}(S^{2n}(\delta) \times S)} (\Phi \circ u)^* \Om^v_{triv},\] where now $\Om^v_{triv} = \Phi_*(\Om^v)$ is a certain symplectic connection form on the trivial fibration $D^{2n}(2\delta) \times S.$ Note that at the points of $\{0\} \times S,$ $\Om^v_{triv} = \om \oplus 0,$ since $\sigma_x$ was a flat section. This means $\Om^v_{triv} = \om \oplus 0 + \eta,$ for a two-form $\eta$ on $D^{2n}(2\delta) \times S$ such that $||\eta_{w,z}|| \leq C_4 \cdot |w|$ for $(w,z) \in D^{2n}(2\delta) \times S,$ and $C_4 \geq 0$ a constant.  


Set $v = \Phi \circ u.$ Note that $J'_0 = \Phi_*(J_0)$ preserves the fibers, and is of the form \[J'_0 (w,z) = \begin{bmatrix}
J^v_{w,z}       & A_{w,z} \\
0       & j_{z} 
\end{bmatrix},\] where $j = j_S$ is a complex structure on $S,$ $J^v_{z,w}$ is a $z$-dependent almost complex structure on $D^{2n}$ compatible with $\om,$ and $A_{w,z}$ is a $(j_z,J^v_{w,z})$ anti-complex map $T_z S \to T_w D^{2n}(2\delta).$ By definition, it is clear that $\{J^v_z\}_{z \in S}$ is contained in a compact set of almost complex structures on $D^{2n}.$ Consider the standard Riemannian metric $g$ on $D^{2n}.$ Then there exists a uniform constant $C_1$ such that $\om(\xi,J^v_{z,w} \xi) \geq C_1 \cdot g_w(\xi, \xi)$ and $||\om || \leq C_2$ for all $z \in S, w \in D^{2n}.$ Furthermore, as $A_{0,z} = 0$ for all $z \in S,$ again by definition, there is a uniform constant $C_3$ such that $||A_{w,z}|| \leq C_3 \cdot |w|.$ Furthermore, we may choose these constants for $\delta = \delta_0,$ and keep them for all $\delta \leq \delta_0.$ Choosing locally two vectors $\del_s,\del_t$ tangent to $S$ with $j \del_s = \del_t,$ we obtain that \begin{IEEEeqnarray}{C} \label{eq: estimate on Om vert} v^*\Om^v_{triv}(\del_s,\del_t) = v^*\Om^v_{triv}(\del_s,j\del_s) =\\ \nonumber \\ \nonumber  =\om_{v(z)} (\del_s v, J^v_{z,v(z)} \del_s v) + \om_{v(z)} (\del_s v, A_{z,v(z)} \del_s v)  + \eta_{z,v(z)}(\del_s v, (J^v_{z,v(z)} + A_{z,v(z)})\del_s v) \geq \\ \nonumber \\ \nonumber
 \geq C(\delta) |\del_s v|^2,\end{IEEEeqnarray} where \[C(\delta) = C_1  - C_2 C_3  \delta- 2C_2 C_1^{-1}C_4 \delta - 4C_3 C_4 \delta^2.\]

 Similarly, \[v^*\Om^v_{triv}(\del_s,\del_t)  \geq C(\delta) |\del_t v|^2.\] Set $\delta_1 = \min\{\delta_0, \delta_{\ast}\},$ where $C(\delta) \geq C_1/2$ for all $\delta \leq \delta_{\ast}$ and choose $\delta \in [\frac{1}{2} \delta_1, \delta_1].$ Now for a compact submanifold with boundary $B \subset v^{-1} (S^{2n}(\delta) \times S),$ we have $\mrm{Area}_g(v|_{B}) \leq C_5 \int_B v^*\Om^v_{triv},$ for a suitable constant $C_5.$ From now on, for generic $\delta \in [\frac{1}{2} \delta_1, \delta_1],$ the argument of Sikorav \cite[Section 4.3]{Sikorav} applies without change to show that $\int_{v^{-1}(S^{2n}(\delta) \times S)} v^* \Om^v_{triv} \geq \epsilon > 0,$ for $\epsilon$ that depends only on the geometric situation: $x,y, J_0, E.$ 

To obtain the result for general data $(J,K)$ sufficiently close to $(J_0,0)$ one may for example apply the appropriate version of Gromov compactness \cite[Section 12]{SeidelThesis}, or argue in an elementary way as above, with a few extra estimates. Indeed, taking $(J,K)$ to be sufficiently $C^{\infty}$-close to $(J_0,0),$ and denoting by $\Om_K$ and $\Om^v_K$ the corresponding connection forms, and by $Hor_K$ the new horizontal distribution, we obtain that $\Om^v_{K,triv} = \Phi_*(\Om^v_K)$ in the trivialization $\Phi$ above near $\sigma_x$ satisfies $\Om^v_{K,triv} = \Om^v_{triv} + \Theta,$ for a two-form $\Theta$ with $||\Theta||<\delta_1,$ while $J' = \Phi_*(J)$ still preserves the fibers, and is of the form \[J'(w,z) = \begin{bmatrix}
J^v_{w,z}       & A_{w,z} + B_{w,z} \\
0       & j_{z} 
\end{bmatrix}  = J'_0(w,z) + \begin{bmatrix}
0     &  B_{w,z} \\
0       &  0
\end{bmatrix} \] where $||B_{w,z}|| < \delta_2,$ for arbitrary pre-fixed $\delta_1, \delta_2 > 0.$ Similarly to \eqref{eq: estimate on Om vert} we obtain the bound \[v^* \Om^v_{triv,K}(\del_s, \del_t) \geq C'(\delta,\delta_1,\delta_2) |\del_s v|^2\] and a similar one in terms of $|\del_t v|^2,$ where now \[C'(\delta,\delta_1,\delta_2) = C(\delta) - C_2 \delta_2 - 2 C_3 \delta \delta_2 - C_2 C_1^{-1} \delta_1 - 2 C_3 \delta_1 \delta - \delta_1 \delta_2.\] Hence there exist $\delta_{1,\ast}, \delta_{2,\ast} > 0,$ and $\delta'_{\ast} >0,$ such that for all $\delta_1 \leq \delta_{1,\ast},$ $\delta_2 \leq \delta_{2,\ast},$ and $\delta\leq \delta'_{\ast},$ we have $C'(\delta,\delta_1, \delta_2) \geq C_1/2.$ Now continuing as above yields the desired estimate.


Finally, for the case of all asymptotics identical, if $u$ is not contained in $U_x = U_x(2\delta),$ then the same argument as above applies to show that $E(u) \geq \epsilon > 0.$ However, $E$ being a branched cover of a cylinder, the curvature of $\Om$ is zero. Therefore, for $(J,K)$ close to $(J_0,0),$ by formula \eqref{eq: main action estimate} the energy $E(u)$ is arbitrarily close to $0.$ This is a contradiction.
\end{proof}

\section{$\zp$-equivariant pants product and coproduct}\label{sec:prod-coprod}


\subsection{The product}

In this section we define the $\zp$-equivariant pants product. It generalizes Seidel's definition \cite{Seidel} of the equivariant pair-of-pants product for $p=2.$ Its Morse analogue, described in Section \ref{subsec:Betz-Cohen}, is related to the Steenrod $p$-th power operations on the cohomology of a compact manifold. Broadly speaking, this product relates the Floer cohomology of $\phi$ and the equivariant Floer cohomology of $\phi^p$ for a prime $p.$ The key point is that the constant ``$p$-legged pants" (see Figure \ref{Figure: prod curve}) with $p$ inputs all given by $x,$ a fixed point of $\phi,$ and $1$ output $x^{(p)}$ being $x$ considered as a fixed point of $\phi^p,$ is in general not regular even for generic choices of auxiliary data, because it may have negative index. However, when set up suitably and counted in a positive-dimensional family, in this case coming from the space of negative flow-lines of the $\zp$-invariant Morse function $\til{F}$ on $S^{\infty},$ the constant ``$p$-legged" pants do contribute to the product. In fact, one of our main technical results, detailed in Section \ref{sec:locally invertible}, states, in rough terms, that their contributions are non-trivial. We refer to \cite{Seidel} for further introduction to this notion in the case $p=2.$ 

We proceed with technical definitions, which by the abundance of parameters that need to be taken into account are really quite elaborate. To help the reader follow them we now outline their meaning. We wish to take the $\zp$-symmetry of the ``$p$-legged pants" curve $S_{\cP}$ defined below into account. This $\zp$-symmetry rotates the output cylindrical end and cyclically permutes the input cylindrical ends, and we choose the auxiliary data accordingly. Furthermore, to be compatible with the Borel-type construction of equivariant cohomology that we described above, the auxiliary data is parametrized by $S^{\infty}$ and is required to satisfy natural $\zp$-equivariance properties. Finally, to simplify the complexes that we consider, as we have done in the definition of equivariant Floer cohomology, we require the auxiliary data to be invariant under the shift map $\tau: S^{\infty} \to S^{\infty}.$ Finally, to define the product operation we consider solutions to the Floer equation with the auxiliary data evaluated on negative gradient trajectories of $\til{F}$ in $\sP_{\alpha}^{i,m}.$ In Section \ref{subsec: chain map} we explain that the product indeed provides a chain map between suitable equivariant complexes.

Let $h \colon S_{\cP} \rightarrow \R \times S^1$ be the branched cover of $\R \times S^1$ at $(0,0) \in \R \times S^1$ of ramification index $p$ defined explicitly via the commutative diagram, where the horizontal arrows are isomorphisms:
\[
\xymatrix@C+2pc{
S_{\cP} \ar[d]^{h} \ar[r] & \C \backslash \{e^{\frac{2\pi i\m}{p}} \} \ar[d]^{y=1-z^p} \\
 \R \times S^1 \ar[r]^-{\psi^{-}}  & \C^*,
}
\] 
where for $(s,t) \in \R \times S^1,$ $\psi^{-}$ is given by \[ \psi^{-}(s,t) = e^{\large{-}2\pi (s+i t)} .\] The covering transformation group $\zp$ acts on $S_{\cP}$. For each positive puncture of $S_{\cP}$, there is a trivialization of the cylindrical ends of $S_{\cP}$ over $s \geq 1,$
\begin{eqnarray} \label{eqn:pos_cyn}
&& \epsilon_i^+ \colon [1, \infty) \times S^1 \rightarrow S_{\cP}, \ \ i \in \zp,\\
&&  h(\epsilon_i^+(s,t))=(s,t) \in \R \times S^1, \ \ \m\cdot (\epsilon_i^+(s,t))= \epsilon_{i+\m \ \mathrm{mod}\  p}^+(s,t) \text{ for } \m\in \zp. \nonumber
\end{eqnarray}
Recall that we denote $S^1_p = \R/p\Z.$ For $t \in S^1_p$ we shall usually denote its class $[t] \in S^1,$ where $S^1 = \R/\Z,$ again by $t.$ For the negative puncture, one has cylindrical trivializations for $s \leq -1$ given by
\begin{eqnarray} \label{eqn:neg_cyn}
&& \epsilon^-_i \colon (-\infty, -1] \times S^1_p \rightarrow S_{\cP}, \ \  i \in \zp,\\
&& h(\epsilon^-_i(s,t))=(s, t), \ \ \m\cdot (\epsilon^-_i(s,t))= \epsilon^-_i(s,t+\m) = \epsilon^-_{i+m}(s,t), \; \m\in \zp 
\end{eqnarray}

Since by the last property, the trivializations $\epsilon^-_{i}$ are equivalent, we work with the fixed choice $\epsilon^{-}_0$ of such a negative end. The curve $S_{\cl{P}}$ and its cylindrical ends are described in Figure \ref{Figure: prod curve}. To streamline the exposition, we already mention that in this paper we adopt the standard convention that positive punctures, namely those equipped with a positive cylindrical end, correspond to inputs in Floer-cohomological operations, and negative punctures, those with negative cylindrical ends, correspond to outputs.

\begin{figure}[htb]
	\centerline{\includegraphics{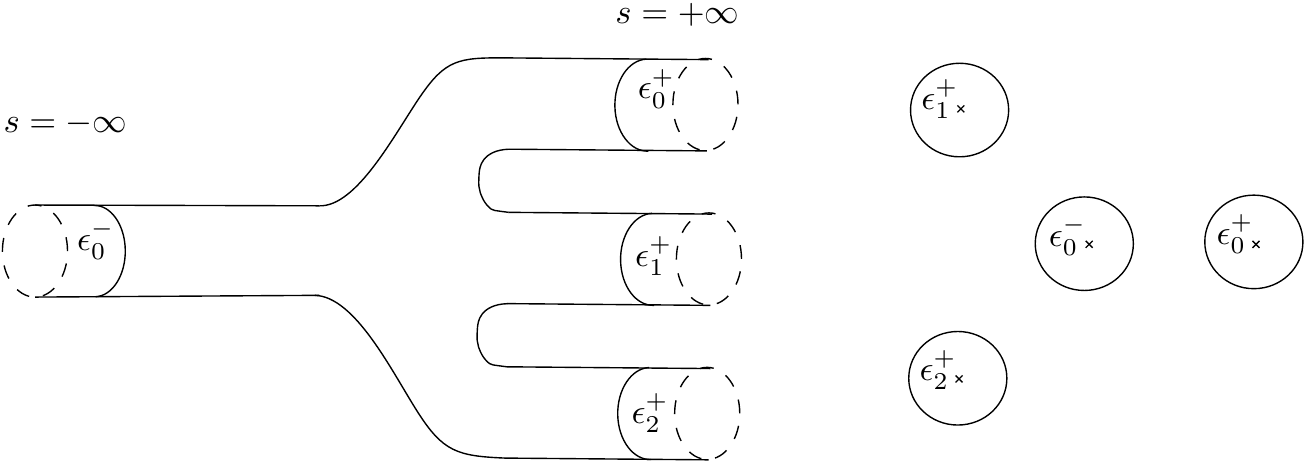}}
	\caption{The curve $S_{\cl{P}}$ and its cylindrical ends for $p=3.$ Left: schematic picture emphasizing the curve structure. Right: schematic picture emphasizing the $\Z/3\Z$ symmetry; this is a $\Z/3\Z$-symmetric configuration of $4$ punctures and cylindrical ends, $3$ positive and $1$ negative, on $\C P^1 \cong \C \cup \{ \infty \}.$ }\label{Figure: prod curve}
\end{figure}

Explicitly, choosing a branch of the logarithm on $\C^*$ around $z=1$ for \eqref{exp:root1} and around $z=-1$ for \eqref{exp:root2}, we write
\begin{eqnarray}
&& {S_{\cP}}=\{ (s,t,y) \in \R \times S^1 \times \C \mathbin{|} y^p=1-e^{-2\pi(s+it)} \} \nonumber \\
&& \label{exp:root1} \epsilon^+_k(s,t)=(s, t, e^{2\pi ki/p}(1-e^{-2\pi(s+it)})^{\frac{1}{p}});  \\
&& \label{exp:root2} \epsilon^-_k(s,t)=(s, t, e^{2\pi ki/p}e^{-2\pi(s+i t)/p}(e^{2\pi(s+i t)}-1)^{\frac{1}{p}}); \\
&&  {\m}\cdot (s,t,y)=(s,t, e^{-2\pi i m/p}\cdot y),\; m \in \zp \nonumber
\end{eqnarray}
Note that in \eqref{exp:root1} the variables are $(s,t) \in [1,\infty) \times S^1,$ while in \eqref{exp:root2} we have $(s,t) \in (-\infty,-1] \times S^1_p.$

Having defined the domain Riemann surface $S_{\cP},$ we prescribe the families of $\zp$-equivariant Floer data 
$$(J_{z, w}, Y_{z}) \in C^{\infty}(S_{\cP} \times S^{\infty}, \cJ(M_{\phi})) \times \Omega^1(S_{\cP}, T(\R \times M_{\phi}))$$
that are parametrized by $S_{\cP} \times S^{\infty}$ and $S_{\cP}$ respectively. Let $\cl{H}(M_{\phi}),$ respectively $\cl{H}(M_{\phi,p}),$ denote the spaces of smooth functions $\til{H}: M_{\phi} \to \R,$ respectively $\til{H}: M_{\phi,p} \to \R,$ coming from smooth functions $H:\R \times M \to \R$ that satisfy $H(t+1,x) = H(t,\phi(x))$ and respectively $H(t+p,x) = H(t, \phi^p(x)).$

One first chooses $(\tilde{J}_{t,w}, \tilde{H}_t)$ in $\mathcal{J}^{\zp}(\R \times M_{\phi,p}) \times  \mathcal{H}(M_{\phi,p})$ and $(\tilde{J}_t^i, \tilde{H}_t^i)$ in $\mathcal{J}(\R \times M_{\phi}) \times \mathcal{H}(M_{\phi})$ for $i=0,1, \cdots,p-1$ as the initial Floer data, with $\til{H}_t \equiv 0,$ $\tilde{H}_t^i \equiv 0.$ Recall that $\cl{J}(M_{\phi})$ can be naturally seen as a subspace of $\cl{J}(M_{\phi,p}).$ Then one can define almost complex structures $J_{s,t,w}^-$ on $\R \times M_{\phi}$ and $J_{s,t,w}^{+, i}$ on $\R \times M_{\phi,p}$ parametrized by $\R \times S^1 \times S^{\infty}$ as follows. For $s \leq 1,$ one requires $J^-_{s,t,w}$ to satisfy the following properties

\begin{enumerate}
\item[$\bullet$](Prescribed on the cylindrical ends): $J_{s,t, w}^-=\tilde{J}_{t,w}$ for $s \leq -2$ and $w \in S^{\infty}.$ 
\item[$\bullet$] (Interpolation): $J_{s,t,w}^-\in \mathcal{J}(M_{\phi})$ for $s \in [-1, 1]$.
\item[$\bullet$]($\zp$-equivariance):  $J_{s,t, \m \cdot w}^-=\phi^{\m}_*J_{s, t+\m,w}^-$ for all $\m\in \zp.$ 
\item[$\bullet$] (Invariance under shift): $J_{s,t,w}^-=J_{s,t,\tau(w)}^-$ for all $w$ in $S^{\infty}.$
\end{enumerate}

Similarly if $s \geq -1$, we ask that
\begin{enumerate}
\item[$\bullet$](Prescribed on the cylindrical ends): $J_{s,t,w}^{+,i}=\tilde{J}_t^i$ for $s \geq 2$ and for all $w \in S^{\infty}.$
\item[$\bullet$] (Interpolation): $J_{s,t,w}^{+,i} = J_{s,t,w}^- \in \mathcal{J}(M_{\phi})$ for $s \in [-1, 1]$ and all $i \in \zp,\; w \in S^{\infty}$.
\item[$\bullet$]($\zp$-equivariance):  $J_{s,t, \m \cdot w}^{+,i}=J^{+,i+m}_{s, t,w}$ for all $\m\in \zp,$ $i \in \zp$ and $w \in S^{\infty}.$ 
\item[$\bullet$] (Invariance under shift): $J_{s,t,w}^{+,i}=J_{s,t,\tau(w)}^{+,i}$ for all $w$ in $S^{\infty}$ and all $i \in \zp.$
\end{enumerate}

Now, consider the symplectic fibration $E_{\cP} = h^*(\R \times M_{\phi}) \to S_{\cP}$ over $S_{\cP}$ obtained from $\R \times M_{\phi} \to \R \times S^1$ by pull-back by $h: S_{\cP} \to \R \times S^1.$ Write $\pi_{\cP}: E_{\cP} \to S_{\cP}$ for the projection map.

Let $w\colon \R \rightarrow S^{\infty}$ be a parametrized Morse flow line  of $\tilde{F} \colon S^{\infty} \rightarrow \R$ in $\sP^{i,\m}_{\alpha}$ defined in Section \ref{sec:Morse} for $\alpha=0,1$. One can then define a family $J_{z,w}$ of domain-dependent almost complex structures parametrized by $z \in {S_{\cP}}$ by setting
\begin{eqnarray}
&& J_{z,w}=h^* J^{+,i}_{s,t,w(s)}=h^* J^-_{s,t,w(s)} \text{ if } z \in \pi^{-1}([-1,1] \times S^1);\\
&& J_{z,w}=h^* J^{+,i}_{s,t,w(s)} \text{ if } z=\epsilon^+_i(s,t) \text{ for all}\; i=0,1,\cdots, p-1;\\
&& J_{z,w}=J^-_{s,t,\m \cdot w(s)} \text{ for all $\m \in \zp$ and } z=\epsilon^-_{\m}(s,t). 
\end{eqnarray}

Note that over the negative end, $E_{\cP} \to S_{\cP}$ is isomorphic to $\R \to M_{\phi,p} \to \R/p\Z.$ Also note that thanks to the Interpolation property, we can consider $J^-_{s,t,w(s)}$ for $s \in [-1,1]$ to be a complex structure in $\cJ(M_{\phi}),$ whence $h^*J^-_{s,t,w(s)}$ is well-defined, and moreover $J_{z,w}$ is smooth along $\{s = 1\},$ and $\{s = -1\}.$

Similarly, one chooses the domain-dependent perturbation term $Y_z \in \Omega^1(S_{\cP}, T(\R \times M_{\phi}))$ that satisfies the following conditions 
\begin{enumerate}
\item[$\bullet$](Constant on the cylindrical ends): $Y_{\epsilon_i^+(s,t)}=X_{H^{i,+}_t} \otimes dt$ and $Y_{\epsilon_{\m}^-(s,t)}=X_{H_{t+\m}^-} \otimes dt$ for our $H^{i,+}_t \in \cH(M_{\phi})$ and $H_t^- \in \cH(M_{\phi,p})$.
\item[$\bullet$](Compactly supported near the ends): $Y_z \equiv 0$ outside the images of the cylindrical parametrizations $\epsilon_i^+$ and $\epsilon^-_j$ for $i,j \in \zp.$ In fact we may assume that $Y_z \equiv 0$ on the image of $\eps^-_j.$
\item[$\bullet$]($\zp$-invariance): $Y_{m \cdot z} =Y_z$ for all $m \in \zp,$ acting as the covering transformation of $h.$
\end{enumerate}
For any Morse flow line $w \colon \R \rightarrow S^{\infty}$ that is asymptotic to $Z_i^{\m}$ at $-\infty$ and $Z_{\alpha}^0$ at $\infty$ for $\alpha=0,1$, one can choose $J_{z,w}$ and $Y_z$ as above. Then we consider the moduli space $\cM^{i,\m}_{\cP,{\alpha}}(x^-,x_0^+, \cdots, x_{p-1}^+)$ of solutions $(w,u)$ consisting of $w \in \sP^{i,m}_{\alpha}$ and a section $u \colon S_{\cP}\rightarrow E_{\cP},$ $\pi_{\cP} \circ u = \pi_{\cP},$ to the parametrized Cauchy-Riemann equation \begin{align}  \label{eqn:Jholpower}
(du-Y_z) \circ j &=J_{z,w}\circ (du-Y_z)
\end{align}
satisfying the asymptotic conditions\begin{equation}
\displaystyle\lim_{s \rightarrow - \infty}u(\epsilon^-_0(s,t))=x^-(t), \ \  \displaystyle\lim_{s \rightarrow - \infty}u(\epsilon^+_k(s,t))=x^+_k(t), \text{ for } k \in \zp.
\end{equation}
We remark that in this case $\displaystyle\lim_{s \rightarrow - \infty}u(\epsilon^-_m(s,t))= \lim_{s \rightarrow - \infty}u(m \cdot \epsilon^-_0(s,t)) = \phi^{m}(x^-(t+m)).$

\begin{rmk}\label{rmk: only M}
One can equivalently formulate the above equation without mentioning symplectic fibrations, as in \cite{Seidel}. We found the above description more easily accessible, but we sketch the other definition now. Consider the $\Z$-cover $\pi_Z: \R \times \R \to \R \times S^1,$ and let $\til{S}_{\cP} = (\pi_Z)^*(S_{\cP})$ be the fiber product of $\pi_Z:\R \times \R \to \R \times S^1$ and $h: S_{\cP} \to \R \times S^1$ along $\R \times S^1.$ Then $\til{S}_{\cP} \to S_{\cP}$ is a $\Z$-cover, with deck transformations generated by $\theta: \til{S}_{\cP} \to \til{S}_{\cP}$ corresponding to $\theta_{Z}: \R \times \R \to \R \times \R,$ $(s,t) \mapsto (s,t+1).$ The curve $\til{S}_{\cP}$ has $p$ {\em disjoint} negative ends $\til{\epsilon}_{-}^m: (-\infty,-1] \times \R \to \til{S}_{\cP},$ and $p$ disjoint positive ends $\til{\epsilon}_{+}^m: [1,\infty) \times \R \to \til{S}_{\cP}$ for $\m \in \zp$ on which $\theta$ acts as follows: $\theta(\til{\epsilon}_{+}^m(s,t)) = \til{\epsilon}_{+}^m(s,t+1),$  $\theta(\til{\epsilon}_{-}^m(s,t)) = \til{\epsilon}_{-}^{m-1}(s,t+1).$ Considering the almost complex structure $J'_{z,w} = J_{\pi(z),w}$ as an almost complex structure on $M$ with $\phi(J'_{\theta(z),w}) = J'_{z,w},$ we write the equation on $u: \til{S}_{\cP} \to M,$ $\phi(u(\theta(z))) = u(z),$ \[(du-Y_z) \circ j =J'_{z,w}\circ (du-Y_z)\] for a Hamiltonian perturbation described above, with asymptotic conditions \[\displaystyle\lim_{s \rightarrow - \infty}u(\til{\epsilon}^-_0(s,t))=x^-, \ \  \displaystyle\lim_{s \rightarrow - \infty}u(\til{\epsilon}^+_k(s,t))=x^+_k,\] for $k \in \zp.$ It is not hard to see that the two definitions are equivalent. The second definition has the advantage of working directly inside $M.$
\end{rmk}

As noticed in \cite[Section 3c]{Seidel}, for constant solutions of \eqref{eqn:Jholpower} one cannot achieve transversality purely by varying the almost complex structures $J_{z,w}.$ This is the reason that we also introduce the Hamiltonian perturbation term $Y_z$ in \eqref{eqn:Jholpower}. For generic choice of $J_{z,w}$ and $Y_z$, the moduli space $\cM^{i,\m}_{\cP,{\alpha}}(x^-; x_0^+, \cdots, x_{p-1}^+) $ can be shown to be a smooth finite dimensional manifold of dimension 
\begin{equation}
\mathrm{dim} \cM^{i,\m}_{\cP,{\alpha}}(x^-;x_0^+, \cdots, x_{p-1}^+)= | x^-|-\sum_k |x^{+}_k|+i-\alpha.
\end{equation}
For each $\m \in \zp$ and $| x^-|=\sum_k |x^-_k|-i+\alpha$, one can define an operation by 
\begin{eqnarray}
&& \cP^{i,\m}_{\alpha} \colon CF^*(\phi)^{\otimes p} \rightarrow CF^{*-i+\alpha}(\phi^p) \\
&& \cP^{i,\m}_{\alpha}(x^+_0, \cdots, x^+_{p-1})=\sum_{x^-} \#\cM^{i,\m}_{\cP,\alpha}(x^-;x_0^+, \cdots, x_{p-1}^+)\,x^-,
\end{eqnarray}
taking signs into account as in Appendix \ref{app:signs and or}. As before if we set $\cP^i_{\alpha}=\sum_{\m \in \zp}\cP^{i,\m}_{\alpha}$, then the $\zp$-equivariant product can be written as
\begin{eqnarray}
&& \cP \colon CF^*(\phi)^{\otimes p}[[u]] \langle \theta \rangle \rightarrow CF^*(\phi ^p)[[u]] \langle \theta \rangle\\
&& \cP(-\otimes 1)=(\cP_0^0 + u \cP^2_0 + \ldots ) \otimes 1+ (\cP^1_0 + u \cP^{3}_0 + \ldots ) \otimes \theta,\\
&& \cP(-\otimes \theta)=(u \cP^2_1 + u^2 \cP^4_1 + \ldots ) \otimes 1+ (\cP^1_1 + u \cP^{3}_1 + \ldots ) \otimes \theta.
\end{eqnarray}

For a generator $X = x_0 \otimes \ldots \otimes x_{p-1}$ of $CF^*(\phi)^{\otimes p},$ we abbreviate this to \[ \cl{P}(X \otimes 1) = P^{1}_{1}(X) \otimes 1 + P^{\theta}_{1}(X) \otimes \theta\] \[ \cP(X \otimes \theta) = P^{1}_{\theta}(X) \otimes 1 + P^{\theta}_{\theta}(X) \otimes \theta.\]   

\subsection{The coproduct}

The $\zp$-equivariant coproduct is analogous to the $\zp$-equivariant product that we defined above, with the difference that it now has $1$ input and $p$ outputs. Intuitively speaking it relates to the $\zp$-equivariant product in the same way as the inverse PSS isomorphism relates to the PSS isomoprhism.

To define the $\zp$-equivariant coproduct, we first define a Riemann surface with one positive puncture and $p$ negative punctures by the following fiber diagram
\[
\xymatrix@C+2pc{
{S_{\cC}} \ar[d]^{\tilde{h}} \ar[r] & \C \backslash \{e^{\frac{2\pi i\m}{p}} \} \ar[d]^{y=1-z^p} \\
 \R \times S^1 \ar[r]^-{\psi^{+}} & \C^*,}
\] where for $(s,t) \in \R \times S^1,$ $\psi^{+}$ is given by \[ \psi^{+}(s,t) = e^{2\pi(s+ i t)} .\]

Similarly to the case of $\zp$-equivariant product, one looks at the symplectic fibration $\pi_{\cC}: E_{\cC} = \til{h}^*(\R \times M_{\phi}) \to S_{\cC},$ and one chooses $p$ positive and $p$ negative cylindrical trivializations as $\tilde{\epsilon}_i^+$ and $\tilde{\epsilon}_j^-$ for $i, j=0,1, \cdots, p-1.$  

\begin{figure}[htb]
	\centerline{\includegraphics{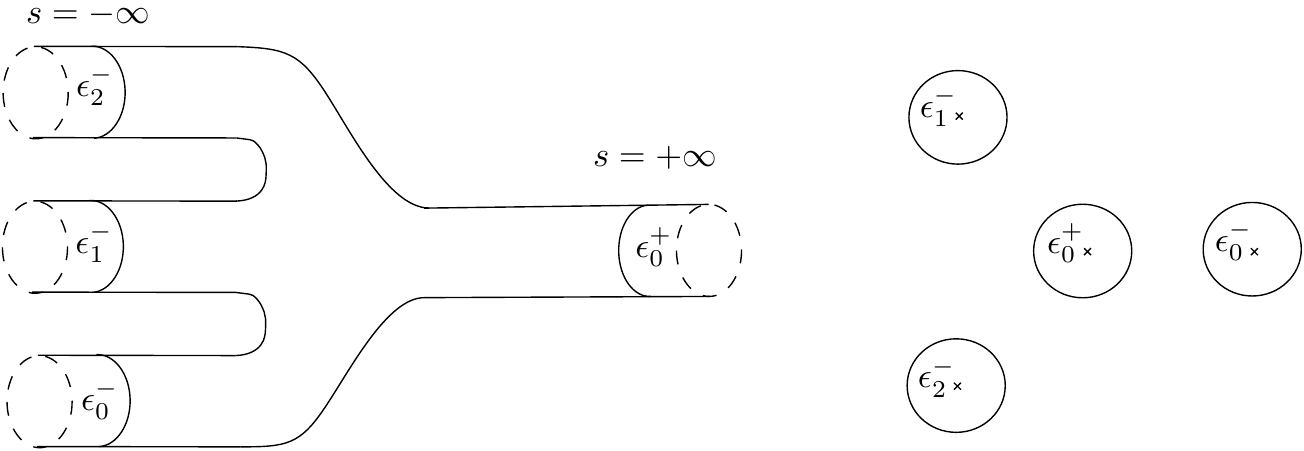}}
	\caption{The curve $S_{\cl{C}}$ and its cylindrical ends for $p=3.$ Left: schematic picture emphasizing the curve structure. Right: schematic picture emphasizing the $\Z/3\Z$ symmetry; this is a $\Z/3\Z$-symmetric configuration of $4$ punctures and cylindrical ends, $1$ positive and $3$ negative, on $\C P^1 \cong \C \cup \{ \infty \}.$ }\label{Figure: coprod curve}
\end{figure}

Then one choose similar families of Floer data $(J_{\tilde{z},w}, Y_{\tilde{z}})$ parametrized by ${S_{\cC}} \times S^{\infty}$ and $S_{\cC}$ separately with the direction of the $s$-parameter reversed relative to the product case.  Finally, one considers the moduli space $\cM^{i,\m}_{\cC,{\alpha}}(x_0^-, \cdots, x_{p-1}^-; x^+)$ of solutions $w \in \sP^{i,m}_{\alpha},$ $u \colon {S_{\cC}} \rightarrow E_{\cC},$ $\pi_{\cC} \circ u = \pi_{\cC},$ to the parametrized perturbed Cauchy-Riemann equation
\begin{equation} \label{eqn:coprod}
(du-Y_{\tilde{z}})\circ j =J_{\tilde{z},w} \circ (du-Y_{\tilde{z}}),
\end{equation}
with asymptotic behavior
\begin{equation}
\displaystyle\lim_{s \rightarrow - \infty} u(\tilde{\epsilon}^-_i(s,t))=x^-_i(t), \ \ \displaystyle\lim_{s \rightarrow - \infty} u(\tilde{\epsilon}^+_0(s,t))=x^+(t).
\end{equation}
Again, a definition as in Remark \ref{rmk: only M} working entirely in $M$ is also available.

For generic choice of $J_z^w$ and $H^w$, the moduli space of non-constant solutions $(u,w)$ to \eqref{eqn:coprod} $\cM^{i,\m}_{\cC,{\alpha}}(x_0^-, \cdots, x_{p-1}^-;x^+)$ is a smooth finite dimensional manifold of dimension
\begin{equation}
\sum_k |x^-_k|-|x^+|+i-\alpha-2n(p-1)
\end{equation}
For each $\m \in \zp$ and $\sum_{k=0}^{p-1} |x^-_k|=| x^+|-i+\alpha + 2n(p-1)$, one can define an operation by 
\begin{eqnarray}
&& \cC^{i,\m}_{\alpha}\colon C^*(\zp; CF^*(\phi^p)) \rightarrow C^{*-i+\alpha+2n(p-1)}(\zp; CF^*(\phi)^{\otimes p}) \\
&& \cC^{i,\m}_{\alpha}(x^+)=\sum_{x^-} \#\cM^{i,\m}_{\cC,{\alpha}}(x_0^-, \cdots, x_{p-1}^-;x^+)(x^-_0 \otimes \cdots \otimes x^-_{p-1}).
\end{eqnarray}
As before if we set $\cC^i_{\alpha}=\sum_{\m} \cC^{i,\m}_{\alpha},$ then the $\zp$-equivariant coproduct is given by
\begin{eqnarray}
&& \cC \colon CF^*(\phi^p)[[u]] \langle \theta \rangle \rightarrow CF^*(\phi)^{\otimes p}[[u]] \langle \theta \rangle\\
&& \cC(-\otimes 1)=(\cC_0^0 + u \cC^2_0 + \ldots ) \otimes 1+ (\cC^1_0 + u \cC^{3}_0 + \ldots ) \otimes \theta,\\
&& \cC(-\otimes \theta)=(u \cC^2_1 + u^2 \cC^4_1 + \ldots ) \otimes 1+ (\cC^1_1 + u \cC^{3}_1 + \ldots ) \otimes \theta.
\end{eqnarray}

\subsection{Chain-map property}\label{subsec: chain map}

Our next goal is to show that the $\zp$-equivariant product and coproduct maps $\cP$ and $\cC$ define chain maps. Both cases are treated similarly, so we focus on \[\cP:C^*(\zp, CF^*(\phi)^{\otimes p}) \to CF^*_{\zp}(\phi^p),\] where the complexes are taken with their respective differentials. The chain map relation follows by standard Gromov-Floer compactness, transversality, and gluing arguments, from looking at compactifications of the $1$-dimensional moduli spaces \[\cM^{i,\m}_{\cP,{\alpha}}(x^-;x_0^+, \cdots, x_{p-1}^+),\] coming from either Floer breaking in the interior, or  from codimension $1$ strata of $\sP^{i,\m}_{\alpha}$ at the boundaries (see \cite[Section 4.3]{Seidel} for a discussion of the analytical issues). 

Specifically, we show that for the differential \[d^{\otimes p, \zp}_{\phi} = d: C^*(\zp, CF^*(\phi)^{\otimes p}) \to C^*(\zp, CF^*(\phi)^{\otimes p})\] from \eqref{eq: equivariant complexes 2} and the differential \[d^{\zp}_{\phi^p}: CF^*_{\zp}(\phi^p) \to CF^*_{\zp}(\phi^p)\] from \eqref{eq: equivariant differential} we have the relation: \begin{equation}\label{eqn: chain map relation}
\cl{P} \circ d^{\otimes p, \zp}_{\phi} = d^{\zp}_{\phi^p} \circ \cl{P}.
\end{equation}

For $X = x_0 \otimes \ldots x_{p-1}$ an arbitrary generator of $CF^*(\phi)^{\otimes p},$ this is equivalent to the following four relations. The first two come from evaluating on $X \otimes 1,$ and the second two come from evaluating on $X \otimes \theta:$


\[ P^1_1(d_{\phi}^{\otimes p}(X)) + P^1_\theta((1-\sigma)(X)) = D^1_1\circ P^1_1 (X) + D^1_{\theta} \circ P^{\theta}_1(X)\]
\[ P^\theta_1(d_{\phi}^{\otimes p}(X)) + P^\theta_\theta((1-\sigma)(X)) = D^\theta_1\circ P^1_1 (X) + D^\theta_{\theta} \circ P^{\theta}_1(X)\]
\[ - P^1_\theta(d_{\phi}^{\otimes p}(X)) + u P^1_1(N(X)) = D^1_1\circ P^1_\theta (X) + D^1_{\theta} \circ P^{\theta}_\theta(X)\]
\[ - P^\theta_\theta(d_{\phi}^{\otimes p}(X)) + u P^\theta_1(N(X)) = D^\theta_1\circ P^1_\theta (X) + D^{\theta}_{\theta} \circ P^{\theta}_\theta(X)\]


Recall that the compactification $\ol{\sP}^{i,m}_{\alpha}$ of the space of parametrized gradient trajectories $\sP^{i,m}_{\alpha}$ has the following codimension $1$ strata: \[\sQ^{i_1,m_1}_{\alpha} \times \sP^{i_2,m_2}_{\alpha_2},\;\;\sP^{i_1,m_1}_{\alpha} \times \sQ^{i_2,m_2}_{\alpha_2},\] where $m_1 + m_2 = m$ in $\zp,$ $i_2 = i \;(\mrm{ mod } \ 2)$ and $i_1 + i_2 - \alpha_2 = i.$ Fixing $i$ and $\alpha,$ the relations above are obtained from the behavior of the compactification $\ol{\cM}^{i,\m}_{\cP,{\alpha}}(x^-;x_0^+, \cdots, x_{p-1}^+)$ of the corresponding $1$-dimensional moduli space over these strata in $\ol{\sP}^{i,\m}_{\alpha},$ for $x_0^+, \cdots, x_{p-1}^+ =X.$ The identities are obtained for pairs $(i,\alpha)$ for which $([i],\alpha) \in \Z/2\Z \times \Z/2\Z$ is $(0,0),$ $(1,0),$ $(0,1),$ $(1,1)$ respectively (in the order of appearance).

Let us explain the first case, for example. Fix $i =2k,$ $\alpha = 0.$ Then $i_2 = i \;(\mrm{ mod } \ 2),$  $i_1 = \alpha_2 \;(\mrm{ mod } \ 2),$ and $i_1 + i_2 - \alpha_2 = i.$ Then the solutions of the limiting Floer equation over strata $\sP^{i_1,m_1}_{0} \times \sQ^{i_2,m_2}_{\alpha_2}$ lead (after considering the various $i=2k,$ and summing over all suitable $m, m_1, m_2$) to the positive order components of the term $D^1_1 \circ P^1_1(X)$ for $\alpha_2 = 0,$ and to the term $D^1_{\theta} \circ P^{\theta}_1(X)$ for $\alpha_2 = 1.$ The strata $\sQ^{i_1,m_1}_{0} \times \sP^{i_2,m_2}_{\alpha_2}$ can contribute non-trivially only when $i_1 = 1,$ and hence $\alpha_2 = 1,$ since otherwise they do not give isolated solutions (see \cite[(4.114)]{Seidel}). This leads to the second term on the left hand side. The first term on the left hand side (as well as the zero-th order component $d_{\phi,p} \circ P^1_1(X)$ in the first term on the right hand side) is obtained from usual Floer breaking in the codimension $0$ strata.

\section{Local Floer cohomology and the action filtration}\label{sec:local-FH}

We describe the local Floer cohomology at an isolated fixed point of a symplectomorphism and its $\zp$-equivariant version. Furthermore, we discuss the action spectral sequence starting with the direct sum of local Floer cohomologies and converging to the total Floer cohomology. This spectral sequence shall subsequently be used to prove that the $\zp$-equivariant pants product is a {\em filtered} chain-homotopy equivalence between the two relevant filtered complexes, and in the proofs of our main applications. We refer to \cite{Ginzburg-CC, GG-local-gap} for more details on local Floer cohomology. 

\subsection{Local Floer cohomology}\label{subsec: loc FH}


Let $\phi$ be a symplectomorphism of a symplectic manifold $M$ as specified in Section \ref{sec:Floer_coho}. Given an isolated fixed point $x$ of $\phi,$ there exists an isolating neighborhood $U$ of $x$ (more precisely, of the image of the flat section $\sigma_x$ in $M_{\phi}$) for Floer cohomology. In particular, all Floer trajectories of each sufficiently $C^2$ small non-degenerate Hamiltonian perturbation $\phi'$ of $\phi$ between generators in $U$ are contained in $U,$ and the resulting Floer cohomology as computed inside $U$ is well-defined and independent of the perturbation. This cohomology is called the local Floer cohomology $HF^{\loc}(\phi,x)$ of $\phi$ at $x.$ Whenever the local Floer cohomology is considered as an {\em ungraded} $\bb K$-module it depends on no additional data. A similar statement and definition applies to an isolated Morse-Bott submanifold $X$ of fixed points $\phi$ (see e.g. \cite{McLean-geodesics, Pozniak, Floer-MorseWitten, Floer-spheres}).


We recall the following additional properties of $HF^{\loc}(\phi,x).$ First, if $x$ is non-degenerate as a fixed point of $\phi,$ then as $\bb K$-modules, \[HF^{\loc}(\phi,x) \cong \bb K.\] Second, let $c \in \spec(\ul \phi)$ be an isolated action value, such that all $x \in \mathrm{Fix}(\phi)$ with $\cA_{\phi}(x) = c$ are isolated. In view of Section \ref{subsec: estimates}, for two distinct fixed points $x,y \in \fix(\phi),$ there exists $\epsilon_0 > 0,$ such that all Floer trajectories, or product structures considered in this paper, with $x,y$ among their asymptotics, carry energy of at least $\epsilon_0.$ Hence for $\epsilon > 0$ sufficiently small, \[HF(\ul \phi)^{(c-\eps, c+\eps)} \cong \bigoplus_{\substack {x \in \fix(\phi),\\ \cl A_{\phi} (x) = c}} HF^{\loc}(\phi, x) .\] Finally, the above ``building block" property implies that if all $x \in \fix(\phi)$ are isolated, and $M$ is aspherical or exact, then for each $a,b \in (\R \setminus \spec(\ul \phi)) \cup \{\pm \infty\},$ $a < b,$  there is a spectral sequence arising from the action filtration, that converges to $HF(\ul \phi)^{(a, b)}$ and has $E_1$-page given by  \[ \bigoplus_{\substack {x \in \fix(\phi),\\ a< \cA_{\phi}(x) < b}} HF^{\loc}(\phi, x),\] filtered by $\cl A_{\phi}.$

\subsection{Local equivariant Floer cohomology}\label{subsec: Local equiv}


The above situation readily extends to the case of equivariant Floer cohomology. Indeed, supposing that all fixed points $\fix(\phi^p)$ of $\phi^p$ are isolated (or more generally belong to isolated connected Morse-Bott submanifolds), there is an upper and a lower bound depending only on $\phi$ and $\dim M$ on the possible indices of the fixed points of a sufficiently $C^2$-small non-degenerate Hamiltonian perturbation of $\phi^p.$  We choose this perturbation to be of the form $\phi_1^p,$ where $\phi_1$ is a sufficiently $C^2$-small non-degenerate perturbation of $\phi.$ In particular, the terms $d^i_{\alpha}$ of the equivariant differential vanish for all $i > i_0(\phi),$ independently of the choice of perturbation data. Therefore, the equivariant differential depends only on the perturbation data $J_{t,w}, H_t$ in a compact family corresponding to $w \in S^{i_0(\phi)}.$ Hence, Proposition \ref{prop: monotonicity} ensures that the perturbation data can be chosen in such a way that the trajectories of the equivariant differential between generators inside a sufficiently small isolating neighborhood $U_x$ of $x \in \fix(\phi^p)$ (again, more precisely of the image of the flat section $\sigma_x$ in $M_{\phi^p}$) stay inside $U_x.$ Furthermore, the same is true for neighborhoods $U_{\phi^{m}x}$ of $\phi^{m} x$ for $m \in \zp.$ Therefore, by definition of the equivariant differential, gluing, and compactness, the critical points of $\phi_1^p$ in $U = \bigcup_{m \in \zp} U_{\phi^{m}x}$ form a complex, and the cohomology of this complex is independent of the Hamiltonian perturbation $\phi_1^p$ of $\phi^p.$ We call this cohomology the equivariant local Floer cohomology $HF^{\loc}_{\zp}(\phi^p,\zp\, x)$ of the orbit $\zp \, x.$ In the special case when $x$ is an iterated fixed point, that is $\phi(x) = x,$ or $\zp\, x = \{x\},$ then $\phi^m (x) = x$ for all $m\in \zp$ and only one isolating neighborhood $U$ of $x$ with respect to $\phi^p$ is necessary, and we abbreviate $HF^{\loc}_{\zp}(\phi^p,\zp\, x)$ to $HF^{\loc}_{\zp}(\phi^p,x^{(p)}).$ We remark that in the case of simple $p$-periodic points, where the orbit $\zp\, x$ has $p$ elements, all the points $\phi^m(x)$ for $m \in \zp$ are distinct, hence the flat sections $\sigma_{\phi^m(x)}$ for $m \in \zp$ have disjoint images in $M_{\phi^p},$ and the isolating neighborhoods $U_{\phi^{m}x}$ can, and should, be chosen to be disjoint. 

The local equivariant Floer cohomology enjoys properties similar to those of usual local Floer cohomology. First if $x$ is non-degenerate as a fixed point of $\phi^p,$ then if $x$ is iterated, we have \[HF^{\loc}_{\zp}(\phi^p,x^{(p)}) \cong H^*(\zp,\bK) = \rp\] as $\rp$-modules, and if $x$ is simple, then $\zp$-action on $\zp\, x$ is free and transitive, and \[HF^{\loc}_{\zp}(\phi^p,\zp\,x) \cong H^*(\zp,\bK[\zp]) = \bK = \rp/\left< u, \theta \right>.\]

Second, let $c \in \spec(\ul \phi^p)$ be an isolated action value, such that all $x \in \mathrm{Fix}(\phi^p)$ with $\cA_{\phi^p}(x) = c$ are isolated. In view of Section \ref{subsec: estimates}, we again obtain that for $\epsilon > 0$ sufficiently small, \[HF^*_{\zp}(\ul \phi^p)^{(c-\eps, c+\eps)} \cong \bigoplus HF^{\loc}_{\zp}(\phi^p, \zp\, x),\] the sum running over the orbits of the $\zp$-action on $\{x \in \fix(\phi^p)\,|\,\cl A_{\phi^p} (x) = c\}.$ 

Furthermore, if all $x \in \fix(\phi^p)$ are isolated, and $M$ is aspherical or exact, then for each $a,b \in (\R \setminus \spec(\ul \phi)) \cup \{\pm \infty\},$ $a < b,$  there is a spectral sequence arising from the action filtration that converges to $HF_{\zp}^*(\ul \phi^p)^{(a, b)}$ and has $E_1$-page given by  \[ \bigoplus_{\substack O \in \{x \in \fix(\phi^p),\\ a< \cA_{\phi}(x) < b\} /\zp} HF^{\loc}_{\zp}(\phi, O),\] filtered by $\cl A_{\phi^p}.$

Finally, tensoring with $\cK = \bK((u))$ over $\bK[[u]]$ everywhere we obtain a similar spectral sequence for the Tate cohomology groups with $E_1$-page given in terms of the local Tate cohomology groups.

\subsection{Local product and coproduct operations}\label{subsec: local prod coprod}


Consider an isolated fixed point $x \in \fix(\phi),$ with isolating neighborhood $U$ of $\sigma_x$ in $M_{\phi}$ that extends to a neighborhood $U_{\cP}$ of $\sigma_x$ in $E_{\cP}$ and a neighborhood $U_{\cC}$ of $\sigma_x$ in $E_{\cC}.$ 

By Proposition \ref{prop: monotonicity}, and the argumentation of Section \ref{subsec: Local equiv}, for a sufficiently $C^2$-small Hamiltonian perturbation of $\phi,$ the moduli spaces defining the product and coproduct operations with all inputs and outputs restricted to lie inside $U$ involve only sections that lie inside $U_{\cP}$ and $U_{\cC}$ respectively. Furthermore, they define chain maps on the suitable local cohomology groups, and hence operations 

\[\cP^{\loc}_{x}: H^*(\zp; CF^{\loc}(\phi,x)^{\otimes p}) \to HF^{\loc}_{\zp}(\phi^p,x^\ip),\] \[\cC^{\loc}_{x}: HF^{\loc}_{\zp}(\phi^p,x^\ip) \to H^*(\zp; CF^{\loc}(\phi,x)^{\otimes p}).\]

Finally, choosing sufficiently small isolating neighborhoods, it is straightforward to deduce that the local Floer cohomology, its
equivariant version at an iterated fixed point, and the local products and coproducts $\cP^{\loc}_x, \cC^{\loc}_x$ depend only on the germ of $\phi$ at $x.$

\section{The local coproduct-product is invertible}\label{sec:locally invertible}


We first prove the assertion of Theorem \ref{thm: main} in the case of a non-degenerate fixed point $x$ of $\phi.$ The case of local Floer cohomology, as well as the general symplectically aspherical and exact cases, will follow directly by a spectral sequence argument. The argument in this section is the main technical novelty of the paper, allowing us to extend results of \cite{Seidel} to primes $p>2.$


%

Fix for the duration of this section a non-degenerate fixed point $x \in \fix(\phi)$ and isolating neighborhoods $U$ of $x$ for Floer cohomology and $U^\ip$ of $x^\ip$ for $\zp$-equivariant Floer cohomology. Recall that \[HF^{\loc}(\phi,x) \cong \bK,\] 
\[HF^*(\zp; CF^{\loc}(\phi,x)^{\otimes p}) \cong \rp,\] \[HF^{*,\loc}_{\zp}(\phi^p,x^\ip) \cong \rp .\] 

Consider the local product and coproduct operators 

\[\cP^{\loc}_{x}: HF^*(\zp; CF^{\loc}(\phi,x)^{\otimes p}) \to HF^{*,\loc}_{\zp}(\phi^p,x^\ip),\] \[\cC^{\loc}_{x}: HF^{*,\loc}_{\zp}(\phi^p,x^\ip) \to HF^*(\zp; CF^{\loc}(\phi,x)^{\otimes p}).\]\\


In the subsections below we show that \begin{equation} \label{eq:key local identity} \cC^{\loc}_x \circ \cP^{\loc}_x = (-1)^n u^{(p-1)n} \cdot \id,\end{equation} and hence it becomes invertible after tensoring with $\cK = \F_p((u)).$ Indeed, $u^{(p-1) n}$ is a unit in $\cK.$ Recall that $\hrp = \F_p((u))\langle \theta \rangle.$ Since $\dim_{\cK} \hrp < \infty,$ this implies that $\cP^{\loc}_x$ becomes invertible after extending coefficients to $\cK.$

\subsection{Invariance properties} \label{subsec:invar}


First, using invariance properties of $\cC^{\loc}_x \circ \cP^{\loc}_x$ under isolated deformations, we show that \begin{equation}\label{eq:weak key local identity}\cC^{\loc}_x \circ \cP^{\loc}_x = c_n \cdot u^{(p-1)n} \cdot \id,\end{equation} for a constant $c_n \in \bK.$  In Section \ref{subsec:Betz-Cohen} we calculate that $c_n = (-1)^n,$ using a reduction to the Morse-theoretic model of Betz-Cohen \cite{BetzCohen} type.

We follow the arguments of Seidel \cite[Section 6]{Seidel}, combined with the additional flexibility provided by an alternative interpretation of $\cC^{\loc}_x \circ \cP^{\loc}_x$ as an operation \[\cl Z^{\loc}_{x}: HF^*(\zp; CF^{\loc}(\phi,x)^{\otimes p}) \to HF^*(\zp; CF^{\loc}(\phi,x)^{\otimes p})\] obtained by counting $p$-tuples of Floer cylinders with a diagonal-type incidence constraint. Intuitively, one should think of the cup product with a suitable equivariant diagonal class. This corresponds to requiring our Floer cylinders to be incident when evaluated at a marked point in each domain curve. Technically speaking, this operation allows us to show a more general isolated-deformation invariance than that of $\cP_x^{\loc}$: indeed, chambers in the linear symplectic group defined by excluding $p$-th roots of unity as eigenvalues, as in \eqref{eq: chambers}, play no role for this new map. In turn, this is useful for reducing the question to Morse theory in the setting of local Floer cohomology (recall that there is in general no inverse pair of PSS isomorphisms in this setting, which presents an additional technical difficulty). 

\medskip


\begin{df}\label{def:Z local}
	The chain-level operation \[\cl Z^{\loc}_{x}: CF^*(\zp; CF^{\loc}(\phi,x)^{\otimes p}) \to CF^*(\zp; CF^{\loc}(\phi,x)^{\otimes p})\] is defined as follows. Consider $p$ cylinders $C^\m = \R \times S^1,$ $\m \in \zp.$ Choose $p$ almost complex structures $J^{\m}_{s,t,w} \in \cJ_{\phi}$ each depending on the point $(s,t) \in C^\m,$ and $w \in S^{\infty}.$ We require that 
	\begin{eqnarray}
	&& J^{\m}_{s,t,\tau(w)} = J^{\m}_{s,t,w}, \;\text{for all}\; w \in S^{\infty}\\
	&& J^{\m,\pm}_{s,t,k\cdot w} = J^{\m+k,\pm}_{s,t,w}, \;\text{for all}\; k\in \zp \\
	&& J^{\m}_{s,t,w} = J_t,\;\text{for}\; |s|\geq 2 
	\end{eqnarray}
	
	
	For a Morse flow-line $w\colon \R \rightarrow S^{\infty}$ in $\sP^{i,\m}_{\alpha},$ for $\alpha \in \{0,1\},$ as defined in Section \ref{sec:Morse}, we set $J^{\m,w}_{s,t} = J^{\m}_{s,t,w(s)}.$
	
	
	
	
	Now as in the definitions in Section \ref{sec:prod-coprod}, introducing $C^2$-small Hamiltonian perturbations $H^{\m}_{s,t}$ with corresponding perturbation form $Y_{\m} = Y^{\m}_{s,t} \otimes dt$ compactly supported away from $(0,0) \in C^{\m},$ we look at the moduli spaces $\cM^{i,\m}_{\cl Z,\alpha}(x_0^{-},\ldots,x_{p-1}^{-};x_0^{+},\ldots,x_{p-1}^{+})$ of solutions $(w,(u_{\m})_{\m \in \zp}),$ to the following parametric Floer equation. Let $\til{C}^{\m} = \R \times \R$ be the universal cover of ${C}^{\m},$ and $\lambda(s,t) = (s,t+1)$ the deck transformation. Identifying $\zp = \{0,1,\ldots,p-1\}$ as a set, $u_{\m}: \til{C}^{\m} \to M$ and $w \in \sP^{\m,i}_{\alpha}$ satisfy \begin{equation} 
	\begin{cases}
	(du_{\m} - Y_{\m})^{(0,1)}_{J^{\m,w}_{s,t}} = 0\\
	u_\m(z)=\phi(u_\m(\lambda(z)))\\
	u_0(0,0) = u_1(0,0) = \ldots = u_{p-1}(0,0). 
	\end{cases}
	\end{equation}
	with asymptotic conditions 
	\begin{equation}
	\displaystyle\lim_{s \rightarrow - \infty}u_k(s,t)=x^-_k(t), \ \  \displaystyle\lim_{s \rightarrow \infty}u_k(s,t)=x^+_k(t)
	\end{equation}
where $\phi(x^{\pm}_k(t+1)) = x^{\pm}_k(t) \equiv x^{\pm}_k$ are the suitable fixed points, considered as twisted loops.
\end{df}

\begin{figure}[htb]
	\centerline{\includegraphics{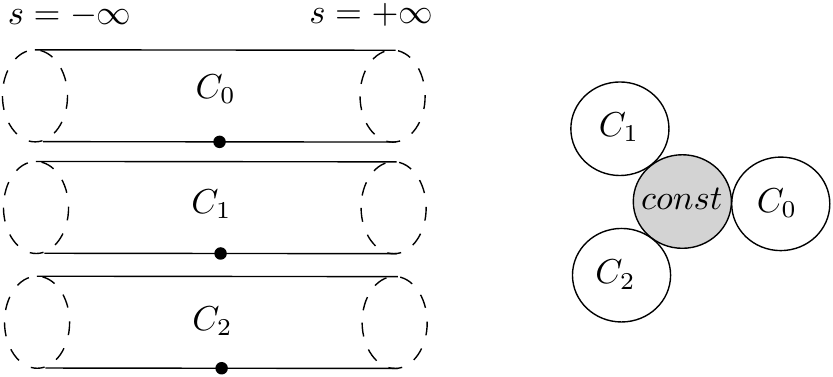}}
	\caption{Left: three Floer cylinders with incidence constraint at a marked point. Right: diagram manifesting the $\Z/3\Z$ symmetry, where the cylinders are viewed ``from the side", and the incidence constraint is viewed as a constant disk in the middle.}\label{Figure: Z map}
\end{figure}

\medskip

\begin{rmk}
It is not difficult to define an equivalent equation alternatively in terms of suitable symplectic fibrations. Indeed, by Section \ref{sec:Floer_coho}, we may consider each $u_{\m}$ as a section of a copy of $\R \times M_{\phi},$ with suitable perturbation and boundary data. Furthermore, the fiber over $(0,0) \in \R \times S^1$ of $M_{\phi}$ is naturally identified with $M.$ This allows us to write the necessary incidence condition. To work locally, we restrict attention to a neighborhood of the flat section $\sigma_x$ by Section \ref{subsec: estimates}.
\end{rmk}

\medskip


It is straightforward to show that transversality for the moduli spaces obtained thus can be achieved by generic choice of $J^{\m}_{s,t,z}$ and $H^{\m}_{s,t},$ the latter being $C^2$-small, making $\cM^{i,\m}_{\cl Z,\alpha}(x^0_{-},\ldots,x^{p-1}_{-};x^0_{+},\ldots,x^{p-1}_{+})$ a smooth manifold of dimension \[ \sum_k |x^-_k|- \sum_k |x^+_k|+i-\alpha-2n(p-1).\] 

Furthermore, the dimension $0$ moduli spaces, corresponding to the condition $\sum_k |x^-_k| = \sum_k |x^+_k|-i+\alpha+2n(p-1),$ are compact by a standard Gromov-Floer compactness argument, since in the local and the weakly exact cases there are no holomorphic curves present. Hence they consist of a finite number of points. 

\medskip

Then $\cl Z_x^{\loc},$ which we abbreviate as $\cl Z,$ is given by the collection of operations 
\begin{align}
&\cl Z^{i,\m}_{\alpha}\colon C^*(\zp; CF^*(\phi)^{\otimes p}) \rightarrow C^{*-i+\alpha+2n(p-1)}(\zp; CF^*(\phi)^{\otimes p}) \\
&\cl Z^{i,\m}_{\alpha}(x^+_0 \otimes \cdots \otimes x^+_{p-1})=\sum \#\cM^{i,\m}_{\cl Z,{\alpha}}(x_0^-, \cdots, x_{p-1}^-;x_0^+, \cdots, x_{p-1}^+)(x^-_0 \otimes \cdots \otimes x^-_{p-1})
\end{align}

Following the usual recipe,  if we set $\cl Z^i_{\alpha}=\sum_{\m} \cl Z^{i,\m}_{\alpha},$ then the $\zp$-equivariant $p$-cylinder map is given by
\begin{eqnarray}
&& \cl Z \colon CF^*(\phi^p)[[u]] \langle \theta \rangle \rightarrow CF^*(\phi)^{\otimes p}[[u]] \langle \theta \rangle\\
&& \cl Z(-\otimes 1)=(\cZ_0^0 + u \cZ^2_0 + \ldots ) \otimes 1+ (\cZ^1_0 + u \cZ^{3}_0 + \ldots ) \otimes \theta,\\
&& \cZ(-\otimes \theta)=(u \cZ^2_1 + u^2 \cZ^4_1 + \ldots ) \otimes 1+ (\cZ^1_1 + u \cZ^{3}_1 + \ldots ) \otimes \theta.
\end{eqnarray}

As for $\cC, \cP,$ the chain map relation for $\cZ$ follows, by Section \ref{subsec: estimates} and standard Gromov-Floer compactness, transversality, and gluing arguments, from looking at compactifications of the $1$-dimensional moduli spaces $\cM^{i,\m}_{\cZ,{\alpha}}(x_0^-, \cdots, x_{p-1}^-;x_0^+, \cdots, x_{p-1}^+).$ Furthermore the following identity holds on the chain level in the isolated non-degenerate case.


\medskip

\begin{lma}
	$\cC^{\loc}_x \circ \cP^{\loc}_x = \cl Z_x^{\loc}$
\end{lma}

\begin{proof}
	The idea behind this proof consists in a degeneration-gluing argument. However, to carry it out, we must replace the $p$ trajectories $u_{\m}: \til{C}^{\m} \to M,$ incident at $u_{\m}(0,0),\; \m \in \zp,$ by a map $u:\til{C} \to M,$ where now $C$ is the nodal curve consisting of the curves $C^{\m}, \; \m \in \zp,$ and one genus zero curve $S \cong \C P^1 \cong \C \cup \infty,$ with nodes given by identifying $(0,0) \in C^{\m}$ with $e^{2\pi \m i /p} \in S.$ Of course in the local, and the weakly exact, cases the restriction $u_{\ast}$ of the holomorphic curve $u$ to $S$ will be constant. However, this turns out to be the right domain to perform gluing. Note that $C$ admits a holomorphic $\zp$-action, given by cyclically permuting the $C^{\m}$ under the natural identification and rotating $S$ by the corresponding $p$-th roots of unity. The degenerations are schematically described in Figures \ref{Figure: degeneration 1} and \ref{Figure: degeneration 2}.
	
	\begin{figure}[htb]
		\centerline{\includegraphics{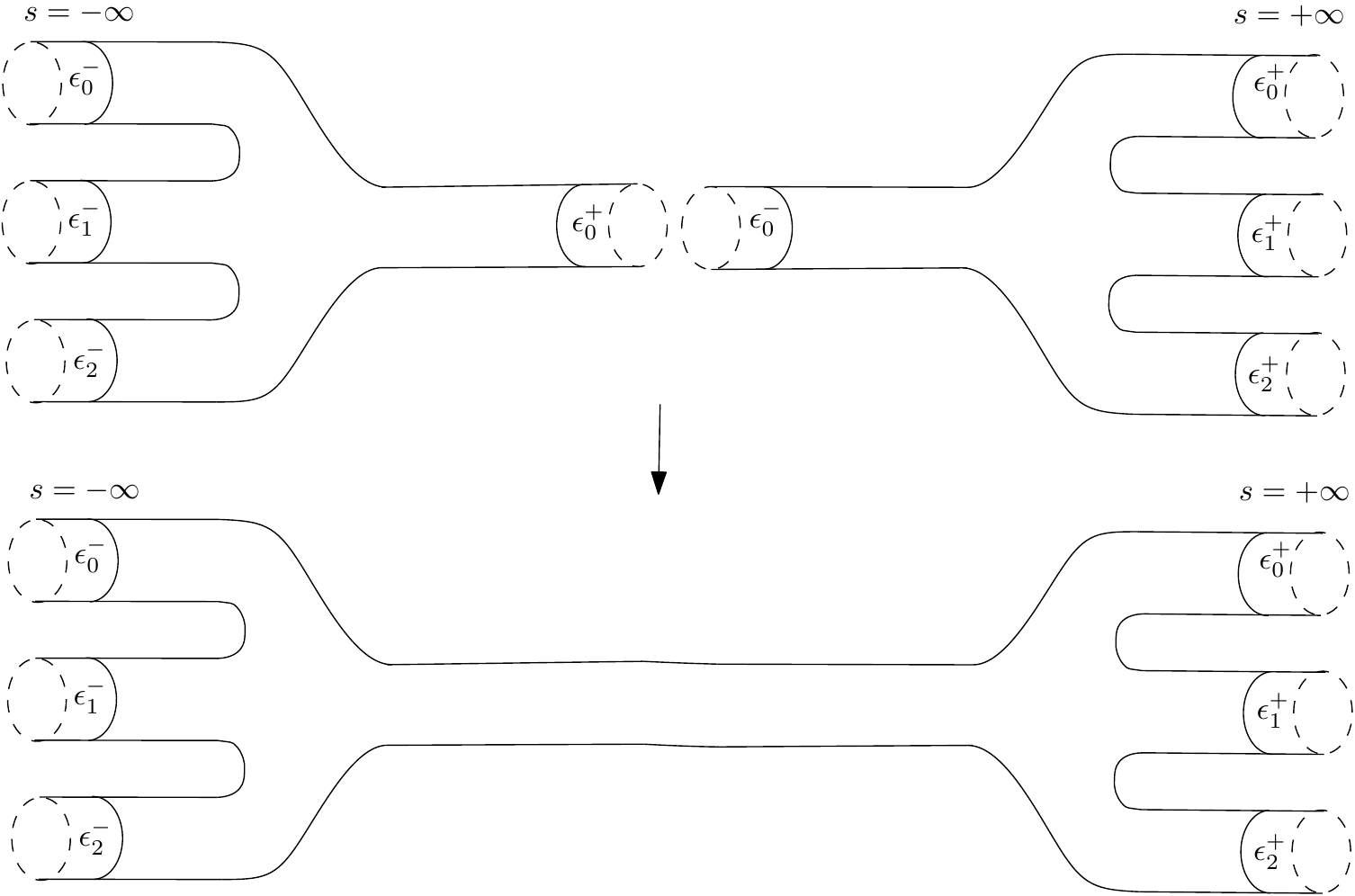}}
		\caption{Gluing of the coproduct and the product.}
			\label{Figure: degeneration 1}
			\bs
			\centerline{\includegraphics{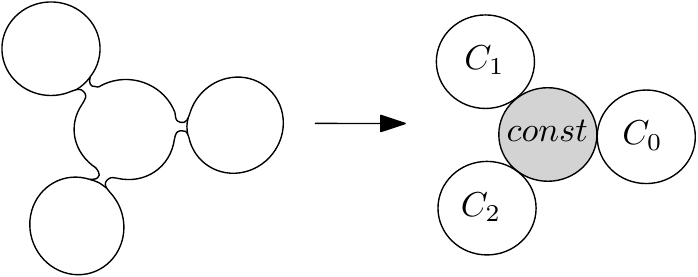}}
			\caption{Degeneration from the glued coproduct-product to $\cl{Z}_x^{\loc},$ ``side-view".}
	\label{Figure: degeneration 2}
	\end{figure}
	
	Consider the following family $\mathbf{\cR} \cong (0,1)$ of Riemann surfaces with $\zp$-action. Choose cylindrical ends $\epsilon^{\m}_{+}:[1,\infty) \times S^1 \to S$ at the points $e^{2\pi \m i /p} \in S$ that are equivariant with respect to the $\zp$-action, $\l \cdot \epsilon^{\m}_{+}(\zeta) = \epsilon^{\m+\l}_{+}(\zeta)$ for all $\zeta \in [1,\infty) \times S^1$ and $\l \in \zp,$ and $\epsilon^{\m}_{-}:(-\infty,1] \times S^1 \to C^{\m}$ that are identified under the isomorphisms $C^{\m} \cong \R \times S^1.$ Performing gluing with parameter $l_{+} \in [1,+\infty)$ we obtain one part of the family, $[1,\infty) \to \mathbf{\cR}$. Another part of the family, $(-\infty,1] \to \mathbf{\cR}$ is given by gluing $S_{\cP}$ and $S_{\cC}$ along $\epsilon_0^{-}: (-\infty,1] \times S^1 \to S_{\cP}$ and $\epsilon_0^{+}: [1,\infty) \times S^1 \to S_{\cC}$ with gluing parameter $-l_{-} \in [1,\infty),$ where $l_{-} \in (-\infty,1].$ Of course, by equivariance, this gives the same Riemann surface as gluing along $\epsilon_m^{-},\epsilon_m^{+}$ would give, for all $m \in \zp.$ Let $r$ denote the natural coordinate on $\R.$ The gluing above, after suitable reparametrization, gives a map $\{|r| \geq 1\} \to \cM_{0,p,p}^{\zp}$ to the moduli space of genus $0$ curves with $\zp$-action, with $p$ negative ordered marked points, and $p$ positive ordered marked points, each $p$-tuple being $\zp$-equivariant. Finally, for $r \in [-1,1],$ we choose an extension of the map ${\{|r| \geq 1 \}} \to \cM_{0,p,p}^{\zp},$ up to reparametrization, to a smooth map $\R \to \cM_{0,p,p}^{\zp}.$ This is indeed possible, by direct construction involving hyperbolic polygons: for example, representing each such complex structure by a hyperbolic metric with cusps at the marked points, and requiring that the metric be invariant under the $\zp$-action (which, we recall, acts freely transitively on the set of negative cusps, and also on the set of positive cusps), as well as under the orientation-reversing involution obtained from complex conjugation on $S$ and $(s,t) \mapsto (-s,t)$ on each $C^{\m},$ the parameter $r \in \R \cong \mathbf{\cR},$ is given, up to reparametrization, by $r = -\log(l),$ where $l$ is the length of the closed geodesic in the homotopy class determined by $\epsilon^{0}_{-}(\mrm{pt} \times S^1).$
	
	
	Denote by $\mathbf{\cS} \to \mathbf{\cR}$ the universal Riemann surface. Note that $\mathbf{\cR}$ admits a natural compactification $\overline{\mathbf{\cR}} \cong [0,1],$ where $\overline{\cR}_0$ is given by $\til{C},$ and $\overline{\mathbf{\cR}}_1$ is given by a nodal surface with the complement of the node given by two connected components isomorphic to $S_{\cC}$ and $S_{\cP}$ respectively. Finally, we note that each $\cl S_r,$ $r \in \mathbf{\cR},$ admits a holomorphic map $\pi_r: \cl S_r \to Z = \R \times S^1,$ that is a branched cover with branch locus consisting of two points in $Z.$ Moreover we may choose cylindrical ends $\mathbf{\epsilon}_{-}^{\m}:(-\infty,-1] \times \mathbf{\cR} \to \mathbf{\cS},$ $\mathbf{\epsilon}_{+}^{\m}:[1,\infty) \times \mathbf{\cR} \to \mathbf{\cS},$ for $\m \in \zp,$ so that for each $r \in \mathbf{\cR},$  ${\epsilon}_{-,r}^{\m} = \mathbf{\epsilon}_{-}^{\m}(-,r),$ and ${\epsilon}_{+,r}^{\m} = \mathbf{\epsilon}_{+}^{\m}(-,r)$ satisfy $\pi_r \circ {\epsilon}_{-,r}^{\m} = \id_{(-\infty,-1] \times S^1},\; \pi_r \circ {\epsilon}_{+,r}^{\m} = \id_{[1,\infty) \times S^1},$ for all $\m \in \zp.$
	
	Counting solutions to the parametric Floer equation on the family $\mathbf{\cS} \to \mathbf{\cR}$ of Riemann surfaces, with Floer data depending on points in $S^{\infty},$ as above, provides a chain homotopy between $\cl Z_x^{\loc}$ and $\cl C_x^{\loc} \circ \cl P_x^{\loc}.$ We sketch the technical details below.

	
We choose Floer data $\{J_{z,w,r}\}_{r \in \mathbf{\cR}},$ $\{Y_{z,r}\}_{r \in \mathbf{\cR}},$ depending on $z \in \mathbf{\cS}_r,$ and $w \in S^{\infty},$ which with respect to the $z,w$ coordinates are constant on the cylindrical ends, and satisfy the interpolation (with respect to the $\pi_r$ map), $\zp$-invariance, and shift-invariance axioms (see Sections \ref{sec:equiv-Floer-coho}, \ref{sec:prod-coprod}). Furthermore, as in \cite[Chapter 9]{Seidel-book}, we choose this Floer data compatible with the compactification $\mathbf{\overline{\cR}}$ of $\mathbf{\cR}$ and choices of Floer data for $S_{\cP},S_{\cC}$ and for $\til{C}.$ Note that in fact, we should also take our fibrations $\pi_{S_r}: E_{S_r} \to S_r$ compatible with the compactification. This does not present a difficulty, as the fibrations in the definitions are merely auxiliary, and all the Floer equations we consider can be written in terms of maps from suitable surfaces to $M.$ 
	
	
Then for $w \colon \R \rightarrow S^{\infty}$ denoting a gradient flow line that is asymptotic to $Z_i^{\m}$ at $-\infty$ and $Z_{\alpha}^0$ at $\infty,$ one considers the moduli space 	$\cM^{i,\m}_{\mathbf{\cK},\alpha}(x_0^{-},\ldots, x_{p-1}^{-};x_0^{+},\ldots, x_{p-1}^{+}) $ of solutions $(r,w,u \colon {\mathbf{\cS}_r} \rightarrow E_{S_r}),$ $\pi_{S_r} \circ u = \pi_{S_r},$ where $r \in \mathbf{\cR},$ to the parametrized perturbed Cauchy-Riemann equation
	\begin{equation} \label{eqn:homot}
	(du-Y_{\tilde{z},r})\circ j =J_{\tilde{z},w,r} \circ (du-Y_{\tilde{z},r}),
	\end{equation}
	with asymptotic behavior
	\begin{equation}
	\displaystyle\lim_{s \rightarrow - \infty} u(\tilde{\epsilon}^-_k(s,t))=x^-_k(t), \ \  \displaystyle\lim_{s \rightarrow - \infty} u(\tilde{\epsilon}^+_k(s,t))=x^+_k(t) \text{ for } k \in \zp.
	\end{equation}
	For generic choice of $J_z^w$ and $H^w_t$, the moduli space of non-constant solutions to \eqref{eqn:homot} $\cM^{i,\m}_{\mathbf{\cK},\alpha}(x_0^{-},\ldots, x_{p-1}^{-};x_0^{+},\ldots, x_{p-1}^{+})$ is a smooth finite dimensional manifold of dimension
	\begin{equation}
	\sum_k |x^-_k|- \sum_k |x^+_k|+i-\alpha-2n(p-1)+1.
	\end{equation}
	For each $\m \in \zp$ and $\sum_{k=0}^{p-1} |x^-_k|=\sum_{k=0}^{p-1} |x^+_k|-i + \alpha+2n(p-1)-1$, one can define an operation by 
	\begin{align}
	 &\cK^{i,\m}_{\alpha}\colon C^*(\zp; CF^*(\phi)^{\otimes p}) \rightarrow C^{*-i+\alpha+2n(p-1)-1}(\zp; CF^*(\phi)^{\otimes p}) \\
	 &\cK^{i,\m}_{\alpha}(x^+_0 \otimes \cdots \otimes x^+_{p-1})=\sum \#\cM^{i,\m}_{\cK,{\alpha}}(x_0^-, \cdots, x_{p-1}^-;x_0^+, \cdots, x_{p-1}^+)(x^-_0 \otimes \cdots \otimes x^-_{p-1})
	\end{align}

	
Again, we set $\cl K^i_{\alpha}=\sum_{\m} \cl K^{i,\m}_{\alpha},$ and consider the map
\begin{eqnarray}
&& \cl K \colon CF^*(\phi^p)[[u]] \langle \theta \rangle \rightarrow CF^*(\phi)^{\otimes p}[[u]] \langle \theta \rangle\\
&& \cl K(-\otimes 1)=(\cK_0^0 + u \cK^2_0 + \ldots ) \otimes 1+ (\cK^1_0 + u \cK^{3}_0 + \ldots ) \otimes \theta,\\
&& \cK(-\otimes \theta)=(u \cK^2_1 + u^2 \cK^4_1 + \ldots ) \otimes 1+ (\cK^1_1 + u \cK^{3}_1 + \ldots ) \otimes \theta.
\end{eqnarray}	
	
By standard compactness, transversality, and gluing arguments, we obtain that \[\cC_x^{\loc} \circ \cP_x^{\loc} - \cl Z_x^{\loc} = d^{\zp,{\loc}}_x \cK - \cK d^{\zp,{\loc}}_x,\] as required.
	


\end{proof}

\begin{rmk}
The same definition of $\cl Z$ works in the general context of local Floer cohomology, as well as globally, in the symplectically aspherical and exact cases, and the above argument can be modified to prove that $\cl Z$ is chain-homotopic to $\cC \circ \cP.$ In general, however, these maps differ on the chain level. In case of Hamiltonian diffeomorphisms, this point is not important for us, since it is straightforward to see that $HF_{\zp}(\phi^p) \cong H^*(M) \otimes R_p$ by either a continuation map argument, or by a suitable equivariant PSS map (see \cite{Wilkins-PSS}). The local version of this map, however, is required for our main invertibility argument. We shall also use the fact that $\cl Z$ is a chain map in a very particular local case for the proof of Lemma \ref{lma: local equality pm}.
\end{rmk}

\medskip

\begin{rmk}
	The above curve $C$ suggests the configurations one should consider in the presence of holomorphic spheres (in the non-local, non weakly exact case). We believe that at least when the manifold is monotone this can be carried out, but we opt to leave this point for discussion elsewhere.
\end{rmk}

For a matrix $A \in \mrm{Sp}(2n,\R)$ and $\lambda \in \C$ set $\rho_{\lambda}(A) = \det(A-\lambda\cdot\id),$ and denote \[\mrm{Sp}(2n,\R)^* = \mrm{Sp}(2n,\R) \setminus \rho_1^{-1}(\{0\}).\] Note that $\mrm{Sp}(2n,\R)^*$ has two connected components, and \[\sign(\rho_1): \pi_0 (\mrm{Sp}(2n,\R)^*) \to \{ \pm 1 \}\] determines an isomorphism of sets. Similarly, let \begin{equation}\label{eq: chambers} \mrm{Sp}(2n,\R)^{p*} = \mrm{Sp}(2n,\R) \setminus \bigcup_{\lambda^p = 1} \rho_{\lambda}^{-1}(\{0\}).\end{equation} 


\medskip

\begin{lma}
	$\cl Z_x^{\loc} = c_{n,\pm} \cdot u^{(p-1)n} \cdot \id,$ where $c_{n,\pm}$ depends only on the connected component of $D\phi_x \in \mrm{Sp}(2n,\R)^*.$ 
\end{lma}

\begin{proof}
We follow \cite[Section 6]{Seidel}, the difference being that we consider the operator $\cl Z_x^{\loc},$ and hence we may deform the differential of the symplectomorphism inside $\mrm{Sp}(2n,\R)^*.$ First of all, by degree reasons, and since $CF^*(\phi)^{\loc}_x$ has, in the non-degenerate case that we consider, the unique generator $x,$ \[\cl Z_x^{\loc}(x^{\otimes p} \otimes 1) = \left(\sum \# \cl M^{2n(p-1),\m}_{\cl Z, 0} (x,\ldots, x; x,\ldots, x) \right) u^{n(p-1)} x^{\otimes p} \otimes 1.\]

It is sufficient to show that for an isolated deformation of the germ of $\phi$ at $x,$ keeping $x$ a non-degenerate fixed point, the count $\sum \# \cl M^{2n(p-1),\m}_{\cl Z, 0} (x,\ldots, x; x,\ldots, x)$ remains invariant. This is carried out precisely as in \cite[Section 6]{Seidel}, by a cobordism argument, which ultimately works because the structure of the compactification of the spaces of Morse flow lines $\sP^{i,m}_{\alpha_0}$ yields the cancellation of the signed counts of the boundary points of the compactified one-dimensional parametric moduli spaces, lying over the interior $(0,1)$ of the parameter space $[0,1],$ after summing over $m \in \zp.$ Specifically, given a point in the boundary of the compactification of a one-dimensional component of the space of solutions to the parametric equation occurring at $r \in (0,1),$ by considering the structure of the corresponding boundary curve, and an index calculation, it is seen as in \cite[Section 6]{Seidel} that the corresponding solution must correspond to the strata \[\sQ^{1,m_1}_{0} \times \sP^{2n(p-1),m_2}_{1},\;\;\sP^{2n(p-1)-1,m_1}_{0} \times \sQ^{2,m_2}_{1},\] where $m_1 + m_2 = m$ in $\zp.$ The principal component of our boundary solution in the compactified parametrized moduli space $\cl M^{2n(p-1),\m}_{\cl Z, 0} (x,\ldots, x; x,\ldots, x)_{\mrm{para}}$ sits therefore in $\cl M^{2n(p-1),m_2}_{\cl Z, 1} (x,\ldots, x; x,\ldots, x)_{\mrm{para}}$ or in $\cl M^{2n(p-1)-1,m_1}_{\cl Z, 0} (x,\ldots, x; x,\ldots, x)_{\mrm{para}},$ and can be seen to be regular. The boundary point is obtained from the principal component and constant non-principal components. However, fixing $m_2$ in the first case, and $m_1$ in the second case, by varying $m_2$ or respectively $m_1,$ we obtain $p$ contributions to the count of boundary points, each one for a different $m.$ The constant non-principal components appear with the same signs as the corresponding unparametrized negative gradient trajectories, and therefore the different contributions in each case cancel out modulo $p$ in the sum over all $m.$ Indeed, in the first case this is because $1-1 = 0\;(\mrm{mod}\; p),$ and in the second case, $1+\ldots+ 1 = 0\;(\mrm{mod}\; p).$


\end{proof}

\begin{lma}\label{lma: local equality pm}
	We have $c_{n,+} = c_{n,-}.$
\end{lma}

\begin{proof}
Here we follow \cite[Section 7]{Seidel}, but for $\cl Z_x^{\loc}.$ Consider the case when locally in a ball $B \subset (\R^{2n},\om_{st}),$ $\phi$ has two non-degenerate fixed points $y,z$ with $dy = z.$ For example, we may take a small Morse function $H$ with $y,z$ being critical points that lie in $B,$ with one gradient trajectory from $y$ to $z,$ and let $\phi$ be the Hamiltonian flow of $H$ for a small positive time $\epsilon >0.$ 

Hence $HF^{\loc}(\phi,B) = 0,$ and hence $H^*(\zp; CF^{\loc}(\phi,B)^{\otimes p}) = 0.$ Moreover, clearly, for the Tate cohomology $\wh{H}^*(\zp; CF^{\loc}(\phi,B)^{\otimes p}) = 0.$ Now, the Tate cohomology can be computed by the action spectral sequence, whose $E^{p+1}$-st page is given by the map induced from $d:\bK\left<y\right> \to \bK \left<z\right>,$ which becomes the identity map after identifying the domain and target with $\bK$ by means of the natural bases $\{y\},\{z\}$ by the natural quasi-Frobenius isomorphisms (see Lemma \ref{lma:quasi_Frob})  \[\wh{H}^*(\zp, CF^{\loc}(\phi,y)^{\otimes p}) \cong \hrp \otimes_{\bK} \bK\left< y \right>^{(1)},\] \[\wh{H}^*(\zp, CF^{\loc}(\phi,z)^{\otimes p}) \cong \hrp \otimes_{\bK} \bK\left< z \right>^{(1)}.\] We proceed to note that $D\phi(y), D\phi(z)$ lie in different components of $\mrm{Sp}(2n,\R)^*.$ Furthermore, since $\cl Z$ induces a chain map between the Tate complexes, it induces a map of the action spectral sequences, for sufficiently small perturbation data, and in particular it induces a chain map on the $E^{(p+1)}$-st page. This immediately yields $c_{n,+} = c_{n,-}.$   
\end{proof}


Below, we apply a suitable Floer-to-Morse reduction to show that $c_{n,+} = (-1)^n,$ and hence by Lemma \ref{lma: local equality pm}, $c_{n,\pm} = (-1)^n.$ We could also calculate $c_{n,-}$ separately in the same way, proving our result without the above lemma.



\subsection{Local Floer-to-Morse collapse}\label{sect: Floer to Morse}



Section \ref{subsec:invar} allows one to reduce the consideration of $\cl Z^{\loc}_x$ to the local case of a fixed ball $B \subset (\C^{n},\om_{st})$ of radius $1,$ and symplectomorphism $\phi$ generated by Hamiltonian $H = \epsilon \cdot |z|^2,$ where $|z|^2 = \frac{1}{\pi}\sum_{j=1}^n |z_j|^2,$ and $1 \gg \epsilon > 0$ chosen arbitrarily. Note that we may choose the complex structure to coincide with the standard one $J_{st}$ outside $\frac{19}{20} B,$ and to be $C^2$-close to $J_{st}$ on $B.$ In this section we reduce the calculation in this case to Morse theory.



\medskip

We start with the classical observation that the local Floer complex of $\phi^1_H$ at $0,$ as computed with an $\om_{st}$-compatible almost complex structure $J$ coinciding with $J_{st}$ outside $\frac{19}{20} B,$ is isomorphic to the local Lagrangian Floer complex $CF^{\loc}(\Delta,H \oplus 0, x)$ of the Lagrangian diagonal \[\Delta \subset X,\] \[ X = \C^n \times (\C^n)^{-} \] at $(\phi^1_H \times \id)^{-1} \Delta \cap \Delta = \{x=(0,0)\},$ with Hamiltonian perturbation $H \oplus 0,$ and almost complex structure $J \oplus -J.$ Furthermore, denoting by $\Delta_H = (\phi^1_{H \oplus 0})^{-1}\Delta,$ under this identification \[\cl Z_{x}^{\loc} = \cl Z_{x}^{\loc,\mrm{Lagr}}: C^*(\zp; CF^{\loc}(\Delta,\Delta_H, x)^{\otimes p}) \to C^*(\zp; CF^{\loc}(\Delta,\Delta_H, x)^{\otimes p}),\] the latter being defined analogously to $Z_{x}^{\loc},$ yet in terms of Lagrangian Floer cohomology. 

Finally, we note that the complex structure $J_{st} \oplus -J_{st}$ on $X \cong T^*\Delta,$ coincides with the complex structure on $X,$ induced by the standard Riemannian metric $g$ on $\Delta.$ We recall that a Riemannian metric induces an $\om_{can} = d(\lambda_{can})$-compatible almost complex structure by identifying $T_{(p,q)} (T^*\Delta) \cong T^*_q \Delta \oplus T_q \Delta,$ via the Levi-Civita connection, and acting by $(\alpha,\xi) \mapsto (- G\xi, G^{-1}\alpha),$ where $G:T_q \Delta \to T^*_q \Delta$ is the isomorphism $G(\xi) = g_q(\xi,-).$ Furthermore, we recall that the symplectomorphism $\Theta: X \to T^* \Delta,$ with the standard symplectic structures, is given by the symplectic Cayley transform $(z,w) \mapsto (p,q),$ $p=\frac{w-z}{i\sqrt{2}}, q = \frac{z+w}{\sqrt{2}}.$ 

Let $\pi: T^*\Delta \to \Delta$ be the natural projection. We rewrite the $\cl Z^{\loc}_{x}$ map yet again as \[\cl Z_{x}^{\loc} = \cl Z_{x}^{\loc,\mrm{cot}}: C^*(\zp; CF^{\loc}(\Delta,\pi^*F, x)^{\otimes p}) \to C^*(\zp; CF^{\loc}(\Delta,\pi^* F, x)^{\otimes p})\] with $F\in \sm{\Delta,\R}$ such that \[\Gamma_{dF} = \Delta_H = (\phi^1_{H \oplus 0})^{-1}\Delta.\]

The Morse function $F$ is given in terms of $H$ by the Cayley transform. In particular, like $H,$ $F$ has a unique non-degenerate minimum at $x.$ More precisely, for $q \in \Delta \cong \C^n,$ \[F(q) = \Theta^{\#} H (q) = \int_0^1 \left<(\pi|_{\Delta_H})^{-1}(t \cdot q), q \right>\,dt.\] In other words $\Theta^{\#}H(q) = \int_{\{(\pi|_{\Delta_H})^{-1}(t\cdot q) \;\mathbin{|}\; t \in [0,1] \}} \lambda.$ This formula establishes a bijective correspondence between functions $H'$ that are $C^2$-close to $H$ and functions $F'$ that are $C^2$-close to $F,$ which is continuous in the $C^2$-topology. Moreover, the last observation applies to $H=0.$ Therefore by making $H$ sufficiently $C^2$ small in a suitable neighborhood of $0$ we can make $F$ as $C^2$-small as necessary in a neighborhood of $x.$ Hence we may consider the latter model for all our purposes.




\medskip

The above Lagrangian reformulation allows us somewhat more freedom in the choice of almost complex structures, while the reformulation in terms of the function $F$ allows us to relate our constructions to Morse theory. We shall use both, together with a classical convexity argument of Floer \cite{Floer-MorseWitten} to show the following result. 

\medskip

\begin{lma}\label{lemma:Floer-to-Morse}
There exists $\epsilon_0 > 0$ such that for Hamiltonian perturbations $F_s$ of $F$ that have $C^2$ norm smaller than $\epsilon_0,$ and coincide with $F$ outside $\frac{1}{20} B,$ all continuation maps \[C(\{F_s\}): CF^{\loc}(\Delta,\pi^*F, x) \to CF^{\loc}(\Delta,\pi^* F, x)\] along a family of Hamiltonians $\{F_s\}_{s\in \R},$ with $F_s = F$ for $|s|\gg 1,$ and almost complex structures $J_s$ suitably chosen, in fact {\em coincide} with the Morse continuation maps along $\{F_s\}$ considered as Morse functions, with suitable Riemannian metrics $g_s.$
\end{lma}

As a consequence, we obtain a Morse-theoretical description of $\cl Z_{x}^{\loc,\mrm{cot}},$ which we detail in slightly larger generality in Section \ref{subsec:Betz-Cohen}, and use in Section \ref{subsec: Wilson theorem}.


%

\begin{proof}[Proof of Lemma \ref{lemma:Floer-to-Morse}]

We begin by describing the class of almost complex structures $J_s$ that we consider. Firstly, as long as a complex structure is sufficiently $C^2$-close to $J_{st} \oplus -J_{st},$ a classical monotonicity argument \cite[Section 4.3]{Sikorav}, together with standard action estimates (similar to \eqref{eq: main action estimate}, see \cite[Section (8g)]{Seidel-book}), shows that all relevant curves do not escape a given fixed neighborhood $U \subset B \times B$ of $x.$ We take as $U$ a symplectically embedded copy of $D^*(\frac{1}{\sqrt{2}} B),$ the unit co-disk bundle taken with respect to the standard metric on $B$ (indeed $x \in U$). Hence, given $g_s$ sufficiently $C^2$-close to $g,$ we may assume that in $U,$ ${J_{s}}$ concides at the zero section with the almost complex structure on the co-disk bundle induced by $g_s.$ In this case, each trajectory $\gamma : \R \to \C^n$ satisfying the Morse continuation equation \begin{equation}\label{eq:Morse continuation} \del_s\gamma + (\nabla F_s) \circ \gamma= 0,\end{equation} \[\gamma(s) \xrightarrow{s \to \pm \infty} x,\] is in fact also a solution of the Floer equation \begin{eqnarray}\label{eq:perturbed Floer}
\del_s u + J_{s} \circ u \;(\del_t u - X^t_{\pi^* F_s}\circ u) = 0,  \end{eqnarray} \[u(s,t) \xrightarrow{s \to \pm \infty} x,\]
on maps $u: \R \times [0,1] \to X.$ Moreover, given that $|F_s|_{C^2(U,g_s)} < \epsilon_0$ for all $s,$ each solution $u$ of \eqref{eq:perturbed Floer}, which necessarily lies inside $U,$ is in fact of the above form. Indeed, writing $u(s,t) = (y(s,t),x(s,t)),$ where $x(s,t) \in \frac{1}{\sqrt{2}} B,$ and $y(s,t) \in T^*_{x(s,t)} (\frac{1}{\sqrt{2}} B),$ and setting \[\eta(s) = \int_{0}^{1}\left |y(s,t) \right |^2_{g_s} \,dt,\] one shows closely following Floer \cite{Floer-MorseWitten} the convexity estimate \[\eta''(s) \geq \delta \cdot \eta(s)\] for a positive constant $\delta = \epsilon_0^2/2 >0.$ As $\eta(s) \xrightarrow{s \to \pm \infty} 0,$ the convexity estimate implies that $\eta \equiv 0,$ whence $y(s,t) \equiv 0.$ Now in view of \eqref{eq:perturbed Floer}, $\del_t x \equiv 0,$ and $\gamma(s) = x(s)$ satisfies \eqref{eq:Morse continuation}.
\end{proof}


%




\subsection{The Betz-Cohen computation for Morse functions}\label{subsec:Betz-Cohen}


\indent In this subsection, we consider the analogue of Fukaya \cite{Fukaya93} and Betz-Cohen \cite{BetzCohen} (see also \cite{BetzThesis},\cite{CohenNorbury}) operations defined by counting of Morse trees with a $\zp$-symmetry. We remark that when $p=2,$ the analogue of our construction has been defined in the work of Seidel \cite{Seidel} and that of Wilkins \cite{Wilkins} recently. 

\indent Given any smooth manifold $M$ and a Morse function $f \in C^{\infty}(M),$ for each fixed prime number $p \geq 2$, we consider the graph $\Gamma_p$ with $p$ inputs and one output, oriented and parametrized as $(-\infty, 0] \cup \bigcup_{i=1}^p [0, \infty)^i \xrightarrow{\sim} \Gamma_p.$ The edges $e_{out}$ and $e_{in}^i$ of tree are parametrized by half-infinite intervals $(-\infty, 0]$ and $[0, \infty)$ respectively. We can choose a domain-dependent perturbation $f_s^i$ of $f$ on each edge of $\Gamma_p$ such that \[f_s^i=f \; \text{for} \; |s| \gg 1 \; \text{in} \;  [0, \infty),\] and \[\; f^0_s = pf \; \text{for} \; |s| \gg 1 \; \text{in} \;  (-\infty, 0].\] The latter choice is convenient for our arguments, but in general one could pick $f^0_s = f$ for  $|s| \gg 1$ in $(-\infty, 0]$ as well.

\begin{figure}[htb]
	\centerline{\includegraphics{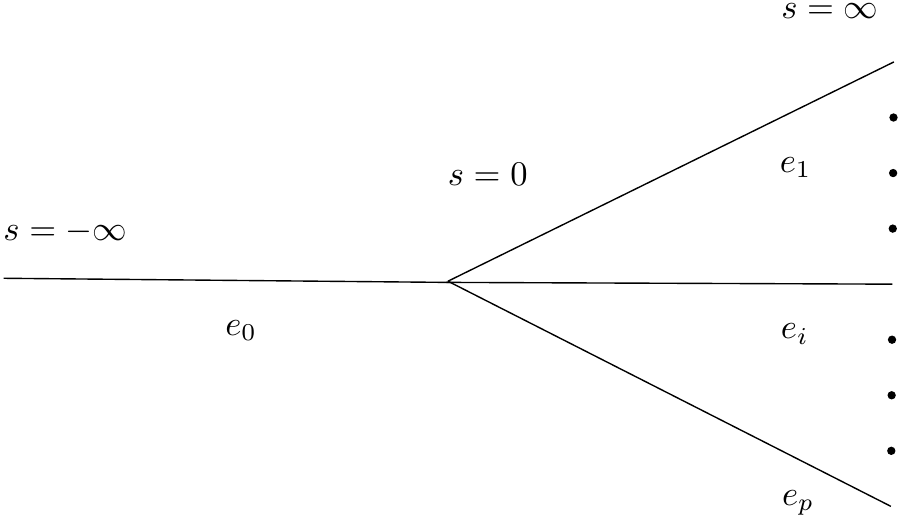}}
	\caption{The graph $\Gamma_p.$}\label{Figure: graph Gamma_p}
\end{figure}


We let $C^*(M):=C^*_{Morse}(M)$ be the Morse cochain complex defined by $f.$ We grade the critical points of $f$ by the Morse index. For each prime $p,$ one can define the $p$-th product map on the Morse cochain complex by $\mathcal{P}_M \colon C^*(M)^{\otimes p} \rightarrow C^*(M)$ by
\begin{equation}
\mathcal{P}_M(x_1 \otimes x_2 \cdots \otimes x_p)=\sum_{x_0 \in \mathrm{Crit}(f)} \#\overline{\mathcal{M}}(\Gamma_p;f_s^i) x_0,
\end{equation}
where $\#\overline{\mathcal{M}}(\Gamma_p;f_s^i) $ denotes the signed counts of the virtual dimension zero part of the moduli space of Morse trajectories $u \colon \Gamma_p \rightarrow M$ satisfy the following conditions
\begin{enumerate}
	\item[(1)] $d(u|_{e_i})/ds=-\nabla f_s^i$ for all $i =0 ,1 \cdots,  p;$
	\item[(2)] $\displaystyle\lim_{s \rightarrow - \infty} u|_{e_{out}}(s)=x_0;$
	\item[(3)]$\displaystyle\lim_{s \rightarrow \infty} u|_{e_{in}^j}(s)=x_j$ for all $j=1,2, \cdots, p.$
\end{enumerate}
We remark that our perturbations $f_s^i$ of $f$ should be chosen satisfying suitable, generically satisfied, transversality conditions so that the moduli space of Morse trajectories $\mathcal{M}(\Gamma_p;f_s^i)$ are regular (we refer to \cite[Section 6]{CohenNorbury} for a discussion of transversality in a more general setting, and more specifically \cite[Appendix B]{Wilkins} for a detailed discussion of transversality for the case $p=2$ in the above construction). Now if all the inputs are the same, that is, one has $x_1=x_2=\cdots=x_p,$ then the $\zp$-action on the domain graph $\Gamma_p,$ which is free and transitive on the input edges, preserves the boundary conditions. We would like to study the equivariant operation that this symmetry produces.

A priori if one takes the perturbation $f_s^1=\cdots= f_s^p$ respectively, the moduli space $\mathcal{M}(\Gamma_p;f_s^i)$ also admits an action by $\zp.$ One may want to study its quotient and define the $\zp$-equivariant product by counting the virtual dimension zero part of $\mathcal{M}(\Gamma_p;f_s^i)/\zp$. However, there is a major problem that $\mathcal{M}(\Gamma_p;f_s^i)$ cannot be made regular for $\zp$-symmetric perturbations $f_s^i,$ since one cannot achieve $\zp$ transversality for elements in $\mathcal{M}(\Gamma_p;f_s^i)$. As a solution, we consider $\zp$-equivariant perturbations $\{f_{s,w}^i\}$ parametrized by $w \in S^{\infty}$ as before. For a given Morse trajectory $w \colon \R \rightarrow S^{\infty}$ which is asymptotic to a critical point $Z_i^{\m}$ of index $i$ when $s\rightarrow -\infty$ and another critical point $Z_{\alpha}^0$ of index $\alpha \in \{0,1\}$ as $s \rightarrow \infty,$ one has that the $\zp$-equivariant perturbation satisfy
\begin{enumerate}[label = (\arabic*)]
	\item \label{func cond 1}(Diagonal $\zp$-action) $f^{\m \cdot i}_{s, \m \cdot w} = f^{i}_{s, w}$ for all $i =1,2, \cdots, p$ and $f_{s, z}^0=f^0_{s,\m \cdot w}$ for all $\m \in \zp$ and $w \in S^{\infty},$ where $\m \in \zp$ acts on $\{1,\ldots,p\}$ by $\tau^m,$ for the cyclic permutation $\tau = (1\, 2 \, \ldots \,p)$ and $\m \cdot w = e^{2\pi im}\, w.$
	\item \label{func cond 2}(Invariant under shift)  $f_{s, \tau(w)}^i=f_{s, w}^i$ for all $i$ and $w \in S^{\infty}.$
\end{enumerate}

Given any parametrized Morse flow line $w \in {\sP}_{\alpha}^{i,\m}$ for fixed $i$ and $\m \in \zp,$ we define the operation $\mathcal{P}_{M,\alpha}^{i,\m} : C^*(M)^{\otimes p} \to C^{*-i+\alpha}(M)$ by
\begin{equation}
\mathcal{P}_{M,\alpha}^{i,\m}(x_1 \otimes \cdots \otimes x_p)=\sum_{x_0} \# \overline{\mathcal{M}}_{\alpha}^{i,\m}(M;f_{s,z}^j) \,x_0, \nonumber
\end{equation}
where $\# \overline{\mathcal{M}}_{\alpha}^{i,\m}(M;f_{s,z}^j) $ denotes the signed counts of the virtual dimension zero part of the moduli space of pairs $(u,w)$ satisfying $w \in {\sP}_{\alpha}^{i,\m}$ and
\begin{equation}
\begin{cases}
d(u|_{e_i})/ds=-\nabla f_{s,w(s)}^i  \text{ for all } i =0 ,1 \cdots,  p; \\
\displaystyle\lim_{s \rightarrow - \infty} u|_{e_{out}}(s)=x_0; \\
\displaystyle\lim_{s \rightarrow \infty} u|_{e_{in}^j}(s)=x_j \text{ for all } j=1,2, \cdots, p.
\end{cases}
\end{equation}

It is then straightforward to check that, as in the definitions of Section \ref{sec:prod-coprod}, the operations $\cl P^{i,m}_{M,\alpha}$ combine to give a chain map \[\cl P_M : C^*(\zp; CM^*(f)^{\otimes p}) \to CM^*(p f) \otimes \rp.\] Observe that Condition \ref{func cond 1} of the perturbations is necessary for $\cl P_M$ to be a chain map, and Condition \ref{func cond 2} yields linearity with respect to multiplication by $u.$


%

\subsection{Morse coproduct, and the local case}\label{subsec: Wilson theorem}
Similarly to the operation $\cl P_M$ defined above, we define operations \[\cl C_M: CM^*(f) \otimes \rp \to C^*(\zp; CM^*(f)^{\otimes p}),\] given by the graph $\overline{\Gamma}_p$ with one input and $p$ outputs, oriented and parametrized as $\bigcup_{i=1}^p (-\infty,0]^i \cup [0,\infty)  \xrightarrow{\sim} \overline{\Gamma}_p.$ 

\begin{figure}[htb]
	{\includegraphics{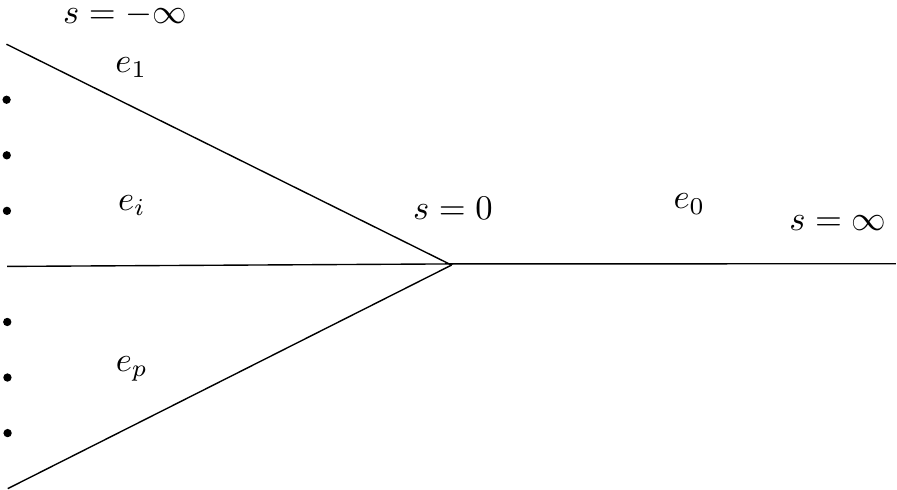}}
	\centering
	\caption{The graph $\overline{\Gamma}_p.$}\label{Figure: graph barGamma_p}
\end{figure}

Consider also the operation \[\cl Z_M: C^*(\zp; CM^*(f)^{\otimes p}) \to C^*(\zp; CM^*(f)^{\otimes p}),\] given by the graph $\Gamma_{p,p}$ with $p$ inputs and $p$ outputs parametrized and oriented as $\bigcup_{i=1}^p (-\infty,0]^i \cup \bigcup_{i=1}^p [0,\infty)^i \xrightarrow{\sim} {\Gamma}_{p,p}.$  The analysis in Section \ref{subsec:invar} above shows that $\cl Z_M$ is chain homotopic to $\cl C_M \circ \cl P_M.$

\begin{figure}[htb]
	{\includegraphics{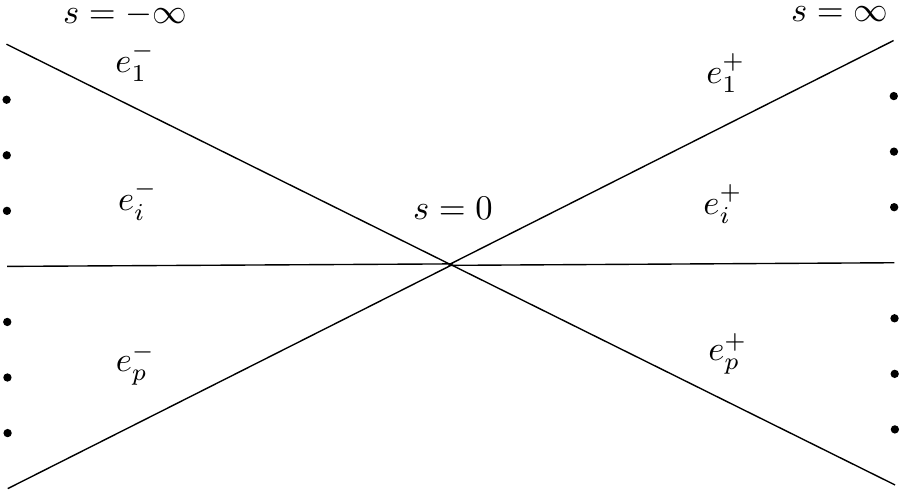}}
	\centering
	\caption{The graph $\Gamma_{p,p}.$}\label{Figure: graph Gamma_pp}
\end{figure}

%
Considering the case of local Morse cohomology of a function $f(z)= \epsilon |z|^2$ on $\C^n$ at $0\in \C^n,$ we get that under the natural isomorphisms of the equivariant cohomology groups with $\rp,$ we have  \[\cl Z_M =  (-1)^n u^{n(p-1)} \cdot \id.\] In this case the construction of Section \ref{subsec:Betz-Cohen} for $\cl Z_M$ is in fact given, in the topological model of equivariant cohomology, by multiplication by the Euler class of the vector bundle $V= E/E^{\zp}$ over $B(\zp),$ where $E = (\R^{2n})^p \times_{\zp} S^{\infty},$ and $E^{\zp} \cong \R^{2n} \times B(\zp)$ is the subbundle in $E$ corresponding to the $\zp$-invariant subspace in $(\R^{2n})^p.$ Indeed the above construction of $\cl{Z}_M$ describes in this case the map dual to the cap product by the Euler class of $E,$ given by intersection product with the zero section in the total space of $E,$ taking into account that the latter, being the total space of a vector bundle, is deformation equivalent to the base (see also \cite{BetzCohen}).

We calculate this Euler class as follows. Using the standard isomorphism $\C^n \cong \R^{2n}$ of real vector spaces, we endow $\R^{2n}$ with a complex structure and note that the $\zp$-action on $(\R^{2n})^p$ is in fact complex-linear. Hence we may consider $E, E^{\zp}$ and $V$ as complex vector bundles of ranks $np, n, n(p-1)$ respectively. In particular the Euler class of $V$ is given by its top Chern class $c_{n(p-1)}(V).$ Considering $(\C^n)^p$ as a complex $\zp$-representation and splitting it into isotypical components, we obtain the isomorphism of complex vector bundles \[V \cong \left( \bigoplus_{\chi \in (\zp)^{\vee}\setminus \{1\}} L_{\chi}\right) ^{\oplus n},\] where $L_{\chi} = \C \times_{\zp} S^{\infty}$ is the complex line bundle over $B(\zp)$ associated to the character $\chi \in (\zp)^{\vee} = \Hom(\zp, \C^{\times}).$ Note that the sum does not include the component of the trivial character, since we took the quotient by the $\zp$-invariants. Now by \cite[Section 8]{Atiyah}, the ($\F_p$-reductions of) $c_1(L_{\chi})$ for $\chi \in (\zp)^{\vee}$ are given by $a_{\chi} u$ for different invertible elements $a_{\chi}$ in $\F_p,$ where $u$ is the standard generator of $H^2(B(\zp),\F_p) \cong \F_p,$ coming from $H^2(\C P^{\infty}, \F_p).$ In view of the Whitney sum formula, we obtain
\[c_{n(p-1)}(V) = \left(\displaystyle\prod_{\chi \in (\zp)^{\vee} \setminus \{1\}} a_{\chi}\right)^n u^{n(p-1)} = (-1)^n u^{n(p-1)},\] the last step being Wilson's theorem, whereby \[\displaystyle\prod_{\chi \in (\zp)^{\vee}\setminus \{1\}} a_{\chi} = -1.\] 




However, we saw in Section \ref{sect: Floer to Morse} that in our particular local case, for $x=0$ a fixed point of the Hamiltonian flow $\phi = \phi^1_H$ of $H(z) = f(z),$ \[\cl Z^{\loc}_x = \cl Z_M.\] 

This finishes the proof.

\section{Proof of Theorem \ref{thm: main}: spectral sequence argument}

First let us consider the case when $\phi$ and $\phi^p$ as above have all fixed points non-degenerate. Consider the equivariant product map 
\[\cP: C^*(\zp; CF^*(\phi)^{\otimes p}) \to CF^*_{\zp}(\phi^p).\] It induces a map on the action spectral sequences associated to the action filtrations on both sides. Tensoring with $\cK = \bK((u))$ over $\bK[[u]],$ we obtain a map of the corresponding spectral sequences for Tate cohomology groups. In view of Lemma \ref{lem:properties} in Section \ref{sec: group_coho} and the description of the action-filtration spectral sequences for the non-equivariant and equivariant Floer cohomology in Sections \ref{subsec: loc FH} and \ref{subsec: Local equiv}, the $E_1$-page on the left is given by \[\bigoplus \wh{H}^*(\zp; HF^{\loc}(\phi,x)^{\oplus p})\] and on the right by  \[ \bigoplus \wh{HF}^{\loc}_{\zp}(\phi^p,x^{(p)}),\] the sum running over all the fixed points $x$ of $\phi,$ and furthermore, by Section \ref{subsec: local prod coprod}, the map $\cP$ induces the map \[\oplus \cP^{\loc}_x\] on these $E_1$ pages. By Section \ref{sec:locally invertible}, this map on the $E_1$ page is invertible, and hence by the spectral sequence comparison argument \cite[Theorem 5.5.11]{Weibel}, we obtain that \[\cP: \wh{H^*}(\zp; HF^*(\phi)^{\otimes p}) \to \wh{HF}^*_{\zp}(\phi^p)\] is an isomorphism. This finishes the proof in the non-degenerate case.

Now, to generalize to the case of local Floer cohomology, we first note that the same argument, after choosing a small isolating neighborhood and a small Hamiltonian perturbation non-degenerate therein, applies to prove that \[\cP^{\loc}_x: \wh{H}^*(\zp; HF^{\loc}(\phi,x)^{\otimes p}) \to \wh{HF}^{\loc}_{\zp}(\phi^p,x^{(p)})\] is an isomorphism for an isolated fixed point of $\phi$ that is isolated as a fixed point of $\phi^p$ as well. 

Finally, we proceed to the general case. For a subset $A \subset \R,$ we introduce the following notation: $A^{+p} = \{a_1 + \ldots + a_p \,|\, a_j \in A,\, 1\leq j \leq p \}.$ Choose an interval $I = (a,b)$ with $pa, pb \in (\R \setminus (\mrm{Spec}(\ul \phi^p) \cup \mrm{Spec}(\ul \phi)^{+p})) \cup \{\pm \infty \}.$ Then the same argument, after choosing sufficiently small non-degerate Hamiltonian perturbations $\phi_1, (\phi_1)^p$ of $\phi,\phi^p,$ we obtain that \[\cP^{p\cdot I}: \wh{H}^*(\zp; CF^*(\phi)^{\otimes p})^{p\cdot I} \to \wh{HF}^*_{\zp}(\phi^p)^{p \cdot I}\] is an isomorphism, where $\wh{H}^*(\zp; CF^*(\phi)^{\otimes p})^{p\cdot I} = \wh{H}^*(\zp; (CF^*(\phi)^{\otimes p})^{p\cdot I}),$ for the quotient complex \[(CF^*(\phi)^{\otimes p})^{p \cdot I} \cong (CF^*(\phi)^{\otimes p})^{< p \cdot b}/ (CF^*(\phi)^{\otimes p})^{<p \cdot a},\] where for $x_1 \otimes \ldots \otimes x_p \in CF^*(\phi)^{\otimes p},$ we set \[\cA_{\phi,\otimes p}(x_1 \otimes \ldots \otimes x_p) = \sum \cA_{\phi}(x_j).\] Observe that there is a natural inclusion of complexes \[(CF^*(\phi)^I)^{\otimes p} \to (CF^*(\phi)^{\otimes p})^{p\cdot I},\] such that the $\zp$-action on the quotient is free. Therefore by Lemma \ref{lem:properties} it induces a natural isomorphism  \[ \wh{H}^*(\zp,(CF^*(\phi)^I)^{\otimes p}) \xrightarrow{\cong}  \wh{H}^*(\zp,(CF^*(\phi)^{\otimes p})^{p\cdot I}).\]

In all cases, we finish the proof by an application of Lemma \ref{lma:quasi_Frob}.



\section{Applications and discussion}\label{sec: applications}

First, in the exact case, we prove the following version of the Smith inequality for fixed point Floer cohomology, extending the work of Seidel to prime orders $p>2.$

\medskip


\begin{thm}\label{thm: Smith exact}
Let $\phi$ be an exact symplectomorphism of a Liouville domain $W$ which is cylindrical at infinity. Then the ranks of its fixed point Floer cohomology, and that of its $p$-th iterate $\phi^p$ satisfy the following inequality: \[\dim_{\F_p} HF^*(\phi) \leq \dim_{\F_p} HF^*(\phi^p)^{\zp} \leq \dim_{\F_p} HF^*(\phi^p).\]
\end{thm}

\begin{proof}
Observe that by Corollary \ref{cor: PF isomorphism} of the main result, Theorem \ref{thm: main}, we obtain \[2 \dim_{\F_p} HF^*(\phi) = \dim_{\F_p((u))} \wh{H}^*(\zp, HF^*(\phi))^{(1)} = \dim_{\F_p((u))} \wh{H}^*_{\zp}(\phi^p).\] However, by estimate \eqref{eq: dim Tate bound} the latter dimension satisfies \[\dim_{\F_p((u))} \wh{H}^*_{\zp}(\phi^p) \leq \dim_{\F_p((u))} \wh{H}^*(\zp, HF^*(\phi^p)),\] whence we obtain the bound \begin{equation}\label{eq: summary bound}2 \dim_{\F_p} HF^*(\phi) \leq \dim_{\F_p((u))} \wh{H}^*(\zp, HF^*(\phi^p))\end{equation} By the structure theorem for modules over PID, and the observation that $\F_p[\zp] \cong \F_p[t]/(t^p),$ where $t = \sigma - 1,$ we have that $HF^*(\phi^p),$ as a $\F_p[\zp]$-module, splits into a direct sum \[HF^*(\phi^p) = \bigoplus_{1 \leq k \leq p} (\F_p[t]/(t^k))^{\oplus m_k},\] for multiplicities $m_k \geq 0.$ It is easy to calculate that 
$$\dim_{\F_p((u))} \wh{H}^*(\zp, HF^*(\phi^p)) = 2 (m_1 + \ldots + m_{p-1}),$$ 
while $2 \dim_{\F_p} HF^*(\phi^p)^{\zp} =2 (m_1 + \ldots + m_{p-1} + m_p).$ Indeed, the Tate differentials in \eqref{eqn:res_Q} become multiplications by $t$ and $t^{p-1},$ and setting $\F_p[t]/(t^k),$ $1 \leq k \leq p$ for $V$ in \eqref{eq: Tate basic 2} immediately yields the result.
 \end{proof}

\medskip

\begin{rmk}\label{rmk: sharp Smith}
Note that if $m_p > 0,$ then the bound \eqref{eq: summary bound} is strictly stronger than the Smith-type inequalities \eqref{eqn:Smith_ineq}, \eqref{eq: Smith in interval}, which are directly analogous to the classical Smith inequality for $\zp$-actions on, say, manifolds. Indeed, $m_p > 0$ would correspond to a summand of $\F_p[\zp] \cong \F_p[t]/(t^p)$ in $HF^*(\phi^p)$ as a $\F_p[\zp]$-module, which does contribute to $\dim_{\F_p} HF^*(\phi^p)^{\zp}.$ We observe that the algebraic methods used in the proof of this bound work in the classical setting of the Smith inequality (see \cite{Borel-Transformation,Bredon-Transformation,Hsiang-Transformation}), for example when the space in question is a finite simplicial complex and the $\zp$-action is simplicial, and provide a sharpening thereof. We have not found this sharper version in the literature. Of course in the classical setting, quite a lot more than the Smith inequality is known by now. For example, the cohomology of the fixed point set of a $\zp$-action was completely described in terms the associated equivariant cohomology, as a module over both $\rp$ and the Steenrod algebra \cite{Smith-revisited}. It would be very interesting to find an analogue of the latter result in Floer theory. 
\end{rmk}

\medskip

\begin{cor} \label{app:MCG}
In particular, if $\dim_{\F_p} HF^*(\phi) > \dim_{\F_p} H^*(W)$ then $[\phi^{p^k}] \neq 1$ in the symplectic mapping class group of $W$ for all $k\geq 0.$ 
\end{cor}

This corollary is immediate, since $HF^*(\id) \cong HF^*(W)$ and $HF^*(\phi)$ is invariant under isotopies of exact symplectomorphisms of $W$ cylindrical at infinity.
Furthermore, iterating the inequality of Theorem \ref{thm: Smith exact} we obtain that $\dim_{\F_p} HF^*(\phi^{p^k})$ is a non-decreasing function of $k \in \Z_{\geq 0}.$

Furthermore, in both the exact and the symplectically aspherical setting, denoting for an open interval $I \subset \R,$ of the form $(a,b),$ $a,b \in \R,$ or $(-\infty,b),$ $b \in \R,$ by \[HF^*(\phi)^I\] the cohomology of the subcomplex generated by points of action value in $I,$ and by $k \cdot I$ for $k>0,$ the interval $(ka, kb)$ in the first case and $(-\infty, kb)$ in the second case, we obtain the following sharpening of Theorem \ref{thm: Smith exact}. We always assume that the finite endpoints of an interval are not in the spectrum of the associated Hamiltonian diffeomorphism.

\medskip

\begin{thm}\label{thm: Smith exact filtered}
Let $\phi$ be an exact symplectomorphism of a Liouville domain $W$ which is cylindrical at infinity or a Hamiltonian diffeomorphism of a closed symplectically aspherical symplectic manifold. Then the ranks of its fixed point Floer cohomology in action window $I,$ and that of its $p$-th iterate $\phi^p$ in action window $p\cdot I$ satisfy the following inequality: \[\dim_{\F_p} HF^*(\phi)^I \leq \dim_{\F_p} (HF^*(\phi^p)^{p\cdot I})^{\zp} \leq \dim_{\F_p} HF^*(\phi^p)^{p \cdot I}.\]
\end{thm}

We remark that this inequality is invariant with respect to shifts of the relevant action functionals, and therefore makes sense independently of them. Furthermore, we note that Theorem \ref{thm: Smith exact filtered} below is interesting for Hamiltonian diffeomorphisms on symplectically aspherical symplectic manifolds, even though the total cohomology is trivial, in the sense that in this case $HF^*(\phi) \cong HF^*(\id) \cong H^*(M)$ for all $\phi \in \Ham(M,\om).$

In both settings, we can consider the system $(HF^*(\phi)^{<t}, \pi^{s,t})$ where $t \in \R,$ $s \leq t,$ and $HF^*(\phi)^{<t} : = HF^*(\phi)^{(-\infty,t)},$ while $\pi^{s,t}: HF^*(\phi)^{<s} \to HF^*(\phi)^{<t}$ is the map induced by inclusions. In case $\phi$ is non-degenerate, it was observed in \cite{PolterovichS} that this system is a persistence module with certain additional constructibility properties. It therefore has an associated finite barcode, determined uniquely up to permutation: a finite multiset $\cB(\phi) = \{(I_j,m_j)\}$ of intervals of the form $I_j = (a_j, b_j]$ or $(a_j, \infty),$ and multiplicities $m_j \in \Z_{>0}.$ One of the properties of such a barcode is that the number of infinite bars, with multiplicities, is equal to the total dimension $B(\phi,\F_p)$ of $HF^*(\phi),$ which in case of $\phi$ a Hamiltonian diffeomorphism of a closed symplectic manifold coincides with the total Betti number $B(\F_p) = \dim_{\F_p} H^*(M, \F_p)$ of the symplectic manifold. We denote by $K(\phi,\F_p)$ the number of finite bars in this barcode, counting with multiplicities. In the non-degenerate case, the number $N(\phi)$ of generators of the Floer complex, which is the number of contractible fixed points of $\phi$ in the Hamiltonian case, and the number $K(\phi,\F_p)$ satisfy the relation \[ N(\phi) = 2 K(\phi,\F_p) + B(\phi, \F_p).\] In the Hamiltonian case this reads \[ N(\phi) = 2 K(\phi,\F_p) + B(\F_p).\] Furthermore, we have the identities for non-spectral $t, a, b \in \R,$

 \[ \dim HF^*(\phi)^{(-\infty, t)} = \displaystyle  \sum_{t \in I_j} m_j,\]
  \[ \dim HF^*(\phi)^{(t,\infty)} = \displaystyle  \sum_{t \in I_j,\; \mrm{finite}} m_j + \sum_{t \notin I_j,\; \mrm{infinite}} m_j,\]
\[ \dim HF^*(\phi)^{(a,b)} =\displaystyle  \sum_{b \in I_j, a \notin I_j} m_j + \sum_{b \notin I_j, a \in I_j} m_j.\] 

In the Hamiltonian case, we normalize actions in such a way that the barcode of $\id$ consists of $B(\F_p)$ infinite bars starting at $0.$ In the closed symplectically aspherical case this is ensured by requiring that the Hamiltonian $H \in \cH$ generating $\phi$ have zero mean over $M$ for all $t \in [0,1].$

%
%
%


	
Theorem \ref{thm: Smith exact filtered} has the following two applications. First, it gives a new proof of a celebrated ``no-torsion" theorem of Polterovich for the Hamiltonian group of closed symplectically aspherical manifolds.

\medskip

\begin{thm}[Polterovich \cite{Polterovich-groups}]\label{thm: Polterovich}
Let $\phi \in \Ham(M,\om)$ be a Hamiltonian diffeomorphism of a symplectically aspherical symplectic manifold, such that $\phi^k = 1$ for some $k \in \Z_{>1}.$ Then $\phi = \id.$
\end{thm}

\begin{proof}
Our new proof proceeds as follows. Let $p$ be a prime dividing $k.$ Then $\phi_1 = \phi^{k/p}$ satisfies $\phi_1^p = \id.$ By an easy recursive argument it is therefore enough to prove the theorem for all $k=p$ prime. Fix a prime $p$ and suppose $\phi \neq \id,$ $\phi^p = \id.$ Fix $\F_p$ as the coefficient field for Floer cohomology. Let $c_+$ and $c_-$ be the maximal and minimal starting point of an infinite bar. By \cite[Theorem 1.3]{SchwarzAspherical} we have $\phi \neq \id$ if and only if $c_+(\phi) > c_-(\phi).$ Therefore there exists an interval $I = (a,b)$ with closure contained in $\R \setminus \{0\},$ such that $a,b$ are not in the spectrum of $\phi,$ $pa,pb$ are not in the spectrum of $\phi^p,$ and \[\dim_{\F_p} HF^*(\phi)^{I} >0.\] Then by Theorem \ref{thm: Smith exact filtered}, we obtain $\dim_{\F_p} HF^*(\phi^p)^{p\cdot I} \geq 1.$ However, since $\phi^p =\id,$ and $0 \notin p\cdot I,$ \[\dim_{\F_p} HF^*(\phi^p)^{p\cdot I} = 0.\] This is a contradiction that finishes the proof. 
\end{proof}

The spectral norm is given in our setting \cite{SchwarzAspherical} by $\gamma(\phi) = c_{+}(\phi) - c_{-}(\phi).$ In the wake of arguments in \cite{S-HZ}, we prove the second application of Theorem \ref{thm: Smith exact filtered} which yields information of the growth rate of the number of fixed points of $\phi^{p^k}$ in terms of the spectral distance $\gamma(\phi^{p^k})$ of $\phi^{p^k}$ to the identity diffeomorphism. 

\medskip

\begin{thm}\label{thm: growth rate}
Let $\phi \in \Ham(M,\om)$ be a Hamiltonian diffeomorphism of a closed symplectically aspherical symplectic manifold, such that $\phi^{p^k}$ for all $k \geq 0$ is non-degenerate. Then setting $N(\phi^{p^k})$ for the number of contractible fixed points of $\phi^{p^k}$ we have \[ \liminf_{k\to \infty} N(\phi^{p^k}) \cdot \gamma(\phi^{p^k}) / p^k > 0. \] In particular if $\liminf_{k \to \infty} \gamma(\phi^{p^k}) = 0,$ then $N(\phi^{p^k})$ grows super-linearly in $p^k.$
\end{thm}


By \cite{SchwarzAspherical} $\gamma$ is continuous in the Hofer metric, while by \cite{BHS} $\gamma$ is continuous in the $C^0$ topology on the Hamiltonian group. Hence we obtain that if $d_{C^0}(\phi^{p^k},\id) \to 0$ or $d_{\mrm{Hofer}}(\phi^{p^k},\id) \to 0$ as $k \to \infty$ then $N(\phi^{p^k})$ grows super-linearly in $p^k.$

Finally, it is conjectured that $\liminf_{l \to \infty} \gamma(\phi^l) > 0,$ and Theorem \ref{thm: growth rate} is to the best of our knowledge a new result in this direction. We refer to \cite{GG-Arnold} for further discussion of this question.

\bs

\begin{rmk}
We expect that using further arguments related to persistence modules, and the Conley conjecture \cite{Ginzburg-CC}, one may extend Theorem \ref{thm: growth rate} to arbitrary Hamiltonian diffeomorphisms $\phi,$ by replacing, for a possibly degenerate $\psi \in \Ham(M,\om)$ the number $N(\psi)$ by the number $N(\psi,\F_p)$ of the endpoints of bars in the barcode of $\psi.$ Of course for $\psi$ non-degenerate $N(\psi) = N(\psi,\F_p).$ These questions shall be investigated in further work \cite{PS-L1}.
\end{rmk}

\begin{proof}
Letting $\beta_{\mrm{tot}}(\phi),$ for a non-degenerate Hamiltonian diffeomorphism $\phi,$ be the sum of the lenghts of the finite bars in $\cl{B}(\phi),$ with $\F_p$ coefficients, and letting $\beta(\phi)$ be the maximal length of a finite bar, we shall prove that \begin{equation}\label{eq: scale} \beta_{\mrm{tot}}(\phi^p) \geq p \cdot \beta_{\mrm{tot}}(\phi).\end{equation} This implies that \[\beta_{\mrm{tot}}(\phi^{p^k}) \geq p^{k-k_0} \cdot \beta_{\mrm{tot}}(\phi^{p^{k_0}})\] for all $k \geq k_0.$ Since the number $K(\phi^p) = \frac{1}{2} (N(\phi^p) - B(\F_p))$ of finite bars in the barcode is finite, we obtain \[ K(\phi^{p^k}) \beta(\phi^{p^k}) \geq p^{k-k_0} \beta_{\mrm{tot}}(\phi^{p^{k_0}}).\] However, by \cite[Theorem C]{KS-bounds} (see also \cite[Proposition 9]{S-HZ}) $\gamma(\phi^{p^k}) \geq \beta(\phi^{p^k}) ,$ hence we get that \[ (N(\phi^{p^k}) - B(\F_p)) \cdot \gamma(\phi^{p^k}) \geq 2 p^{k-k_0} \beta_{\mrm{tot}}(\phi^{p^{k_0}}).\] Now, by the argument of \cite[Theorem A]{SZ-92}, $N(\phi^{p^k})$ is unbounded. As it is clearly non-decreasing, this means that there exists $k_0$ such that for all $k\geq k_0,$ $N(\phi^{p^k}) >  B(\F_p)$ whence $\beta_{\mrm{tot}}(\phi^{p^k}) > 0.$ Then for all $k \geq k_0,$ \[ N(\phi^{p^k})  \gamma(\phi^{p^k}) \geq p^{k} c,\] where $c = 2 p^{-k_0} \beta_{\mrm{tot}}(\phi^{p^{k_0}}) > 0.$ 

It remains to prove \eqref{eq: scale}. We claim that for each generic point $t \in \R,$ the number $m(t,\cl{B(\phi)})$ of finite bars of $\cl{B}(\phi)$ containing $t$ is at most the number $m(p\cdot t, \cl{B}(\phi^p))$ of finite bars of $\cl{B}(\phi^p)$ containing $p \cdot t.$ Indeed, by applying Theorem \ref{thm: Smith exact filtered} to the windows $(-\infty, t)$ and $(t,\infty),$ and taking the sum, we get that \[ 2\cdot m(t,\cl{B(\phi)}) + B(\F_p) \leq 2\cdot m(p\cdot t, \cl{B}(\phi^p)) + B(\F_p),\] or as required, \[ m(t,\cl{B(\phi)})  \leq m(p\cdot t, \cl{B}(\phi^p)).\] Now, \eqref{eq: scale} follows by taking integrals with respect to $t.$
\end{proof}

\appendix
\section{Orientations and signs}\label{app:signs and or}

As we need to work with $\F_p$ coefficients when defining the relevant Floer cochain groups and various operations, in this appendix we discuss how to choose consistent orientations on our moduli spaces in order to obtain well-defined signed counts in Sections \ref{sec:Floer_coho}, \ref{sec:equiv-Floer-coho}, \ref{sec:prod-coprod}. The material here is standard, but it was to the best of our knowledge not applied in the situation that we work in. We refer for example to \cite{Seidel-book,abu_viterbo,Zapolsky} for a detailed discussion of this approach in general and to \cite{Zhao} for a discussion in the case of $S^1$-equivariant cohomology which is similar to our setting.

\indent Let $\mathring{S}=S-(Z^{+}\cup Z^-)$ be a punctured Riemann surface, where $Z^{\pm}= \bigcup_{i=1}^{p^{\pm}}\{ z_i^{\pm}\}$ are finite subsets of a closed Riemann surface $S$. One can equip it with cylindrical ends \begin{equation}\epsilon_i^- \colon (- \infty, -1]\times S^1 \rightarrow \mathring{S} 
\text{ and } \epsilon_j^+ \colon [1, + \infty)\times S^1 \rightarrow \mathring{S}   \text{ for } i=1,\cdots, p^-, \ \ j=1,\cdots, p^+. \nonumber
\end{equation}

We consider the solutions $u \colon \mathring{S} \rightarrow \R \times M_{\phi}$ to the perturbed Cauchy-Riemann equation of the form
\begin{equation}\label{eq: perturbed CR}
(du-Y_{w(s)}) \circ j =J_{s,t,w(s)}\circ (du-Y_{w(s)}), 
\end{equation}

where $w(s) \in \mathcal{P}^{i,\m}_{\alpha}$ is a parametrized Morse flow line of the $\zp$-equivariant Morse function on $S^{\infty}$ and for $z \in S^{\infty},$ $Y_z \in \Omega^1( \mathring{S}, T(\R \times M_{\phi})).$ Furthermore, the solution $u$ satisfies the asymptotic conditions 
\begin{equation} \label{eqn:pos_neg_constraints}
\displaystyle\lim_{s \rightarrow \pm \infty} u(\epsilon_{i}^{\pm}(s,t)) =\gamma_i^{\pm}(t), \ \ i=1,\cdots, p^{\pm} 
\end{equation}
for some sets of Reeb orbits $\Gamma^{\pm}=\bigcup_{i=1}^{p^{\pm}}\{\gamma_i^{\pm}\}$ respectively at the positive and the negative ends of this symplectization $\R \times M_{\phi}.$ We denote by $\mathscr{M}(\Gamma^+, \Gamma^-)$ the moduli space of solutions to the equation \eqref{eq: perturbed CR} satisfying \eqref{eqn:pos_neg_constraints} modulo automorphisms. We will first explain how to orient this moduli space as follows.
\indent Given any path $\Psi(t)$ for $t \in [0,1]$ in $\mathrm{Sp}(2n)$ such that $\Psi(0)=\id$ and $\det(\id-\Psi(1)) \neq 0$, one can reparametrize $\Psi(t)$ so that it is constant for $t$ near $\{0,1\}$ and associate it with a {\em loop} of symmetric matrices $\sS(t)$ that satisfies
\begin{equation}\label{eqn:symmetric_matrix}
\dot{\Psi}(t)=J_0 \sS(t)\cdot \Psi(t).
\nonumber \end{equation}
Such a loop $\sS(t)$ will be called nondegenerate if $\det(\id-\Psi(1)) \neq 0$.
For nondegenerate $\sS(t)$, we denote by $\mathscr{O}_{+}(\sS)$ and $\mathscr{O}_{-}(\sS)$ the spaces of all operators of the form
\begin{eqnarray}
&& D_{\Psi}\colon W^{1,p}(\C, \R^{2n}) \rightarrow L^p(\C, \R^{2n}),\ \ p>2 \label{eqn:opd}  \\
&& D_{\Psi}(X)=\partial_s X + J_0\partial_t X + \wh{\sS} \cdot X, \nonumber
\end{eqnarray}
where $\wh{\sS} \in C^{0}(\C, \mathfrak{gl}(2n))$ is required to satisfy
\begin{eqnarray}
&& \wh{\sS}(e^{s+2\pi it})=\sS(t) \text{ for } s \gg 0, \text{ if } D_{\Psi} \in \mathscr{O}_+(\sS), \nonumber \\
&& \wh{\sS}(e^{-s-2\pi it})=\sS(t) \text{ for } s \ll 0, \text{ if }D_{\Psi} \in \mathscr{O}_-(\sS). \nonumber
\end{eqnarray}
Similarly, for collections of loops of nondegenerate symmetric matrices 
$$\Gamma^-=\{ \sS_i^-(t) \}_{1 \leq i \leq p^{-}} \text{ and } \Gamma^+= \{ \sS_j^+(t) \}_{1 \leq j \leq p^{+}}$$
corresponding to paths $\Psi_i^-(t)$ and $\Psi_j^+(t)$ in $\mathrm{Sp}(2n)$ for $i=1,\cdots, p^-$ and $j=1, \cdots, p^+$ respectively, one defines $\mathscr{O}(\Gamma^-, \Gamma^+)$ as the space of all operators 
\begin{eqnarray}
&& D\colon W^{1,p}(\mathring{S}, \R^{2n}) \rightarrow L^p(\mathring{S}, \R^{2n}), \ \ p>2  \nonumber 
\end{eqnarray}
such that there exists $\wh{\sS} \in C^{0}(\mathring{S}, \mathfrak{gl}(2n))$ satisfying the following conditions on all the cylindrical ends
\begin{equation}
\wh{\sS}(\epsilon_i^-(s,t))=\sS_i^-(t) \text{ for } s \ll 0 \text{ and } \wh{\sS}(\epsilon_j^+(s,t))=\sS_j^+(t) \text{ for } s\gg0, \nonumber
\end{equation}
and the operator $D$ coincides with $\partial_s X + J_0\partial_t X + \sS_i^{\pm} \cdot X$ on the positive and negative ends respectively.

One can verify that $\mathscr{O}_{+}(\sS)$, $\mathscr{O}_{-}(\sS)$ and $\mathscr{O}(\Gamma^-, \Gamma^+)$ consist of Fredholm operators. There are determinant line bundles $\Det(\mathscr{O}_{\pm}(\sS))$, $\Det(\mathscr{O}(\Gamma^-, \Gamma^+)$ defined over the spaces $\mathscr{O}_{\pm}(\sS)$ and $\mathscr{O}(\Gamma^-,\Gamma^+),$ whose fibers over an element $D$ in $\mathscr{O}_{\pm}(\sS)$ or $\mathscr{O}(\Gamma^-, \Gamma^+)$ are given by the determinant line bundle $\det(D)$ of the Fredholm operator $D$ (see \cite{Zinger}). If we fix nondegenerate asymptotic data $\sS(t)$ and $\Gamma^{\pm}$, the spaces $\mathscr{O}_{\pm}(\sS)$ and $\mathscr{O}(\Gamma^-,\Gamma^+)$ are contractible. This implies that line bundles $\Det(\mathscr{O}_{\pm}(\sS))$ and $\Det(\mathscr{O}(\Gamma^-, \Gamma^+)$ are trivial for fixed asymptotic data.\\
\indent For given nondegenerate asymptotic data $\Gamma^-$ and $\Gamma^+$, we consider the following Fredholm operators 
\begin{align*}
\text{ \ \ \  } K \in  \mathscr{O}(\Gamma^-, \Gamma^+) &\text{ and } L_j^+ \in \mathscr{O}_-(\sS_j^+) \text{ for }j=1,\ldots, p^+ \text{ or } \nonumber \\
 L_i^-  \in  \mathscr{O}_+(\sS_i^-) &\text{ and } K\in \mathscr{O}(\Gamma^-, \Gamma^+) \text{ for } i=1,\ldots, p^- \text{ or }\nonumber   \\
  K_1 \in \mathscr{O}(\Gamma^-_1, \Gamma^+_1)  &\text{ and } K_2 \in  \mathscr{O}(\Gamma^-_2, \Gamma^+_2) \text{ and }  p_1^+=p^-_2. \nonumber
\end{align*}
There is a linear gluing operation, denoted as $K \#_{\rho} L$, for Fredholm operators $K$ and $L$ defined on the glued Riemann surfaces $\bigcup_{i=1}^{p^-}\C \#_{\rho} \mathring{S},$ $\mathring{S} \#_{\rho} \bigcup_{j=1}^{p^+}\C$ and $ \mathring{S}_1 \#_{\rho}  \mathring{S}_2$ under the identifications 
\begin{eqnarray}
&& \mathring{S} \supset \epsilon_i^-([\rho,2\rho] \times S^1)   \rightarrow  \{ z \mathbin{|} e^{\rho} \leq |z| \leq e^{2\rho} \} \subset \C \colon \nonumber \\
&&  \ \ \ \ \ \ \ \ \ \ \ \ \ \ \ \ \ \ \ \  (s,t) \mapsto  e^{3\rho-s-2\pi it}, \nonumber
\end{eqnarray}
\begin{eqnarray}
&& \mathring{S} \supset \epsilon_j^+([\rho,2\rho] \times S^1) \rightarrow  \{ z \mathbin{|} e^{\rho} \leq |z| \leq e^{2\rho} \} \subset \C \colon \nonumber \\
&& \ \ \ \ \ \ \ \ \ \ \ \ \ \ \ \ \ \ \ \ (s,t) \mapsto  e^{s+2\pi it}.\nonumber
\end{eqnarray}
For $\rho\gg0$, we then obtain the glued operators $K \#_{\rho}\{L_j^+\}_{j=1}^{p^+},$ $\{L_i^-\}_{i=1}^{p^-}\#_{\rho} K$ and $K_1 \#_{\rho} K_2$ in $\mathscr{O}_-(\sS_-),$ $\mathscr{O}_+(\sS_+)$ and $\mathscr{O}(\Gamma_1^-, \Gamma_2^+)$ in each case.
With respect to this gluing operation, it is shown in \cite[Proposition 9]{FH_orient} that there are canonical isomorphisms
\begin{eqnarray} \label{eqn:gluing}
&& \det(K \#_{\rho}\{L_j^+\}_{j=1}^{p^+})\cong \det(K) \otimes \det(L_1^+) \otimes \cdots \otimes \det(L_{p^+}^+) \\
&& \det(\{L_i^-\}_{i=1}^{p^-}\#_{\rho} K)\cong  \det(L_1^-) \otimes \cdots \otimes \det(L_{p^-}^-)\otimes \det(K)  \nonumber \\
&& \det(K_1\#_{\rho} K_2) \cong \det(K_1) \otimes \det(K_2), \nonumber
\end{eqnarray}
up to multiplication by a positive real number. 

We now proceed to various operations. For any $J_t \in \mathscr{J}( M_{\phi})$, the complex vector bundle $(\gamma^*T(\R\times M_{\phi}), J_t)$ over $S^1$ can be trivialized. We choose a trivialization along the Reeb orbit $\gamma(t)$ given by $\xi(t)\colon \R^{2n} \rightarrow T_{\gamma(t)}(\R \times M_{\phi})$ and obtain a path $\Psi_x(t)$ in $\mathrm{Sp}(2n)$ as the composition of
\begin{equation}
\R^{2n} \xrightarrow{\xi(0)} T_{\gamma(0)}\widehat{M} \xrightarrow{d\psi^t_H} T_{\gamma(t)}\widehat{M} \xrightarrow{\xi^{-1}(t)} \R^{2n}.
\nonumber \end{equation} 
Given Reeb orbits $\gamma_i^-(t)$ and $\gamma_j^+(t)$ in the positive and the negative ends of $\R \times M_{\phi}$, we denote by $\sS_i^-(t)$ and $\sS_j^+(t)$ the loops of symmetric matrices which generate $\Psi_{i^-}(t)$ and $\Psi_{j^+}(t)$, respectively. The construction in \eqref{eqn:opd} yields operators $D_{\Psi_{\gamma_i^-}}$ and $D_{\Psi_{\gamma_j^+}}$ in $\mathscr{O}_-(\sS_i^-)$ and $\mathscr{O}_+(\sS_j^+)$ respectively. For a Reeb orbit $\gamma \in \Gamma^+ \cup \Gamma^{-}$ we introduce the notation 
 $$o_{\gamma}:=|\det(D_{\Psi_{\gamma}})|$$ 
 where $|\cdot|$ denotes the graded abelian group generated by the two orientations of $\det(D_{\Psi_{\gamma_i}})$ and modulo the relation that the sum vanishes. In the case that there is a free $\R$ action on the moduli space $\widetilde{\mathscr{M}}(\Gamma^-, \Gamma^+)$ of solutions to \eqref{eq: perturbed CR}, we take $u \in \widetilde{\mathscr{M}}(\Gamma^-, \Gamma^+)$ and define $o_u:=|\det(D_u)|$, where  $D_u \in \mathscr{O}(\Gamma^-, \Gamma^+)$ is the linearization of the Floer equation at $u$ with respect to a trivialization of $u^*(T\widehat{M}) \rightarrow \R \times S^1$ that agrees with the trivializations of $\gamma_i^*(T\widehat{M}) \rightarrow S^1$ as $s \rightarrow \pm \infty$.
For different choices of trivializations, Lemma 13 in \cite{FH_orient} and Proposition 1.4.10 in \cite{abu_viterbo} show that the corresponding determinant line bundles $\det(D_{\Psi_{\gamma_i}})$ and $\det(D_u)$ are isomorphic. This implies that $o_{\gamma_i^-}$, $o_{\gamma_j^+}$and $o_u$ are well-defined for $i=1,\cdots,p^-$ and $j=1,\cdots, p^+$.
By the gluing property \eqref{eqn:gluing}, we have a canonical isomorphism
\begin{equation} 
o_u \otimes o_{\gamma_1^-}\otimes \cdots \otimes o_{\gamma_{p^-}^-}\cong o_{\gamma_1^+}\otimes \cdots \otimes o_{\gamma_{p^+}^+}.
\nonumber \end{equation}

Together with the fact that $o_u \cong |\R \partial_s| \otimes |\mathscr{M}(\Gamma^-, \Gamma^+)|$, where $\mathscr{M}(\Gamma^-, \Gamma^+) = \til{\mathscr{M}}(\Gamma^-, \Gamma^+)/\R,$ we obtain an isomorphism
\begin{equation} \label{eqn:oriet2}
o_{\gamma_1^+}\otimes \cdots \otimes o_{\gamma_{p^+}^+}\cong |\R \partial_s| \otimes |\mathscr{M}(\Gamma^-, \Gamma^+)| \otimes o_{\gamma_1^-}\otimes \cdots \otimes o_{\gamma_{p^-}^-} ,
\end{equation}
where $\R \partial_s$ is the 1-dimensional subspace of $\ker(D_u)$ spanned by translation in positive $s$-direction.  (If there is no translation automorphism for $u$ in the case of the continuation maps, then we set $o_u \cong |\mathscr{M}(\Gamma^-, \Gamma^+)|$.) For $\sum_i|\gamma_i^-|=\sum_j|\gamma_j^+|+1$, we have that $T\mathscr{M}(\Gamma^-, \Gamma^+)$ is canonically trivial as $\mathscr{M}(\Gamma^-, \Gamma^+)$ is a 0-dimensional manifold. By comparing the fixed orientations on both sides of \eqref{eqn:oriet2}, we obtain an isomorphism
\begin{equation}\label{eqn:d_u}
\epsilon_u\colon o_{\gamma_1^+}\otimes \cdots \otimes o_{\gamma_{p^+}^+} \rightarrow o_{\gamma_1^-}\otimes \cdots \otimes o_{\gamma_{p^-}^-} .
\end{equation}

%

Now the linearization of  the equation \eqref{eq: perturbed CR} yields a linear map \begin{equation}
D\colon \R \oplus T_w {\sQ}^{i,m}_{\alpha} \oplus W^{1,p}(\R \times S^1, u^*(T\widehat{M})) \rightarrow L^p(\R \times S^1, u^*(T\widehat{M})),\nonumber
\end{equation}
where $p>2,$ and ${\sQ}^{i,m}_{\alpha} = W^u(Z^m_i) \cap W^s(Z^0_{\alpha}) /\R,$ where $W^u(Z_i^m)$ is the unstable manifold of the critical point $Z^m_i$ of index $i$ on $S^{\infty},$ and similarly $W^s(Z_{\alpha}^0)$ is the stable manifold of $Z^0_{\alpha}$ inside $S^N$ for $N$ large enough. The orientations of the spaces $W^u(Z_i^m),$ $W^s(Z_{\alpha}^0)$ can be chosen compatibly with the inclusions $S^N \to S^{N'},$ $N < N'.$ The Floer data $(J_{w,s,t},H_{w,s,t})$ is regular if the linear map $D$ is surjective. This is equivalent to surjectivity of the linear map $T_w {\sQ}^{i,m}_{\alpha} \rightarrow \coker(D_u),$ where $D_u$ is the linearized operator associated to an element $(u,w)$ in $\mathscr{M}_i(\Gamma^-,\Gamma^+)$ with $w$ fixed, and the map is given by the restriction of $D$ to the $T_w {\sQ}^{i,m}_{\alpha}$ summand. In this case there is an isomorphism of determinant lines
\begin{equation} \label{eqn:orh1}
\det(T_u\til{\mathscr{M}}_i(\Gamma^-, \Gamma^+))\cong \det(D_u) \otimes \det(T_w {\sQ}^{i,m}_{\alpha}),
\end{equation}
which can either be seen to come either from the theory of Fredholm triples \cite{Zinger,Zapolsky}, or from the exact sequence:
\begin{equation}
0 \rightarrow T_u\widetilde{\mathscr{M}}_i(\Gamma^-, \Gamma^+) \rightarrow  T_w {\sQ}^{i,m}_{\alpha} \oplus \ker(D_u) \rightarrow \coker(D_u) \rightarrow 0,
\end{equation}
where the first map is given by $p \oplus q,$ where $p$ is the projection to $T_w {\sQ}^{i,m}_{\alpha},$ and $q$ is the projection to $\ker(D_u)$ inside $T_u\widetilde{\mathscr{M}}_i(\Gamma^-, \Gamma^+) = \ker(D);$ the second map is the restriction of $D$ post-composed with projection to $\coker(D_u).$

By trivializing $\det(\R)$ by $\del_{s}$ for $s$ the natural coordinate on $\R,$ \eqref{eqn:orh1} induces the isomorphism
\begin{equation} \label{eqn:orh1notil}
\det(T_u {\mathscr{M}}_i(\Gamma^-, \Gamma^+))\cong \det(D_u) \otimes \det(T_w {\sQ}^{i,m}_{\alpha}).
\end{equation}

Now \eqref{eqn:orh1}, \eqref{eqn:orh1notil} combined with an isomorphism coming from gluing theory yields an isomorphism
\begin{equation} \label{eqn:orh2}
|\det(T_u\mathscr{M}_i(\Gamma^-, \Gamma^+))| \otimes o_{\gamma_1^+}\otimes \cdots \otimes o_{\gamma_{p^+}^+}  \cong |\det(T_w {\sQ}^{i,m}_{\alpha})| \otimes o_{\gamma_1^-}\otimes \cdots \otimes o_{\gamma_{p^-}^-}.
\end{equation}
We  choose a coherent orientation for each unstable manifold $W^u(Z^m_i)$ of a critical point $Z^m_i$ on $S^{\infty},$ and on the unstable manifolds of $Z^0_{\alpha}$ in $S^N$ for various $N,$ compatibly with the inclusions. This induces a coherent system of orientations on the spaces $\sP^{i,m}_{\alpha}$ and ${\sQ}^{i,m}_{\alpha}.$ For $|\gamma_0|=|\gamma_1|-i-1+\alpha$, the moduli space $\mathscr{M}^{i,\m}_{\alpha}(\Gamma^-, \Gamma^+)$ is zero-dimensional and $T_u\mathscr{M}^{i,\m}_{\alpha}(\Gamma^-, \Gamma^+)$ is canonically trivial. Now by \eqref{eqn:orh2}, one obtains an isomorphism
\begin{equation}
\epsilon^{i,\m}_{\alpha}(u)\colon o_{\gamma_1^+}\otimes \cdots \otimes o_{\gamma_{p^+}^+} \rightarrow o_{\gamma_1^-}\otimes \cdots \otimes o_{\gamma_{p^-}^-} 
\nonumber \end{equation}
for each $u$ in $\mathscr{M}^{i,\m}_{\alpha}(\Gamma^-, \Gamma^+)$. When $p^+=p$ and $p^-=1$, then this gives the signs in the definition of the $\zp$-equivariant product map, and similarly when $p^+=1$ and $p^-=p,$ this defines the signed count for the coproduct map. For example, the operations $\mathcal{P}^{i,\m}_{\alpha}$ that define the $\zp$-equivariant product map can be then rewritten as
\begin{eqnarray}
&& \sP^{i,\m}_{\alpha} \colon CF^*(\phi)^{\otimes p} \rightarrow CF^{*-i+\alpha}(\phi^p) \\
&&\sP^{i,\m}_{\alpha}|_{o_{\gamma^+_0}\otimes \cdots \otimes o_{\gamma^+_{p-1}}}=\bigoplus_{|\gamma^-|=\sum|\gamma_i^+|-i+\alpha}\sum_{u \in \mathscr{M}_{\alpha}^{i,\m}(\Gamma^-, \Gamma^+)} \epsilon^{i,\m}_{\alpha}(u).
\end{eqnarray}

As the construction of the isomorphisms was canonical, it is compatible with gluing which can be used to show that differentials square to zero, and that various Leibnitz identities hold. We refer to \cite{Seidel-book,abu_viterbo,Zapolsky,Zhao} for further discussion of canonical orientations and their compatibility properties.

\end{document}